\documentclass[11pt]{amsart}
\usepackage{tikz}

\usepackage{amsmath}
\usepackage{amssymb}
\usepackage{mathrsfs}
\usepackage[all]{xy}

\newtheorem{theorem}{Theorem}[section]
\newtheorem{claim}[theorem]{Claim}

\newtheorem{lemma}[theorem]{Lemma}

\newtheorem{proposition}[theorem]{Proposition}

\theoremstyle{definition}
\newtheorem{definition}[theorem]{Definition}

\theoremstyle{remark}
\newtheorem{remark}[theorem]{Remark}

\def\calP{\mathcal P}

\def\l{{\langle}}
\def\r{{\rangle}}

\newcount\skewfactor
\def\mathunderaccent#1#2 {\let\theaccent#1\skewfactor#2
\mathpalette\putaccentunder}
\def\putaccentunder#1#2{\oalign{$#1#2$\crcr\hidewidth
\vbox to.2ex{\hbox{$#1\skew\skewfactor\theaccent{}$}\vss}\hidewidth}}

\newcommand{\lusim}[1]{\smash{\underset{\raisebox{1.2pt}[0cm][0cm]{$\sim$}}
{{#1}}}}
\def\smallbox#1{\leavevmode\thinspace\hbox{\vrule\vtop{\vbox
   {\hrule\kern1pt\hbox{\vphantom{\tt/}\thinspace{\tt#1}\thinspace}}
   \kern1pt\hrule}\vrule}\thinspace}

\DeclareMathOperator{\cof}{cof}
\DeclareMathOperator{\dom}{dom}
\DeclareMathOperator{\supp}{Supp}
\DeclareMathOperator{\Succ}{Succ}

\DeclareMathOperator{\Add}{Cohen}
\DeclareMathOperator{\Pri}{Prikry}
\DeclareMathOperator{\Lot}{LOTT}
\DeclareMathOperator{\rng}{rng}
\newcommand{\cf}{{\rm cf}}
\def\Ult{{\rm Ult}}

\title{On Cohen and Prikry forcing notions}
\date{\today}

\author{Tom Benhamou}

\address{School of Mathematical Sciences, Raymond and Beverly Sackler Faculty of Exact Science, Tel-Aviv University, Ramat Aviv 69978, Israel}
\email{tombenhamou@tauex.tau.ac.il}

\author{Moti Gitik}
\thanks{ The work of the second author was partially supported by ISF grants No.1216/18, 882/22.}
\address{School of Mathematical Sciences, Raymond and Beverly Sackler Faculty of Exact Science, Tel-Aviv University, Ramat Aviv 69978, Israel}
\email{gitik@tauex.tau.ac.il}


\begin{document}
\let\labeloriginal\label
\let\reforiginal\ref
\def\ref#1{\reforiginal{#1}}
\def\label#1{\labeloriginal{#1}}

\begin{abstract}
 \begin{enumerate}
  \item We show that it is possible to add $\kappa^+-$Cohen subsets to $\kappa$ with a Prikry forcing over $\kappa$.
This answers a question from \cite{YairTomMoti}.

  \item A strengthening of non-Galvin property is introduced. It is shown to be consistent using a single measurable cardinal which improves
  a previous result by S. Garti, S. Shelah, and the first author \cite{BenGarShe}.  
  \item A situation with Extender-based Prikry forcings is examined. This relates to a  question of H. Woodin.
\end{enumerate}
\end{abstract}
\maketitle
\section{Introduction}
\subsection{Intermediate models of the tree-Prikry forcing} In many mathematical theories, such as groups, vector spaces, topological spaces, graphs etc., the study of submodels of a given model is indispensable to the understanding of the model and in some sense measures its complexity.
In forcing theory,  subforcings  of a given forcing generate intermediate models to a generic extension by the forcing. Hence, the study of intermediate models is somehow parallel to the one regarding subforcings.
There are numerous classification results in this spirit, for example, some forcing such as the Sacks forcing \cite{Sacks} and variants of the tree-Prikry forcing \cite{MinimalPrikry} do not have proper intermediate models. Other forcings such as the Cohen forcing \cite{kanamori1994}, Random forcing \cite{Maharam}, Prikry forcing \cite{PrikryCaseGitikKanKoe}, and Magidor forcing \cite{TomMoti}, \cite{partOne} have intermediate models  of the same type. A tree Prikry forcing or its particular case, which will be central for us in this paper, the Prikry forcing with a non-normal ultrafilter can behave differently.  
For example, 
under suitable large cardinal assumptions, every $\kappa$-distributive forcing of cardinality $\kappa$ is a projection of this forcing. Actually, more is true, under the assumption that $\kappa$ is $\kappa$-compact there is a single Prikry-type forcing which absorbs all the $\kappa$-distributive forcings of cardinality $\kappa$ (see  \cite{GitikOnCompactCardinals}). In the absence of very large cardinals the situation changes, indeed, \cite{YairTomMoti} Hayut and the authors proved that if a certain $<\kappa$-strategically closed forcing of cardinality $\kappa$ is a projection of the tree-Prikry forcing then it is consistent that there is a cardinal $\lambda$ with high Mitchell order, namely $o(\lambda)>\lambda^+$. In \cite{TomMoti}, the authors proved that starting from a measurable cardinal (which is the minimal large cardinal assumption in the context of Prikry forcing) it is consistent that there is a (non-normal) ultrafilter $U$, such that the Prikry forcing with $U$ projects onto the Cohen forcing $\Add(\kappa,1)$, this was improved later in \cite{YairTomMoti} to a larger class of forcing notions called \textit{Masterable forcings}. 
In the context of Prikry-type forcings, the existence of such embeddings and projections allows one to iterate distributive forcing notions on different cardinals, see \cite[Section 6.4]{Gitik2010}.

It remained open whether it is possible to get more Cohen subsets of $\kappa$ after forcing with the Prikry forcing with a $\kappa$-complete ultrafilter $U$ over $\kappa$. This was asked explicitly in \cite{YairTomMoti}.

The basic difficulty is that the size of $\Add(\kappa,\kappa^+)$ is $\kappa^+$ and it is not hard to see (Proposition \ref{Galvin inplies no Cohens}) that this cannot happen, if $U$ has the Galvin property.

We formulate a certain strengthening of the negation of the Galvin property, show its consistency starting with a measurable cardinal and finally apply it in order to construct an ultrafilter $U$ 
such that the Prikry forcing (For a formal definition of the Prikry forcing with non-normal ultrafilter see Definition \ref{Prikry with non normal}) with it adds a generic subset to $\Add(\kappa,\kappa^+)$.

\subsection{Extender-based Prikry forcing and a question of Woodin}

Magidor and the second author developed the Extender-based Prikry forcing in \cite{Git-Mag} to violate the SCH under mild large cardinal assumptions. Later Merimovich \cite{Mer},\cite{Mer1}, presented a variation of this forcing which will be used in this paper.

H. Woodin asked\footnote{We would like to thank Mohammad Golshani for reminding us of the exact formulation of Woodin's question.} in the early $90$s whether, assuming that there is no inner model with a strong cardinal, it is possible to have a model $M$ in which $2^{\aleph_\omega}\geq \aleph_{\omega+3}$, GCH holds  below $\aleph_\omega$, there is an inner model $N$ such that $\kappa=(\aleph_\omega)^M$ is a measurable and 
$2^\kappa\geq (\aleph_{\omega+3})^M$.

A natural approach to tackle Woodin's question is to use the Extender-based Prikry with interleaved collapses forcing, defined my the second author and M. Magidor in \cite{Git-Mag}. This forcing collapses a measurable cardinal to $\aleph_\omega$ and simultaneously blows up the powerset of that measurable. Hence, if one can show that the a generic extension by the Extender-based Prikry forcing has an intermediate model where $\kappa$ stays measurable and $2^\kappa$ is large, this will provide a positive answer to Woodin's question.  In this paper we show that this approach is doomed. More precisely, we address in general the question whether it is possible to add many subsets of $\kappa$ $ \l x_\alpha \mid  \alpha<\lambda\r, \  \lambda\geq \kappa^{++}$ with the Extender-based Prikry forcing over $ \kappa$ such that $\kappa$ remains a regular cardinal in
 $V[ \l x_\alpha \mid  \alpha<\lambda\r]$. We give a negative answer to this question with respect to the Extender-based Prikry forcing as defined in \cite{Git-Mag} and the Merimovich version of the forcing presented in \cite{Mer1,Mer2}. In particular, as a consequence of our results (Theorems \ref{Woodin-answer},\ref{CarmiVersion theorem}), the Extender-based Prikry forcing cannot be used to answer Woodin's question.

\subsection{The Galvin property} F. Galvin \cite{MR0369081}, in 70th, showed that if $\kappa^{<\kappa}=\kappa$ and $F$ is a normal filter over $\kappa$ then the following combinatorial property holds:
$$\text{ For every }\{X_i\mid i<\kappa^+\}\subseteq F\text{ there is }I\subseteq [\kappa^+]^\kappa\text{ such that } \cap_{i\in I}X_i\in F.$$
We denote this statement by $Gal(F,\kappa,\kappa^+)$.  In particular, this holds for the club filter $Cub_{\kappa}$ as it is a normal filter over a cardinal $\kappa$.

In \cite{AbrahamShelah1994}, Uri Abraham and Saharon Shelah constructed a model where  $ Gal(Cub_{\kappa^+},\kappa^+,\kappa^{++})$ fails for a regular $\kappa$. Shimon Garti  \cite{MR3787522},\cite{MR3604115} and later together with the first author and Alejandro Poveda \cite{NegGalSing} continued the investigation of the Galvin property for the club filter. 
The Galvin property for  $\kappa$-complete ultrafilters over a measurable cardinal $\kappa$ was used recently in \cite{GitDensity} and \cite{Parttwo}.
The question of failure of the Galvin property for such ultrafilters was shown to be independent.
Namely, in \cite{Parttwo} the authors observed that in $L[U]$ every $\kappa$-complete ultrafilter has the Galvin property, and Garti, Shelah and the first author, starting with a supercompact cardinal,
produced a model with a $\kappa$-complete ultrafilter which contains $Cub_\kappa$ and fails to satisfy the Galvin property.

In section $2$, we isolate a property of sequences we call a \textit{strong witness for the failure of Galvin's property} which implies in particular the failure of Galvin's property. This property is used in theorem \ref{universe theorem}, where we start from a single measurable cardinal, and construct a model with an ultrafilter which fails to satisfy the Galvin property. This improves the initial large cardinal assumption of \cite{BenGarShe}. 

Later in theorem \ref{main theorem}, we were able to slightly modify the construction of theorem \ref{universe theorem}, construct an ultrafilter $W$ and a strong witness for the failure of the Galvin property for it, which serves to glue together initial segments of functions and obtain  $\kappa^+$-mutually generic Cohen function on $\kappa$. This idea is generalized to longer sequences (and in turn to more Cohen functions) in Theorems \ref{Non-galvin for longer},\ref{Many-Cohens}.

\vskip 0.3 cm
Our main results are:

\vskip 0.2 cm
\textbf{Theorem  \ref{universe theorem}.}
\textit{ Assume GCH and let $\kappa$ be measurable in $V$. Then there is a cofinality preserving forcing extension $V^*$ in which there is a $\kappa$-complete ultrafilter $W$ over $\kappa$ which concentrates on regulars, extends $Cub_\kappa$, and has a strong witness for the failure of Galvin's property.}

\vskip 0.2 cm
\textbf{Theorem \ref{main theorem}.}
 \textit{
Assume $GCH$ and that $\kappa$ is a measurable cardinal in $V$. Then there is a cofinality preserving forcing extension $V^*$ in which $GCH$ still holds, and there is a $\kappa$-complete ultrafilter $U^*\in V^*$ over $\kappa$ such that forcing with Prikry forcing $Pikry(U^*)$ introduces a $V^*$-generic filter for $Cohen^{V^*}(\kappa,\kappa^+)$.}

\vskip 0.2 cm
\textbf{Theorem \ref{Non-galvin for longer}.} \textit{
Assume GCH and that there is a $(\kappa,\kappa^{++})$-extender over $\kappa$ in $V$. Then there is a cofinality preserving forcing extension $V^*$ such that $V^*\models 2^\kappa=\kappa^{++}$, in $V^*$ there is a $\kappa$-complete ultrafilter $W$ over $\kappa$ which concentrates on regulars, extends $Cub_\kappa$, and has a strong witness of length $\kappa^{++}$ for the failure of Galvin's property. }

\vskip 0.2 cm
\textbf{Theorem  \ref{Many-Cohens}.} \textit{
Assume $GCH$ and that $E$ is a $(\kappa,\kappa^{++})-$ extender in $V$. Then there is a cofinality preserving forcing extension $V^*$ in which $2^{\kappa}=\kappa^{++}$ and a non-Galvin ultrafilter $W\in V^*$ such that forcing with $\Pri(W)$ introduces a $V^*$-generic filter for $Cohen^{V^*}(\kappa,\kappa^{++})$-generic filter.}

\vskip 0.2 cm
\textbf{Theorem 
\ref{Woodin-answer}.}
\textit{
Let $\calP_E$ be the Extender-based Prikry forcing of \cite{Git-Mag},
and $G\subseteq \calP$ be a generic. Suppose that $A\in V[G]\setminus V$ is a subset of $\kappa$.
Then $\kappa$ changes its cofinality to $\omega$ in $V[A]$.}

\vskip 0.2 cm
\textbf{Theorem 
\ref{CarmiVersion theorem}.}
\textit{Assume GCH, let $E$ an extender over $\kappa$ and $\mathbb{P}_E$ be the Merimovich version of the Extender-based Prikry forcing of \cite{Mer,Mer1,Mer2}.
Let $G$ be a generic subset of $\mathbb{P}_E$ and let
$\l {A}_\alpha\mid \alpha<\kappa^{++} \r$ be different subsets of $\kappa$ in $V[G]$.
Then there is $I\subseteq \kappa^{++}, I\in V, |I|=\kappa$ such that $\kappa$ is a singular cardinal of cofinality $\omega$ in
$V[\l {A}_\alpha\mid \alpha\in I\r]$.
In particular, there is no intermediate model of $V[G]$ where $\kappa$ is measurable and $2^\kappa>\kappa^+$.}

\vskip 0.3 cm

This paper is organized as follows:
\begin{itemize}
    \item Section 1: We provide the basic definitions and  background for this paper.
    \item Section 2: We prove Theorems \ref{universe theorem}, \ref{main theorem}.
    \item Section 3: We prove Theorems \ref{Non-galvin for longer}, \ref{Many-Cohens}.
    \item Section 4: We prove Theorems \ref{Woodin-answer},\ref{CarmiVersion theorem}.

\end{itemize}

\section{Basics}
\subsection{The forcing notions}
In our notations $p\leq q$ means that $q$ is stronger than $p$. 
We assume that the reader is familiar with the forcing method and iterated forcing.
Most of our notations are inspired by \cite{CummingsHand},\cite{Gitik2010} where we refer the reader for more information regarding forcing and iterations. Let us present the definitions of the forcing we intend to use:
\begin{definition}
The forcing adding $\lambda$-many Cohen functions to $\kappa$ denoted by $\Add(\kappa,\lambda)$ consists of all partial functions $f:\kappa\times\lambda\rightarrow\{0,1\}$ such that $|f|<\kappa$. The order is defined by $f\leq g$ iff $f\subseteq g$. 
\end{definition}
\begin{definition}\label{Prikry with non normal}
Let $U$ be a $\kappa$-complete non-trivial ultrafilter over $\kappa$ and $\pi:\kappa\rightarrow\kappa$ be the function representing $\kappa$ in the $\Ult(V,U)$. The Prikry forcing with $U$, denoted by $\Pri(U)$ consists of all sequences $\l\alpha_1,...,\alpha_n, A\r$ such that:
\begin{enumerate}
    \item $\l\alpha_1,...,\alpha_n\r$ is an $\pi$-increasing sequence of ordinals below $\kappa$ i.e. for every $1\leq i<n$, $\alpha_i<\pi(\alpha_{i+1})$
    \item $A\in U$, $\pi(\min(A))>\alpha_n$.
\end{enumerate}
The order is defined by $\l\alpha_1,...,\alpha_n,A\r\leq \l\beta_1,...,\beta_m,B\r$ iff:
\begin{enumerate}
    \item $n\leq m$ and for every $i\leq n$, $\alpha_i=\beta_i$.
    \item for every $n<i\leq m$, $\beta_i\in A$.
    \item $B\subseteq A$.
\end{enumerate}
If $n=m$ we say that $q$ directly extends $p$ and denote it by $p\leq^* q$.
\end{definition}
If $U$ is normal then we can take $\pi=id$ and the forcing $\Pri(U)$ is the standard Prikry forcing. The requirement that the sequence is $\pi$-increasing ensures that the forcing $\Pri(U)$ is forcing equivalent to the tree-Prikry forcing defined in \cite{Gitik2010}. Also, it enables to define a diagonal intersection suitable for the non-normal case, namely, for $\{A_i\mid i<\kappa\}\subseteq U$ define 
$$\Delta^*_{i<\kappa}A_i:=\{\alpha<\kappa\mid\forall i<\pi(\alpha).\alpha\in A_i\}.$$
This kind of diagonal intersection instead of the standard one is used to prove the Prikry property of $\Pri(U)$.

Later we will need the easy direction of the Mathias criterion \cite{Mathias} for Prikry-generic sequences, and the proof can be found in \cite[Corollary 4.22]{TomTreePrikry}:
\begin{lemma}\label{easy direction of Mathias}
Let $G\subseteq \Pri(U)$ be a generic filter producing a Prikry sequence $\{c_n\mid n<\omega\}$. Then for every $A\in U$, there is $N<\omega$ such that for every $n\geq N$, $c_n\in A$.
\end{lemma}
For more information regarding the tree-Prikry forcing see \cite{Gitik2010} or \cite{TomTreePrikry}. 
In the following, we define the notion of \textit{lottary sum}. The terminology “lottery sum” is due to Hamkins, although the concept of the
lottery sum of partial orderings has been around for quite some time and has been referred
to, for example, as “disjoint sum of partial orderings”:
\begin{definition}\label{lottery}
Let $\mathbb{P}_0,\mathbb{P}_1$ be two forcing notions. The \textit{lottery sum} of $\mathbb{P}_0$ and $\mathbb{P}_1$ denoted by $\Lot(\mathbb{P}_0,\mathbb{P}_1)$ is the forcing whose underlining set is $\mathbb{P}_0\times\{0\}\cup\mathbb{P}_1\times\{1\}$ and the order is define by $\l p,i\r\leq \l p',j\r$ iff $i=j$ and $p\leq_{\mathbb{P}_i} p'$.
\end{definition}
The forcing $\Lot(\mathbb{P}_0,\mathbb{P}_1)$ generically chooses $\mathbb{P}_0$ or $\mathbb{P}_1$ and adds a $V-$generic filter for it. 
As Hamkins observed in \cite{Hamkins2000-HAMTLP}, iterating such forcing notions leaves a certain amount of freedom when lifting ground model embeddings, this will be exploited in most of our construction.

In section~$4$ we will discuss the Extender-based Prikry forcing which was originally defined by Magidor and the second author in \cite{Git-Mag}. A more recent variation of it is due to Carmi Merimovich \cite{Mer,Mer1,Mer2}. 

Let us present the two versions. Let $E$ be a $(\kappa,\lambda)-$extender and $j=j_E:V\rightarrow M_E\simeq Ult(V,E)$ the natural elementary embedding (See \cite{Jech2003} for the definition of extenders and related constructions) and suppose that $f_\lambda:\kappa\rightarrow\kappa$ is a function such that $j(f_\lambda)(\kappa)=\lambda$ (our result uses $\lambda=\kappa^{++}$ and we can simply take $f_\lambda(\nu)=\nu^{++}$). 
Let us first present the \textit{Merimovich version} of the Extender-based Prikry forcing.

For each set of cardinality $\leq\kappa$, $d\in[\lambda\setminus \kappa]^{\leq\kappa}$ with $\kappa\in d$. Define 
$$E(d)=\{X\in V_\kappa\mid (j\restriction d)^{-1}\in j(X)\}.$$
If $A\in E(d)$ we can assume that for every $\nu,\mu\in A$, $\nu:d\rightarrow\kappa$ is order preserving, $\kappa\in\dom(\nu)$, $|\nu|\leq \nu(\kappa)$, $\nu(\kappa)=\mu(\kappa)\rightarrow \dom(\nu)=\dom(\mu)$.
Merimovich calls such a set a \textit{good set}.

\begin{definition}\label{Carmi Version}
The \textit{conditions of $\mathbb{P}_E$} are pairs $p=\l f^p,A^p\r$ such that \begin{enumerate}
    \item $f^p:d\rightarrow [\kappa]^{<\omega}$ is the ``Cohen Part" of the condition, $d\in[\lambda\setminus\kappa]^{<\omega}$, $\kappa\in d$.
    \item $A^p\in E(d)$ is a good set.
    \item  for every $\nu\in A^p$ and $\alpha\in\dom(\nu)$, $\max(f^p(\alpha))<\nu(\kappa)$.
\end{enumerate}
The \textit{order of $\mathbb{P}_E$} is defined in two steps: a direct extension is defined by $\l f,A\r\leq^* \l g,B\r$ if
\begin{enumerate}
    \item 
$f\subseteq g$.  
\item $B\restriction \dom(f):=\{\nu\restriction \dom(\nu)\cap\dom(f)\mid \nu \in B\}\subseteq A$.
\end{enumerate}
A one-point extension of $p=\l f,A\r$ for $\nu\in A$ is defined by $p^{\smallfrown}\nu=\l g,B\r$ where
\begin{enumerate}
    \item $\dom(g)=\dom(f)$.
    \item for every $\alpha\in\dom(g)$
$$g(\alpha)=\begin{cases} f(\alpha)^{\smallfrown}\nu(\alpha) & \alpha\in\dom(\nu)\\
f(\alpha) & else\end{cases}.$$
\item $B=\{\mu\in A\mid \sup_{\alpha\in\dom(\nu)}(\nu(\alpha)+1)\leq \mu(\kappa)\}$.
\end{enumerate}  
an $n$-point extension $p{}^{\smallfrown}\vec{\nu}$ is defined recursively by consecutive one-point extensions.
A general extension is defined by $p\leq q$ iff for some $\vec{\nu}\in [A^p]^{<\omega}$, $p^{\smallfrown}\vec{\nu}\leq^* q$. 
\end{definition}

As in Merimovich \cite{Mer1}, we will sometime replace the large set $A$ in a condition $\l f,A\r$ with a Tree $T$ which is $E(\dom(f))$-fat.

Let us now present the \textit{original version} defined by M. Magidor and the second author from \cite{Git-Mag}.
Define for every $\kappa\leq\alpha<\lambda$:
$$U_\alpha:=\{X\subseteq\kappa\mid \alpha\in j(X)\}$$
These are $P-$point ultrafilters. 
For every $\alpha\leq \beta<\lambda$ we define that $\alpha\leq_E\beta$ if there is some $f:\kappa\rightarrow\kappa$, $j(f)(\beta)=\alpha$. This implies that $f$ Rudin-Keisler projects $U_\beta$ onto $U_\alpha$. For every such pair $\alpha\leq_E\beta$ fix such a projection $\pi_{\beta,\alpha}$  such that $\pi_{\alpha,\alpha}=id$. The projections to the normal measure $U_\kappa$ have a uniform definition, $\pi_{\alpha,\kappa}(\nu)=\nu^0$ where $\nu^0$ is the maximal inaccessible $\nu^*\leq\nu$ such that $f_\lambda\restriction \nu^*:\nu^*\rightarrow \nu^*$, $f_\lambda(\nu^*)>\nu$ and $\pi_{\alpha,\kappa}(\nu)=0$ if there is no such $\nu^*$. Suppose that the system $\l U_\alpha,\pi_{\alpha,\beta}\mid \alpha\leq\beta<\lambda, \alpha\leq_E\beta\r$ is a \textit{nice system} (See \cite{Git-Mag} or \cite[Discussion after Lemma 3.5]{Gitik2010}). Let us say that $\nu$ is \textit{permitted} for $\nu_0,...\nu_n$ is $\nu^0>\max_{i=0,..,n}\nu^0_i$. 
\begin{definition}\label{Original Version}

 \textit{The condition} of the forcing $\mathcal{P}_E$ are pairs $p=\l f,T\r$ such that:
 \begin{enumerate}
     \item $f:\lambda\setminus\kappa\rightarrow[\kappa]^{<\omega}$, $\kappa\in\dom(f)$, $|f|\leq\kappa$.
     \item For each $\alpha\in\supp(p):=\dom(f)$, $\pi_{\alpha,\kappa}''f(\alpha)$ is a finite increasing sequence.
     \item The domain of $f$ has a $\leq_E$-maximimal element $mc(p):=\alpha=\max(\supp(p))$.
     \item $\pi_{mc(p),\kappa}''f(mc(p))=f(\kappa)$.
     \item For every $\gamma\in \supp(p)$, $\pi_{mc(p),\gamma}(\max(f(mc(p)))$ is not permitted to $f(\gamma)$.
     \item $T$ is a  $U_{mc(p)}$-splitting tree with stem $f(mc(p))$, namely, for $s\in T$, either $s\leq t$, or $s\geq t$ and $\Succ_T(s):=\{\alpha<\kappa\mid s{}^{\smallfrown}\alpha\in T\}\in U_{mc(p)}$. 
     \item For every $\nu\in \Succ_T(f(mc(p)))$, $$|\{\gamma\in \supp(p)\mid \nu\text{ is permitted to } f(\gamma)\}|\leq \nu^0.$$
 \end{enumerate}    
 \textit{The order} is defined $p\leq q$ if:
 \begin{enumerate}
     \item $\supp(p)\subseteq \supp(q)$.
     \item for $\gamma\in\supp(p)$, $f^q(\gamma)$ is an end-extension of $f^p(\gamma)$.
     \item $f^q(mc(p))\in T^p$.
     \item for $\gamma\in \supp(p)$, $f^q(\gamma)\setminus f^p(\gamma)=\pi_{mc(p),\gamma}''f^q(mc(p))\setminus f^p(mc(p))\restriction (i+1)$, where $i$ is maximal such that $f^q(mc(p))$ is not  permitted for $f^p(\gamma)$.
     \item $\pi_{mc(q),mc(p)}''T^q\subseteq T^p$.
     \item For every $\gamma\in\supp(p)$, and $\nu\in \Succ_{T^q}(f^q(mc(q)))$, such that $\nu$ is permitted for $f^q(\gamma)$ (So by condition (7) there are only $\nu^0$-many such $\gamma$'s) then $\pi_{mc(q),\gamma}(\nu)=\pi_{mc(p),\gamma}(\pi_{mc(q),mc(p)}(\nu))$
 \end{enumerate}
\end{definition}
\subsection{Canonical functions} The main construction of this paper uses the notion of canonical functions:
\begin{definition}
For every limit ordinal $\delta<\kappa^+$, fix a cofinal sequence $\bar{\delta}=\{ \delta_i\mid i<cf(\delta)\}$. Let us define inductively functions $\tau_\alpha:\kappa\rightarrow\kappa$ for $\alpha<\kappa^+$:
$$\tau_0(x)=0.$$
$$\tau_{\alpha+1}(x)=\tau_\alpha(x)+1.$$
$$\text{For limit }\delta, \ \tau_{\delta}(x)=\sup_{y<\min(x,cf(\delta))}\tau_{\delta_y}(x).$$

\end{definition}
\begin{proposition}\label{properties of canonical functions}
Let $\lambda\leq\kappa$ be a regular cardinal. Then:
\begin{enumerate}
    \item For every $\alpha<\beta<\lambda^+$, $\{\nu\mid \tau_{\alpha}(\nu)\geq\tau_{\beta}(\nu)\}$ is bounded in $\lambda$.
    \item For every any $\alpha<\lambda^+$, $\tau_{\alpha}:\lambda\rightarrow\lambda$.
    \item For every normal measure $\mathcal{V}$ on $\lambda$, and for every $\alpha<\lambda^+$, $[\tau_{\alpha}]_{\mathcal{V}}=\alpha$.
    \item If $\lambda<\kappa$, then for every $\beta$, $\tau_{\beta}(\lambda)<\lambda^+$.
\end{enumerate}
\begin{proof}
For $(1)$, we prove inductively on $\beta<\lambda^+$ that for every $\alpha<\beta$, $(1)$ holds.
For $\beta=0$ this is vacuous. The successor stage  is also easy since for every $x$, $\tau_{\beta}(x)<\tau_{\beta+1}(x)$ so if $\alpha<\beta$ then by induction hypothesis there is $\xi<\lambda$ from which $\tau_{\beta}$ dominates $\tau_{\alpha}$ i.e. $\forall\nu\in (\xi,\lambda).\tau_\alpha(\nu)<\tau_{\beta}(\nu)$. It follows that for the same $\xi$, $\tau_{\alpha}(\nu)<\tau_{\beta+1}(\nu)$. As for limit points $\delta$. Fix any $\alpha<\delta$, then there is $i< cf(\delta)\leq \lambda$ such that $\delta_i>\alpha$. By induction hypothesis there is $\xi_i<\lambda$ such that $\tau_{\delta_i}(\nu)>\tau_{\alpha}(\nu)$ for every $\nu\in(\xi_i,\lambda)$. Let $\xi^*:=\max\{\xi_i, i\}+1<\lambda$. It follows that for every $\nu\in(\xi^*,\lambda)$, $\nu>i$, hence $$\tau_{\delta}(\nu)=\sup_{y<\min(\nu,cf(\delta))}\tau_{\delta_y}(\nu)\geq \tau_{\delta_i}(\nu)>\tau_{\alpha}(\nu).$$

Prove $(2),(3),(4)$ by induction on $\alpha<\lambda^+$. For $\alpha=0$ this is trivial. Suppose that $(2),(3),(4)$ holds for $\alpha$ then clearly by induction hypothesis $\tau_{\alpha+1}:\lambda\rightarrow\lambda$, and $\tau_{\alpha+1}(\lambda)=\tau_{\alpha}(\lambda)+1<\lambda^+$, namely $(2),(4)$ follows. Also, $\lambda=\{\nu<\lambda\mid \tau_{\alpha}(\nu)+1=\tau_{\alpha+1}(\nu)\}\in\mathcal{V}$, hence by L\'{o}s theorem and the induction hypothesis:
$$\alpha+1=[\tau_{\alpha}]_{\mathcal{V}}+1=[\tau_{\alpha+1}]_{\mathcal{V}}.$$ 
 Suppose that $\delta<\lambda^+$ is limit, then by induction hypothesis, for every $x<\lambda$ and $y<\min(x,cf(\delta))<\lambda$,  $\tau_{\delta_y}(x)<\lambda$ . It follows from the regularity of $\lambda$ that $$\tau_\delta(x)=\sup_{y<\min(x,cf(\delta))}\tau_{\delta_y}(x)<\lambda.$$  This concludes $(2)$. Also, $(4)$ follows similarly using the regularity of $\lambda^+$. As for $(3)$, we use $(1)$ to conclude that for every $\alpha<\delta$, $\{\nu<\lambda\mid \tau_{\alpha}(\nu)\geq \tau_{\delta}(\nu)\}$ is bounded. Hence by induction $\alpha=[\tau_{\alpha}]_{\mathcal{V}}<[\tau_{\delta}]_{\mathcal{V}}$. It follows that $\delta\leq [\tau_{\delta}]_{\mathcal{V}}$. For the other direction, suppose that $[f]_{\mathcal{V}}<[\tau_{\delta}]_{\mathcal{V}}$, then $$E:=\{x<\lambda\mid f(x)<\tau_{\delta}(x)\}\in\mathcal{V}.$$ By definition of $\tau_{\delta}$, for every $x\in E$, there is $y_x<\min(x,cf(\delta))$ such that $\tau_{\delta_{t_x}}(x)>f(x)$. the function $x\mapsto y_x$ is regressive, and by normality we conclude that there is $y^*<cf(\delta)$ and $E'\subseteq E$ such that for every $x\in E'$, $f(x)<\tau_{\delta_{y^*}}(x)$.
Hence $[f]_{\mathcal{V}}<[\tau_{\delta_{y^*}}]_{\mathcal{V}}=\delta_{y^*}<\delta$ and in turn $\delta=[\tau_{\delta}]_{\mathcal{V}}$. 
\end{proof}
\end{proposition}
\section{The results where GCH holds}
\subsection{Non-Galvin ultrafilter from optimal assumption}

In \cite{BenGarShe}, Garti, Shelah and the first author constructed a model with a $\kappa$-complete ultrafilter which contains $Cub_\kappa$ and fails to satisfy the Galvin property.
The initial assumption was a supercompact cardinal and the construction went through adding slim Kurepa trees.

Here we present a different construction. Our initial assumption will be a measurable cardinal and the property obtained will be a certain strengthening of the negation of the Galvin property.
It will be used further to produce many Cohens.

Let us first present the stronger form of negation:
\begin{definition}
Let $U$ be a $\kappa-$complete ultrafilter non-normal over $\kappa$.
We call a family  $\{ A_\alpha \mid \alpha<\kappa^+\}\subseteq U$ a \emph{strong witness} for the failure of the Galvin property iff for every subfamily $\l A_{\alpha_\xi} \mid \xi<\kappa\r$ of size $\kappa$ the following hold:
$$ \text{ for every }\zeta, \kappa\leq \zeta< [id]_U, \ [id]_U \not \in A'_{\alpha_\zeta},$$
$$\text{where }\l A'_{\alpha_\zeta} \mid \zeta<j_U(\kappa)\r=j_U(\l A_{\alpha_\xi} \mid \xi<\kappa\r).$$
\end{definition}

\begin{remark}
\begin{enumerate}
  \item Note that the interval $[\kappa, [id]_U)$ is non-empty since $U$ is not normal.

  \item The family $\{ A_\alpha \mid \alpha<\kappa^+\}$ witnesses the failure of the Galvin property for $U$.
  
  \textit{Proof.} since whenever  $\l A_{\alpha_\xi} \mid \xi<\kappa\r$ is a subfamily of size $\kappa$, then $\bigcap_{\xi<\kappa} A_{\alpha_\xi}$ is not in $U$. Otherwise, suppose that $\bigcap_{\xi<\kappa} A_{\alpha_\xi}=B \in U$.
     Then $[id]_U \in j_U(B)$, but $j_U(B)=\bigcap_{\zeta<j_U(\kappa)}A'_{\alpha_\zeta}$. However, $[id]_U \not \in A'_{\alpha_\zeta}$, for every $\zeta, \kappa\leq \zeta< [id]_U$. Contradiction.
     \end{enumerate}

\end{remark}
\begin{lemma}
Suppose that $\{ A_\alpha \mid \alpha<\kappa^+\}$ is a strong witness for the failure of the Galvin property of the ultrafilter $U$ over $\kappa$.
Let $U^0=\{X\subseteq \kappa \mid \kappa\in j_U(X)\}$ be a projection of $U$ to a normal ultrafilter, $\nu\mapsto \pi_{nor}(\nu)$ a projection map  and
$k: M_{U^0}\to M_U$ the corresponding elementary embedding. Assume that $crit(k)=j_{U^0}(\kappa)=[id]_U$. Then $[id]_U \not \in B$, for every $B \in j_U(\{ A_\alpha \mid \alpha<\kappa^+\})$ which is in  $\rng(k)\setminus \rng(j_U)$.

\end{lemma}
\begin{proof} Let $B$ be as in the statement of the lemma. Pick $A'\subseteq j_{U^0}(\kappa)$ such that $k(A')=B$. Then $A' \not \in \rng(j_{U^0})$, since otherwise its image $B$ will be in the 
range of $j_U=k\circ j_{U^0}$. Denote by
$$\{A'_{\nu}\mid \nu<j_{U^0}(\kappa^+)\}=j_{U^0}(\{A_i\mid i<\kappa^+\})$$  $$\{A''_{\nu}\mid \nu<j_{U}(\kappa^+)\}=j_{U}(\{A_i\mid i<\kappa^+\})$$

Since $U^0$ is normal, there is $f:\kappa\rightarrow \kappa^+$ such that $A'=A'_{j_{U^0}(f)(\kappa)}$ and thus $$B=k(A')=k(A'_{{j_{U^0}(f)(\kappa)}})=A''_{j_{U}(f)(\kappa)}.$$
Since $B$ is not in the range of $k$, $f$ is not constant. Recall that $\{ A_\alpha \mid \alpha<\kappa^+\}$ is a strong witness for $U$  being non-Galvin ultrafilter over $\kappa$.
Apply this to the family $\{ A_{f(\nu)} \mid \nu<\kappa \}.$
It follows that $[id]_U\not \in A''_{j_{U}(f)(\kappa)}=B$\end{proof}

Before proving the main result of this section we present two preservation theorems for being a strong witnesses for the failure of the Galvin property. These theorems are not used later and the reader can proceed directly to Theorem \ref{universe theorem}.
\begin{theorem}
Assume $2^\kappa=\kappa^+$.
Suppose that the family $\{ A_\alpha \mid \alpha<\kappa^+\}$ is a strong witness for $U$ being a non-Galvin ultrafilter over $\kappa$.
Let $U^0=\{X\subseteq \kappa \mid \kappa\in j_U(X)\}$ be a projection of $U$ to a normal ultrafilter, $\nu\mapsto \pi_{nor}(\nu)$ a projection map  and 
$k: M_{U^0}\to M_U$ the corresponding elementary embedding. Assume that $crit(k)=j_{U^0}(\kappa)$ and $[id]_U=j_{U^0}(\kappa)$.
\\Suppose that $V^*$ is an extension of $V$ in which all the embeddings $j_{U^0},j_U, k$ extend to an elementary  embedding $j^{0*}:V^*\to M^{0*},j^*:V^*\to M^*$,
$k^*:M^{0*}\to M^*$. Define $U^*=\{X \subseteq \kappa\mid [id]_U \in j^*(X)\}$.
\\Then $\{ A_\alpha \mid \alpha<\kappa^+\}$ is a strong witness that $U^*$ is a non-Galvin ultrafilter over $\kappa$.

\end{theorem}
\begin{proof}
Note that $(\kappa^+)^{V^*}=(\kappa^+)^V$. Just otherwise, $(\kappa^{++})^V $ will be $\leq (\kappa^+)^{V^*}$, and then, $j^*(\kappa)>(\kappa^{++})^V$.
This is impossible, since $j^*$ extends $j_U$. The rest follows from the previous lemma and the fact that $[\kappa, [id]_U)\subseteq\rng(k)\setminus\rng(j_{U})$ since $crit(k)=j_{U^0}(\kappa)=[id]_U$. 
\end{proof}

\begin{theorem}
Assume $2^\kappa=\kappa^+$.
Suppose that $\{ A_\alpha \mid \alpha<\kappa^+\}$ is a strong witness for $U$ being a non-Galvin ultrafilter over $\kappa$ which contains $Cub_\kappa$ and  be a witnessing family.
\\Let $V^*$ be a $\kappa-$c.c. extension of $V$ in which $j_U$ extends to an elementary  embedding $j^*:V^*\to M^*$, where $M^*$ is a corresponding extension of $M_U$.
\\Define $U^*=\{X \subseteq \kappa\mid [id]_U \in j^*(X)\}$.
\\Then $\{ A_\alpha \mid \alpha<\kappa^+\}$ is a strong witness that $U^*$ is a non-Galvin ultrafilter over $\kappa$.

\end{theorem}
\begin{proof}
Suppose now that  $\l A_{\alpha_\xi} \mid \xi<\kappa\r$ is a subfamily of $\{ A_\alpha \mid \alpha<\kappa^+\}$ of size $\kappa$ in ${V^*}$.
\\Work in $V$. Let $\lusim{\alpha}_\xi$ be a name of $\alpha_\xi$.
By $\kappa-$c.c., then for every $\xi<\kappa$ there will be $s_\xi\subseteq \kappa^+$ of cardinality less than $\kappa$, such that $\Vdash \lusim{\alpha}_\xi\in s_\xi$.
\\Let $S=\sup_{\xi<\kappa}s_\xi$. Enumerate $S=\l \beta_i\mid i<\kappa\r$ such that we if $\beta_i\in s_\zeta$ and $\beta_j\in s_\mu$ where $\zeta<\mu$ then $i<j$ i.e. enumerate first $s_0$ then $s_1$ and so on, such that the resulting enumeration of $S$ is of order-type $\kappa$. This is possible since each. $s_\zeta$ has cardinality less than $\kappa$.
Define
$$C=\{\nu<\kappa \mid \forall \xi<\nu (\sup(\gamma\mid \beta_\gamma\in s_\xi)<\nu)\}.$$
Clearly, $C$ is a club.
 Hence $[id]_U\in j_U(C)$. Then, by elementarity, for every $\zeta<[id]_U$, and every $\beta_i\in s'_\zeta$, $i<[id]_U$.

Let us use the fact the the sequence $\l A_\alpha\mid \alpha<\kappa^+\r$ is a strong witness for $U$ being non-Galvin, hence $[id]_U \not \in A'_{\beta_\zeta}$, for every  $\kappa\leq \zeta< [id]_U$.
Fix any $\kappa\leq\xi<[id]_U$, then by elementarity we have $\Vdash \lusim{\alpha}'_\xi\in s'_{\xi}$  in $M_U$. Therefore there is some $\gamma<\kappa$ such that $\alpha'_\xi=\beta_\gamma$. Clearly, $\gamma\geq\kappa$, and by the closure property of $[id]_U$, we conclude that $\gamma<[id]_U$. 
Hence, in $M^*$, $[id]_U \not \in A'_{\beta'_\gamma}=A'_{\alpha'_\xi}$, as wanted.
\end{proof}

\begin{theorem}\label{universe theorem}
 Assume GCH and let $\kappa$ be measurable in $V$. Then there is a cofinality preserving forcing extension $V^*$ in which there is a $\kappa$-complete ultrafilter $W$ over $\kappa$ which concentrates on regulars, extends $Cub_\kappa$, and has a strong witness for the failure of Galvin's property.
\end{theorem}
\begin{proof}
The forcing is simply adding for each inaccessible $\alpha\leq\kappa$, $\alpha^+$-many Cohen functions to $\alpha$. Namely, consider the Easton support iteration $$\l\mathcal{P}_\alpha,\lusim{Q}_\beta\mid \alpha\leq\kappa+1, \beta\leq\kappa\r$$ such that for $\alpha\leq\kappa$, $\lusim{Q}_\alpha$ is trivial unless $\alpha$ is inaccessible, in which case it is a $\mathcal{P}_\alpha$-name for $\Add(\alpha,\alpha^+)$. 

Let $G:=G_\kappa*g_\kappa$ be $V$-generic for $\mathcal{P}_\kappa*\lusim{Q}_\kappa$. Denote $\l f_{\kappa,\alpha}\mid \alpha<\kappa^+\r$ be the enumeration of the $\kappa^+$ Cohen functions added by $g_\kappa$. 
The idea is that the sets which are going to be a strong witness for the failure of the Galvin property are $\l A_\alpha\mid \alpha<\kappa^+\r$, where
$$A_\alpha=\{\beta<\kappa\mid f_\alpha(\beta)=1\}.$$

The next step is to construct the measure for this witness by extending ground model embeddings to $V[G]$.
 Let $U\in V$ be a normal measure over $\kappa$ and consider the second ultrapower by $U$ and the corresponding commutative diagram  
 $$j_1:=j_{U}:V\rightarrow M_U=:M_1, \ j_2:=j_{U^2}:V\rightarrow M_{U^2}=:M_2$$
 $$ k:M_1\rightarrow M_{2}, \ j_{2}=k\circ j_1$$
 where $k$ is simply the ultrapower embedding defined in $M_U$ using the ultrafilter $j_{1}(U)$. Denote $\kappa_1=j_1(\kappa)$ and $\kappa_2=j_{2}(\kappa)$, then $k(\kappa_1)=\kappa_2$.
 
 By Easton support and elementarity, $$j_1(\mathcal{P}_\kappa*\lusim{Q}_\kappa)=\mathcal{P}_\kappa*\lusim{Q}_\kappa*\mathcal{P}_{(\kappa,\kappa_1)}*\lusim{Q}_{\kappa_1}.$$ 
 Where $\mathcal{P}_{(\kappa,\kappa_1)}*\lusim{Q}_{\kappa_1}$ is the quotient forcing above $\kappa$, which is forcing equivalent to the continuation of the iteration above $\kappa$ using the same recipe as $\mathcal{P}_\kappa$.
 
In $V[G]$, let us first construct an $M$-generic filter for $j_1(\mathcal{P}_\kappa*\lusim{Q}_\kappa)$. Take $G_{\kappa}*g_{\kappa}$ to be the generic up to $\kappa$ including $\kappa$. Above $\kappa$, from the point of view of $V[G]$, we have $\kappa^+$-closure for   $\mathcal{P}_{(\kappa,\kappa_1)}$. By $GCH$, and since $j_1$ is an ultrapower by a measure, there are only $\kappa^+$-many dense open subsets of this forcing to meet. Therefore we can construct in $V[G]$ by standard construction an $M_1[G]$-generic filter $G_{(\kappa,\kappa_1)}$ for $\mathcal{P}_{(\kappa,\kappa_1)}$. By $\kappa_1^+-cc$ of $Q_{\kappa_1}$, we can find $g'_{\kappa_1}$  which is $M_1[G*G_{(\kappa,\kappa_1)}]-$generic for $Q_{\kappa_1}$. 
    We need to change the values of $g'_{\kappa_1}=\l f'_{\kappa_1,\alpha}\mid\alpha<\kappa_1^{+}\r$ to $g_{\kappa_1}=\l f_{\kappa_1,\alpha}\mid \alpha<\kappa_1^{+}\r$ such that for every $\alpha<\kappa^{+}$, $f_{\kappa_1,j_1(\alpha)}\restriction\kappa=f_{\kappa,\alpha}$. This will ensure that the
Silver criterion to lift an elementary embedding holds, namely, $j_1''G_{\kappa}*g\subseteq G_{\kappa}*g*G_{(\kappa,\kappa_1)}*g'_{\kappa_1}$. Also, we would like to tweak the values of $f_{\kappa_1,j_1(\alpha)}(\kappa)$ to ensure that the sets $A_\alpha$ are members of the ultrafilter generated by $\kappa$. By the definition of $A_\alpha$, the way to do this is to set $f_{\kappa_1,j_1(\alpha)}(\kappa)=1$. 

Formally, for each condition $p\in \Add(\kappa_1,\kappa_1^+)^{M_1[G_{\kappa}*G*G_{(\kappa,\kappa_1)}]}$,  define a function $p^*$ with $\dom(p^*)=\dom(p)$ and for every $\l \gamma,\alpha\r\in\dom(p^*)$,
$$p^*(\l\gamma,\alpha\r)=\begin{cases} f_{\kappa,\beta}(\gamma) & \gamma<\kappa\wedge j_1(\beta)=\alpha\\
1 & \gamma=\kappa\wedge j_1(\beta)=\alpha\\
p(\l\gamma,\alpha\r) & else\end{cases}.$$ Let $g_{\kappa_1}:=\{p^*\mid p\in g'_{\kappa_1}\}$. Clearly, the functions $\l f_{\kappa_1,\alpha}\mid\alpha<\kappa_1^+\r$ derived from $g_{\kappa_1}$ satisfy that $f_{\kappa_1,j_1(\beta)}\restriction\kappa=f_{\kappa,\beta}$ and $f_{\kappa_1,j_1(\beta)}(\kappa)=1$ for every $\beta<\kappa^+$. It remains to show that $g_{\kappa_1}$ is generic:
\begin{lemma}\label{less thank kappa many changes} The filter
$g_{\kappa_1}$ is $\Add(\kappa_1,\kappa_1^+)^{M_1[G_{\kappa}*G*G_{(\kappa,\kappa_1)}]}$-generic filter over $M_1[G_{\kappa}*g*G_{(\kappa,\kappa_1)}]$.
\end{lemma}
\begin{proof}
 First let us prove that $g_{\kappa_1}\subseteq \Add(\kappa_1,\kappa_1^+)^{M_1[G_{\kappa}*G*G_{(\kappa,\kappa_1)}]}$. Indeed, $g'_{\kappa_1}\subseteq \Add(\kappa_1,\kappa_1^+)^{M_1[G_{\kappa}*G*G_{(\kappa,\kappa_1)}]}$ and for any $p\in g'_{\kappa_1}$, $$M_1[G_{\kappa}*G*G_{(\kappa,\kappa_1)}]\models|p|<\kappa_1,$$ hence $\dom(p)_{\leq\kappa}:=\{\alpha\mid\exists\l\gamma,\alpha\r\in\dom(p),\gamma\leq\kappa\}$ is bounded in $\kappa_1^+$ while $j_1''\kappa^+$ is unbounded. It follows that there is $\theta<\kappa^+$ such that $$\dom(p)_{\leq\kappa}\cap j_1''\kappa^+\subseteq j_1''\theta.$$ Hence from the $V$-perspective, $|\dom(p)_{\leq\kappa}\cap j_1''\kappa^+|\leq\kappa$. The difference between $p$ and $p^*$ is only on the coordinates of $\dom(p)_{\leq\kappa}\cap j_1''\kappa^+$ and by closure of $M_1[G_{\kappa}*g*G_{(\kappa,\kappa_1)}]$ to $\kappa$-sequences it follows that $$p^*\in \Add(\kappa_1,\kappa_1^+)^{M_1[G_{\kappa}*G*G_{(\kappa,\kappa_1)}]}, \ g_{\kappa_1}\subseteq\Add(\kappa_1,\kappa_1^+)^{M_1[G_{\kappa}*G*G_{(\kappa,\kappa_1)}]}.$$
 
 To see that $g_{\kappa_1}$ is generic over $M_1[G_{\kappa}*G*G_{(\kappa,\kappa_1)}]$, let $D\in M_1[G_{\kappa}*G*G_{(\kappa,\kappa_1)}]$ be dense open. In $M_1[G_{\kappa}*G*G_{(\kappa,\kappa_1)}]$, define $D^*$ to consist of all conditions 
 $p\in \Add(\kappa_1,\kappa_1^+)$. Such that $$\forall q. \dom(q)=\dom(p)\wedge |\{x\mid p(x)\neq q(x)\}|\leq \kappa\rightarrow q\in D$$
 then $D^*$ is dense open. To see this, pick any $p\in \Add(\kappa_1,\kappa_1^+)^{M_1[G_{\kappa}*G*G_{(\kappa,\kappa_1)}]}$ and enumerate by $\l q_r\mid r<\theta\r$ all the conditions $q$ such that  $$\dom(q)=\dom(p)\wedge |\{x\mid p(x)\neq q(x)\}|\leq\kappa.$$ Note $\theta<\kappa_1$ since $\kappa_1$ is inaccessible in $ M_U[G_{\kappa}*g*G_{(\kappa,\kappa_1)}]$. 
 We define inductively and increasing sequence $\l p_r\mid r<\theta\r$, and exploit the $\kappa_1$-closure of $\Add(\kappa_1,\kappa_1^+)$ to take care of limit stages. Define $p_0=p$, and suppose that $p_r$ is defined, let $p_{r+1}':=q_r\cup p_r\restriction(\dom(p_r)\setminus \dom(p))$,  find $p_{r+1}'\leq t_{r+1}\in D$ which exists by density and set $$p_{r+1}=p_r\restriction \dom(p)\cup t_{r+1}\restriction(\dom(t_{r+1})\setminus \dom(p)).$$ Then $p_r\leq p_{r+1}$. Let $$p^*:=\cup_{r<\theta} p_{r}$$ then $p^{*}$ has the property that for $\kappa$ many changes of $p^*$ from the domain of $p$ stays inside $D$. Namely any $q$ with $\dom(q)=\dom(p^*)$, $$q\restriction (\dom(p^*)\setminus \dom(p))=p^*\restriction(\dom(p^*)\setminus \dom(p))$$ and $|\{x\in \dom(p)\mid p(x)\neq q(x)\}|\leq\kappa$, $q\restriction \dom(p)=q_r$ for some $r$, therefore $q\geq t_{r+1}\in D$. Now we define inductively $\l p^{(r)}\mid r<\kappa^+\r$, $p^{(0)}=p$ at limit we take union, and at successor step we take $p^{(r+1)}=(p^{(r)})^*$.
 We claim that $p_*:=\cup_{r<\kappa^+}p^{(r)}\in D^*$. First note that $\kappa^+<\kappa_1$ hence $|p_*|<\kappa_1$ (all the definition is inside $M_U[G_{\kappa}*g_{\kappa}*G_{(\kappa,\kappa_1)}]$). Let $q$ be any condition with $\dom(q)=\dom(p^*)$ and denote by $$I=\{x\in \dom(p_*)\mid q(x)\neq p_*(x)\}$$ and suppose that $|I|\leq \kappa$. Since $\dom(p_*)=\cup_{r<\kappa^+}\dom(p^{(r)})$ and $\dom(p^{(r)})$ is $\subseteq$-increasing, there is $j<\kappa^+$ such that $I\subseteq \dom(p^{(j)})$. The condition $q\restriction I$ is enumerated in the construction of $p^{(j+1)}$, hence $q\restriction \dom(p^{(j+1)})\in D$ and since $D$ is open, $q\in D$. This means that $p_*\in D^*$.
 
 Finally, by genericity of $g'_{\kappa_1}$, we can find $p\in D^*\cap g'_{\kappa_1}$. By definition, $p^*\in g_{\kappa_1}$ and since $\dom(p^*)=\dom(p)$ and $|\{x\mid p(x)\neq p^*(x)\}|\leq \kappa$ it follows that $p^*\in D$.
\end{proof}

Denote by $H=G_\kappa*g_{\kappa}*G_{(\kappa,\kappa_1)}*g_{\kappa_1}$, then $j_1''G\subseteq H$. Let $$j_1^*:V[G]\rightarrow M_1[H]$$ be the extended ultrapower and derive the normal ultrafilter over $\kappa$,
$$U_1:=\{X\subseteq \kappa\mid \kappa\in j^*_1(X)\}$$
then $U\subseteq U_1$ and $j^*_1=j_{U_1}$. Indeed let $k_1:M_{U_1}\rightarrow M_1[H]$ be the usual factor map $k_1(j_{U_1}(f)(\kappa))=j^*_1(f)(\kappa)$. We will prove that $k_1$ is onto and therefore $k_1=id$. For every $A\in M_1[H]$, there is a name $\lusim{A}\in M_1$ such that $A=(\lusim{A})_H$. $M_U$ is the ultrapower by $U$, hence there is $f\in V$ such that $j_1(f)(\kappa)=\lusim{A}$. By elementarity for every $\alpha<\kappa$, $f(\alpha)$ is name. In $V[G]$ define $f^*(\alpha)=(f(\alpha))_G$, then by elementarity
$$k_1(j_{U_1}(f)(\kappa))=j^*_1(f^*)(\kappa)=(j^*_1(f)(\kappa))_{j(G)}=(j_1(f)(\kappa))_H=(\lusim{A})_H=A.$$

Denote by $M_1^*=M_1[H]$ and consider $j^*_1(U_1)\in M^*_1$. Let us now define inside $M^*_1$ an $M_{2}$-generic filter for $$j_{2}(\mathcal{P}_\kappa*\lusim{Q}_\kappa)=\mathcal{P}_{\kappa_1}*\lusim{Q}_{\kappa_1}*\mathcal{P}_{(\kappa_1,\kappa_2)}*\lusim{Q}_{\kappa_2},$$ 
in a similar fashion as $H$ was defined. First we take $H$ to be the generic for $\mathcal{P}_{\kappa_1}*\lusim{Q}_{\kappa_1}$. Note that $M_{2}$ is closed under $\kappa_1$-sequences with respect to $M_1$.  Therefore, from the $M^*_1$-point of view, $\mathcal{P}_{(\kappa_1,\kappa_2)}*\lusim{Q}_{\kappa_2}$ is $\kappa_1^+$-closed, and we can construct an $M_{2}[H]$-generic filter  $G_{(\kappa_1,\kappa_2)}*g'_{\kappa_2}\in M_1^*$ for it. We change the values of $g'_{\kappa_2}$ a bit differently from the way we changed the values of $g'_{\kappa_1}$.  If $\alpha<\kappa_1^+$ is of the form $j_1(\beta)$ let $f_{\kappa_2,k(\alpha)}(\kappa_1)=1$ (to guarantee that $A_\alpha$'s belongs to the ultrafilter generated by $\kappa_1$) and if $\alpha\in\kappa_1^+\setminus j_1''\kappa^+$
let $f_{\kappa_2,k(\alpha)}(\kappa_1)=0$. Also, we would like that $f_{\kappa_2,\kappa_1}(0)=\kappa$. Formally, for every $p\in \Add(\kappa_2,\kappa_2^+)^{M_2[H*G_{(\kappa_1,\kappa_2)}]}$,
 define $p^*$ to be a function with $\dom(p)=\dom(p^*)$ and for every  $\l\gamma,\alpha\r\in\dom(p^*)$,
 $$p^*(\l\gamma,\alpha\r)=\begin{cases} f_{\kappa_1,\beta}(\gamma)& \gamma<\kappa_1\wedge \alpha=k(\beta) \\
 1 & \gamma=\kappa_1\wedge \alpha=k(j_1(\beta))\\
 0 & \gamma=\kappa_1\wedge \alpha=k(\beta),\beta\notin j_1''\kappa^+\\
 \kappa & \gamma=0\wedge \alpha=\kappa_1\\
 p(\l\gamma,\alpha\r) & else\end{cases}.$$
 
Denote by $g_{\kappa_2}=\{p^*\mid p\in g'_{\kappa_2}\}\in V[G]$ the resulting filter. It is important that for each $p\in g'_2$, the set $$X_1:=j_{2}''\kappa^+\cap \dom(f)_{\leq\kappa_1}=\{j_2(\alpha)\mid \exists \l\gamma,j_2(\alpha)\r\in\dom(f), \gamma\leq\kappa_1\}$$ has size at most $\kappa$. This ensured that $X_1\in M_1^*$. Also, $k''\kappa_1^+$ is unbounded in $\kappa_2^+$ and conditions in $\Add(\kappa_2,\kappa_2^+)^{M_2[H*G_{(\kappa_1,\kappa_2)}]}$ have $M_{2}[H*G_{(\kappa_1,\kappa_2)}]$-cardinality less than $\kappa_2$, which guarantees that for each $p\in \Add(\kappa_2,\kappa_2^+)$, $$X_2:= k''\kappa_1^+\cap \dom(p)_{\leq\kappa_1}$$ has size at most $\kappa_1$. Note that  $p^*$  is definable in $M_1^*$ from the parameters $p,X_1,X_2\in M_1^*$, and $p^*$ differs from $p$ at most on $\kappa_1$-many values. By the closure of $M_{2}[H*G_{(\kappa_1,\kappa_2)}]$ to $\kappa_1$-sequences from $M_1^*$, $$p^*\in M_{2}[H*G_{(\kappa_1,\kappa_2)}]\text{ and  }g_{\kappa_2}\subseteq \Add(\kappa_2,\kappa_2^+)^{M_{2}[H*G_{(\kappa_1,\kappa_2)}]}.$$ The genericity argument of \ref{less thank kappa many changes} extends to the models $M_1$ and $M_{2}[H*G_{(\kappa_1,\kappa_2)}]$, hence $g_{\kappa_2}$ is $M_{2}[H*G_{(\kappa_1,\kappa_2)}]$-generic. 
Denote by $M_2^*=M_{2}[H*G_{(\kappa_1,\kappa_2)}*g_{\kappa_2}]$. It follows that $k$ can be extended (in $V[G]$) to $k^*$ and also $j_{2}$ to $j^*_2=k^*\circ j^*_1:V[G]\rightarrow M^*_2$. Finally, let $$W:=\{X\in P^{V[G]}(\kappa)\mid \kappa_1\in j^*_2(X)\}\in V[G].$$
Let us prove that $W$ witnesses the theorem: 
\begin{claim}\label{final claim}
$W$ is a $\kappa$-complete ultrafilter over $\kappa$ such that:
\begin{enumerate}
    \item $j_W=j^*_2$ and $[id]_W=\kappa_1$.
    \item  $Cub_\kappa\subseteq W$.
    \item $\{\alpha<\kappa\mid cf(\alpha)=\alpha\}\in W$.
    \item $\l A_\alpha\mid \alpha<\kappa^+\r$ is a strong witness for the failure of the Galvin property. 
\end{enumerate}
\end{claim}
\begin{proof}
 To see (1), let us denote by $j_W:V[G]\rightarrow M_W$ the ultrapower embedding by $W$ and $k_W:M_W\rightarrow M^*_2$ defined by $k_W([f]_W)=j^*_2(f)(\kappa_1)$ the factor map satisfying $k_W\circ j_W=j^*_2$. Let us argue that $k_W$ is onto and therefore $k_W=id$ and $[id]_W=\kappa_1$. Indeed, let $A\in M_2^*$ then there is $\lusim{A}\in M_2$ such that $(\lusim{A})_{j_2^*(G)}=A$. Since $j_2=j_{U^2}$ there is $h\in V$ such that $j_2(h)(\kappa,\kappa_1)=\lusim{A}$. Note that $\kappa=j^*_2(f_\kappa)_{\kappa_1}(0)$, hence define in $V[G]$, $h^*(\alpha)=(h(f_{\kappa,\alpha}(0),\alpha))_G$. We have that
 $$k_W([h^*]_W)=j_2^*(h^*)(\kappa_1)=(j_2(h)(\kappa,\kappa_1))_{j_2^*(G)}=(\lusim{A})_{j_2^*(G)}=A$$
 To see $(2)$, for every club $C\in Cub_\kappa$, $j^*_2(C)$ is closed and $j^*_1(C)$ is unbounded in $\kappa_1$. Since $crit(k^*)=\kappa_1$ and $j^*_{2}(C)=k^*(j^*_1(C))$ it follows that $j^*_2(C)\cap\kappa_1=j^*_1(C)$, hence $j^*_2(C)\cap\kappa_1$ is unbounded in $\kappa_1$ which implies that $\kappa_1\in j^*_2(C)$.
 
 For $(3)$, since $M_2^*\models cf(\kappa_1)=\kappa_1$, it follows that $\{\alpha\mid cf(\alpha)=\alpha\}\in W$. Finally, for every $\alpha<\kappa^+$, $$j^*_2(A_\alpha)=\{\beta<\kappa_2\mid f_{\kappa_2,j_2(\alpha)}(\beta)=1\}.$$ Since $j_2(\alpha)=k(j_1(\alpha))$, by the definition of $g_{\kappa_2}$, $f_{\kappa_2,j_2(\alpha)}(\kappa_1)=1$, thus $\kappa_1\in j^*_2(A_\alpha)$, and by definition of $W$, $A_\alpha\in W$.
 
 For (3), let
 $\{A_{\alpha_i}\mid i<\kappa\}$ be any subfamily of length $\kappa$ and $\kappa\leq \eta<[id]_W=\kappa_1$. 
 Denote $$j_2*(\l A_{\alpha_i}\mid i<\kappa\r)=\l A^{(2)}_{\alpha^{(2)}_i}\mid i<\kappa_2\r, \ j_1^*(\l A_{\alpha_i}\mid i<\kappa\r)=\l A^{(1)}_{\alpha^{(1)}_i}\mid i<\kappa_1\r$$
 Since $\kappa\leq\eta<\kappa_1$, then $\eta\notin j_1''\kappa^+$ and thus $\alpha^{(1)}_\eta\notin j_1''\kappa^+$. Also, $k(\alpha^{(1)}_\eta)=\alpha^{(2)}_{k(\eta)}=\alpha^{(2)}_\eta$. Hence by definition, $f_{\kappa_2,\alpha^{(2)}_\eta}(\kappa_1)=0$, hence $\kappa_1\notin A'_{\alpha^{(2)}_\eta}$
\end{proof}
\end{proof}
\subsection{Adding $\kappa^+$-Cohen subsets to $\kappa$ by Prikry forcing}
In this section we will construct a model in which there is a $\kappa$-complete ultrafilter $W$ such that forcing with $\Pri(W)$ adds a generic for $\Add(\kappa,\kappa^+)$. Let us first observe that such an ultrafilter must fail to satisfy the Galvin property:
\begin{proposition}\label{Galvin inplies no Cohens}
If $Gal(U,\kappa,\kappa^+)$ holds then $\Pri(U)$ does not add a $V$-generic filter for $\Add(\kappa,\kappa^+)$. \end{proposition}
\begin{proof}
Suppose that $Gal(U,\kappa,\kappa^+)$ holds and let $G\subseteq \Pri(U)$ be $V$-generic. By \cite[Proposition 1.3]{GitDensity} every set $A\in V[G]$ of size $\kappa^+$ contains a set $B\in V$ of cardinality $\kappa$. Toward a contradiction suppose that $H\in V[G]$ is a $V$-generic filter for $\Add(\kappa,\kappa^+)$. Code $H:\kappa\times\kappa^+\rightarrow 2$ as $X\subseteq\kappa^+$, just pick a bijection $\phi$ from $\kappa^+$ to $\kappa^+\times \kappa$ and let $X=\{\alpha<\kappa^+\mid H(\phi(\alpha))=1\}$. The set $X$ does not contain an old subset of cardinality $\kappa$, this is a contradiction. To see this,  let $Y\in V$ such that $|Y|=\kappa$, proceed with a density argument: any condition $p\in \Add(\kappa,\kappa^+)$ has size $<\kappa$ and therefore can be extended to a condition $p'$ such that for some $y\in Y$, $\phi(y)\in \dom(p')$ and $p'(\phi(y))=0$.
\end{proof}

Hence the failure of the Galvin property is necessary.




\begin{theorem}\label{main theorem}
Assume $GCH$ and that $\kappa$ is a measurable cardinal in $V$. Then there is a cofinality preserving forcing extension $V^*$ in which $GCH$ still holds, and there is a $\kappa$-complete ultrafilter $U^*\in V^*$ over $\kappa$ such that forcing with Prikry forcing $Pikry(U^*)$ introduces a $V^*$-generic filter for $Cohen^{V^*}(\kappa,\kappa^+)$.
\end{theorem}
\begin{proof}
The model $V^*$ is obtained by iterating with Easton support the lottery sum of Cohen forcings for adding
 $\alpha^+-$Cohen functions $\l f_{\alpha \gamma}\mid \gamma<\alpha^+\r$ over $\alpha$, and Cohen${}^2$ for adding two blocks of $\alpha^+-$Cohen functions $$\l f_{\alpha \gamma}\mid \gamma<\alpha^+\r,\l h_{\alpha\gamma}\mid \gamma<\alpha^+\r.$$
More specifically, let
$$\l \calP_\alpha, \lusim{Q}_\beta \mid \alpha\leq \kappa+1, \beta\leq\kappa \r$$
denotes the Easton support iteration, such that for each $\alpha<\kappa$, $\lusim{Q}_\alpha$ is the trivial forcing unless $\alpha$ is inaccessible in which case $\lusim{Q}_\alpha$ is a $\calP_\alpha$-name for the lottery sum $$\Lot(\Add(\alpha,\alpha^+),\Add(\alpha,\alpha^+)\times \Add(\alpha,\alpha^+)).$$
At $\kappa$ itself  we let $\lusim{Q}_\kappa=\Add(\kappa,\kappa^+)$.
Let
$G_\kappa*F_\kappa$ be a $V$-generic subset of $P_\kappa*\lusim{Q}_\kappa$ and let $V^*=V[G_\kappa*F_\kappa]$. We denote by $F_\alpha:=\l f_{\alpha \gamma}\mid \gamma<\alpha^+\r$ the generic Cohen function if $\Add(\alpha,\alpha^+)$ was forced in $G_\kappa$ and by $$F_\alpha:=\l f_{\alpha \gamma}\mid \gamma<\alpha^+\r,\ H_\alpha:=\l h_{\alpha,\gamma}\mid \gamma<\alpha^+\r$$ if $\Add(\alpha,\alpha^+)\times \Add(\alpha,\alpha^+)$ was.

 Let $U\in V$ be a normal ultrafilter, $j_1:=j_U:V\rightarrow M_U$ the corresponding elementary embedding, $\kappa_1=j_1(\kappa)$, $k:=j_{j_1(U)}:M_U\rightarrow M_{U^2}$, $j_2=k\circ j_1$, and $\kappa_2=j_2(\kappa)$. Let us extend $j_1, k,j_2$ in $V[G_{\kappa}*F_\kappa]$:
\\We first extend $j_{1}: V \to M_{{U}}$ to $j_1^{*}: V[G_\kappa*F_\kappa]\to M_{U}[G_{\kappa_1}*F_{\kappa_1}]$.
Do this by taking first $G_{\kappa_1}\cap P_{\kappa}=G_\kappa$, at $\kappa$ we force with the lottery sum so we can choose to force only one block of Cohens and take $F_\kappa$ as a generic. Then defining a master condition sequence, using the closure of the forcing above $\kappa$ in $M_{U}$ exploiting $GCH$ to ensure that there are only $\kappa^+$-many dense sets to meet. This defines $G_{\kappa_1}$. As for $F_{\kappa_1}$, we first find an $M_U[G_{\kappa_1}]$-generic $F'_{\kappa_1}\times H'_{\kappa_1}\in V[G_{\kappa}*F_\kappa]$ again using $GCH$, closure of $M_U[G_{\kappa_1}]$ under $\kappa$-sequences and the closure of the forcing $(\Add(\kappa_1,\kappa_1^+)^2)^{M_U[G_{\kappa_1}]}$. Let us alter some values of $F'_{\kappa_1}$ and $H'_{\kappa_1}$ to define $F_{\kappa_1}=\l f_{\kappa_1,\gamma}\mid \gamma<\kappa_1^+\r$ and $H_{\kappa_1}=\l h_{\kappa_1,\gamma}\mid \gamma<\kappa_1^+\r$ such that for every $\alpha<\kappa_1^+$:
\begin{enumerate}
    \item  $f_{\kappa_1,j_1(\alpha)}\restriction \kappa=h_{\kappa_1,j_1(\alpha)}\restriction\kappa=f_{\kappa,\alpha}$.
    \item $f_{\kappa_1,j_1(\alpha)}(\kappa)=\alpha$.
\end{enumerate}
Formally, we change every pair of partial functions $p=\l p_0,p_1\r \in F'_{\kappa_1}\times H'_{\kappa_1}$ to the pair of partial functions $p_*=\l p_0^*,p_1^*\r$ such that $\dom(p_0^*)=\dom(p_0)$, $\dom(p_1^*)=\dom(p_1)$ and for every $\l \alpha,\delta\r\in \dom(p_0)$:
$$p^*_0(\l\alpha,\delta\r)=\begin{cases} f_{\kappa,\alpha_0}(\delta) &\exists \alpha_0<\kappa^+.\alpha=j_1(\alpha_0)\text{ and }\delta<\kappa\\ \alpha_0 & \exists \alpha_0<\kappa^+.\alpha=j_1(\alpha_0)\text{ and }\delta=\kappa\\ p_0(\l\alpha,\delta\r) & else\end{cases}$$
$$p^*_1(\l\alpha,\delta\r)=\begin{cases} f_{\kappa,\alpha_0}(\delta) &\exists \alpha_0<\kappa^+.\alpha=j_1(\alpha_0)\text{ and }\delta<\kappa\\ p_1(\l\alpha,\delta\r) & else\end{cases}$$

Note that for every $p_0,p_1\subseteq \Add(\kappa_1,\kappa_1^+)^{M_U[G_{\kappa_1}]}$ we only change $\kappa$-many values as $M_U[G_{\kappa_1}]\models|\dom(p_0)|,|\dom(p_1)|<\kappa_1$ hence $$|j_1''\kappa^+\cap\{\alpha\mid\exists\delta.\l\alpha,\delta\r\in \dom(p_0)\}|\leq\kappa$$ since $j_1(\kappa^+)=\bigcup j_1''\kappa^+$, the same holds for $p_1$. It follows that $$p^*\in (\Add(\kappa_1,\kappa_1^+)^2)^{M_U[G_{\kappa_1}]}.$$ Changing less than $\kappa_1$-many values of a generic for $\Add(\kappa_1,\kappa_1^+)^2$ does not impact the genericity. Hence $F_{\kappa_1}\times H_{\kappa_1}:=\{p^*\mid p\in F'_{\kappa_1}\times H'_{\kappa_1}\}\in V[G_{\kappa}*F_\kappa]$ is still $M_U[G_{\kappa_1}]$-generic. 

Since at $\kappa$ we only force $\Add(\kappa,\kappa^+)$, in order to extend $j_1$ we only need a generic for $\Add(\kappa_1,\kappa_1^+)$ in the $M_U$-side. We constructed $F_{\kappa_1}$ so that  $j_1''F_\kappa\subseteq F_{\kappa_1}$ hence $j_1''G_\kappa*F_\kappa\subseteq G_{\kappa_1}*F_{\kappa_1}$ ($H_{\kappa_1}$ will be used later).
Thus in $V[G_{\kappa}*F_{\kappa}]$, we have extended $j_1\subseteq j_1^*:V[G_{\kappa}*F_{\kappa}]\rightarrow M_U[G_{\kappa_1}*F_{\kappa_1}]$. Let us note that $j_1^*$ is actually the elementary embedding derived from the normal measure $U\subseteq U^0:=\{X\in P^{V[G_\kappa*F_\kappa]}(\kappa)\mid \kappa\in j_1^*(X)\}$:

Clearly the function $k_0:M_{U^0}\rightarrow M_U[G_{\kappa_1}*F_{\kappa_1}]$ defined by $k_0([f]_{U^0})=j^*_1(f)(\kappa)$ is elementary. To see the $k_0=id$ let us prove that $k_0$ is onto. Fix $A=(\lusim{A})_{G_{\kappa_1}*F_{\kappa_1}}\in M_U[G_{\kappa_1}*F_{\kappa_1}]$ and let $f\in V$ be such that $j_1(f)(\kappa)=\lusim{A}$ and define in $V[G_{\kappa}*F_\kappa]$ the function $f^*(x)=(f(f_{\kappa,\kappa}(x)))_{G_{\kappa}*F_{\kappa}}$. Then
$$k_0(j_{U^0}(f^*)(\kappa))=j^*_1(f^*)(\kappa)=(j^*_1(f)(j^*_1(f_{\kappa,\kappa})(\kappa)))_{G_{\kappa_1}*F_{\kappa_1}}=$$
$$=(j_1(f)(\kappa))_{G_{\kappa_1}*F_{\kappa_1}}=(\lusim{A})_{G_{\kappa_1}*F_{\kappa_1}}=A.$$
 
Recall that we have constructed the function $H_{\kappa_1}\in V[G_{\kappa}*F_{\kappa}]$ such that $F_{\kappa_1}\times H_{\kappa_1}$ is $M_U[G_{\kappa_1}]-$generic for $\Add(\kappa_1,\kappa_1^+)^2$. Now we wish to extend $k:M_U\rightarrow M_{U^2}$ to $k^*:M_U[G_{\kappa_1}*F_{\kappa_1}]\rightarrow M_{U^2}[G_{\kappa_2}*F_{\kappa_2}]$ in $V[G_{\kappa}*F_{\kappa}]$. We do this by taking $G_{\kappa_2}\cap \kappa_1= G_{\kappa_1}$, at $\kappa_1$ we force $\Add(\kappa_1,\kappa_1^+)\times \Add(\kappa_1,\kappa_1^+)$ putting the generic $F_{\kappa_1}\times H_{\kappa_1}$, then exploiting the closure and $GCH$ to complete to a generic $G_{\kappa_2}*F'_{\kappa_2}\in V[G_\kappa*F_\kappa]$. Finally, we wish to modify some values of $F'_{\kappa_2}$ to a generic $F_{\kappa_2}=\l f_{\kappa_2,\gamma}\mid\gamma<\kappa_2^+\r$ so that for every $\alpha<\kappa_1^+$:
\begin{enumerate}
    \item $f_{\kappa_2,k(\alpha)}\restriction \kappa_1=f_{\kappa_1,\alpha}$.
    \item For $\alpha\in j_1''\kappa^+$, $f_{\kappa_2,k(\alpha)}(\kappa_1)=1$.
    \item For $\alpha\in \kappa_1^+\setminus j_1''\kappa^+$, $f_{\kappa_2,k(\alpha)}(\kappa_1)=0$
    \item $f_{\kappa_2,\kappa_1}(\kappa_1)=\kappa$.
\end{enumerate}
Again, this is possible since we do not change too many values of $F'_{\kappa_2}$. At this point, let us emphasize that we do not use $H_{\kappa_1}$ in the generic we have in the $M_U$-side \footnote{Since over $V$, at $\kappa$ we forced one copy of Cohen's i.e. $\Add(\kappa,\kappa^+)$, over $M_U$ we need to force only one copy of $\Add(\kappa_1,\kappa_1^+)$, thus we only need the generic $F_{\kappa_1}$.}. The generic $H_{\kappa_1}$ is used in the construction of the generic on the $M_{U^2}$-side where we can choose (due to the lottery sum) to force at $\kappa_1$ two copies of $\Add(\kappa_1,\kappa_1^+)$, of course, that at $\kappa_2=j_{2}(\kappa)$ we are still obligated to force one copy of $\Add(\kappa_2\kappa_2^+)$ which contains the point-wise image of $F_{\kappa_1}$ under the factor map $k$.  

Hence we extended in $V[G_{\kappa}*F_\kappa]$, $k\subseteq k^*:M_U[G_{\kappa_1}*F_{\kappa_1}]\rightarrow M_{U^2}[G_{\kappa_2}*F_{\kappa_2}]$. 
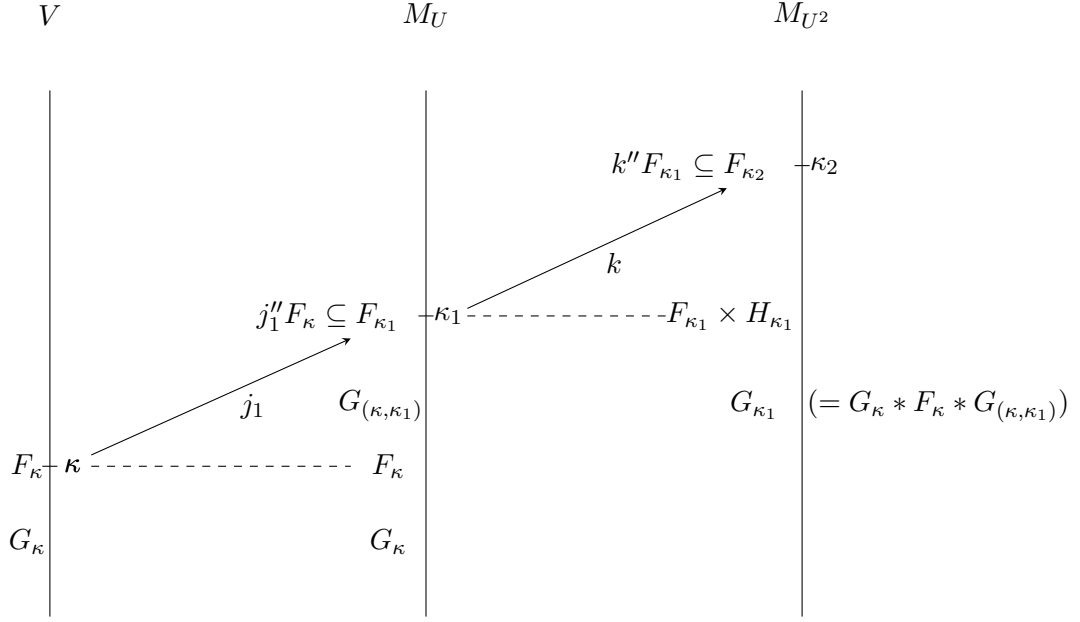
\begin{center}
\begin{tikzpicture}
 \draw (0,0)--(0,7);
 \draw (-0.1,2)--(0.1,2);
 \draw [-stealth] (0.55,2.15)--(4,3.7);
 \draw [dashed] (0.55,2)--(4,2);
 \draw (5,0)--(5,7);
 \draw (4.9,4)--(5.1,4);
 \draw [-stealth] (5.55,4.1)--(9,5.7);\draw [dashed] (5.55,4)--(8.2,4);

 \draw (10,0)--(10,7);
 \draw (9.9,6)--(10.1,6);
\node at (0,8){$V$};
\node at (5,8){$M_U$};
\node at (10,8){$M_{U^2}$};
    \node at (0.3,2){$\kappa$};
    
    \node at (-0.3,1){$G_\kappa$};
    \node at (4.5,1){$G_\kappa$};
    \node at (-0.3,2){$F_\kappa$};
    \node at (4.4,2.8){$G_{(\kappa,\kappa_1)}$};
    \node at (11.3,2.8){$G_{\kappa_1} \ \ \ (=G_\kappa*F_\kappa*G_{(\kappa,\kappa_1)})$};
    \node at (9.05,4){$F_{\kappa_1}\times H_{\kappa_1}$};
    
    \node at (4.5,2){$F_\kappa$};
    
    \node at (5.3,4){$\kappa_1$};
    
    \node at (3.7,4){$j_1''F_\kappa\subseteq F_{\kappa_1}$};
    \node at (10.3,6){$\kappa_2$};
    \node at (8.5,6){$k''F_{\kappa_1}\subseteq F_{\kappa_2}$};
    \node at (0.3,2){$\kappa$};
    \node at (2.7,2.8){$j_1$};
    \node at (7.5,4.7){$k$};
    \end{tikzpicture}
    \end{center}

Let $j_2^*=k^*\circ j_1^*$, $V^*=V[G_{\kappa}*F_\kappa]$, $M_1^*=M_U[G_{\kappa_1}*F_{\kappa_1}]$ and $M_2^*=M_{U^2}[G_{\kappa_2}*F_{\kappa_2}]$. 

In $V^*$, define 
$$U^*=\{X\subseteq\kappa\mid \kappa\in j_2^*(U)\}$$
$$W=\{X\subseteq \kappa\mid \kappa_1\in j^*_2(X)\}$$
and for every $\alpha<\kappa^+$:
$$A_\alpha=\{\nu<\kappa\mid f_{\kappa,\alpha}(\nu)=1\}.$$
Then as in Claim \ref{final claim}, we have that
$W$ is a $\kappa$-complete ultrafilter over $\kappa$ such that:
\begin{enumerate}\item $j_1^*=j_U^*$, $j^*_2=j_W$ and $[id]_W=\kappa_1$.
    \item $\l A_\alpha\mid\alpha<\kappa^+\r$ is a strong witness for $W$ being non-Galvin.
    \item $Cub_\kappa\subseteq W$.
    \item $L_0=\{\alpha<\kappa\mid \Add(\alpha,\alpha^+)\times \Add(\alpha,\alpha^+)\text{ was forced in }G_\kappa\}\in W$.
\end{enumerate}
Also, recall that $j_2:V\rightarrow M_2$ is also the ultrapower by $U\times U$ under the identification(isomorphism): $$j_{U^2}(f)(\kappa,\kappa_1)=j_{2,1}(j_1(\nu\mapsto f(\nu,*))(\kappa))(\kappa_1).$$
Clearly, the projections $\pi_{1},\pi_2:\kappa\times\kappa\rightarrow \kappa$ on the first and second coordinates resp. Rudin-Keisler project $U^2$ on $U$. Also, $W\cap V=U^*\cap V=U$ and $U^*\leq_{R-K} W$ and the projection map is denoted by $\nu\mapsto \pi_{nor}(\nu)$. \footnote{Explicitly, one can define in $V[G]$ the function $f(\alpha)=f_{\kappa,\alpha}(\alpha)$. Then $j^*_2(f)(\kappa_1)=f_{\kappa_2,\kappa_1}(\kappa_1)=\kappa$.}.  

Let us prove that $W$ witnesses the theorem:
\begin{theorem}\label{Proof for W is Cohen}
Let  $H\subseteq \Pri(W)$ be a $V^*$-generic filter. There is $G^*\in V^*[H]$ which is $V^*$-generic for $Cohen(\kappa,\kappa^+)^{V^*}$.  
\end{theorem}
\begin{proof}[ Proof of Theorem \ref{Proof for W is Cohen}:]
Let $\l c_n \mid n<\omega \r $ be the $W$-Prikry sequence corresponding to $H$. Suppose without loss of generality that for every $n<\omega$, $c_n\in L_0$, this will hold from a certain point and the proof can be adjusted in a straightforward way. This guarantees that the generic  $H_{c_n}=\l h_{c_n,\gamma}\mid \gamma<\alpha^+\r$ for the second component of the generic we have in $G_\kappa$ for $\Add(c_n,c_n^+)\times \Add(c_n,c_n^+)$ is defined  for every $n<\omega$. The functions $h_{c_n,\gamma}$ will be used below to define the Cohen generic functions.

Define, for every $n<\omega$, the set
$$Z_n=\{\alpha<\kappa^+\mid \{c_m \mid n\leq m<\omega\} \subseteq A_\alpha \text{ and } n \text{ is least possible } \}.$$
For every $\alpha<\kappa^+$, let $n_\alpha$ be the unique $n$ such that $\alpha\in Z_n$.
Let $\alpha<\kappa^+$, define $f^*_\alpha:\kappa\to \kappa$ as follows:

Fix a sequence $\l s_\alpha \mid \alpha<\kappa^+ \r\in V^*$ of canonical functions in $\prod_{\nu<\kappa}\nu^+$.

$$f^*_\alpha\restriction c_{n_\alpha}= h_{c_{n_\alpha} s_\alpha(c_{n_\alpha})},$$

$$f^*_\alpha\restriction [c_{m-1}, c_m)= h_{ c_m, s_\alpha(c_m)}\restriction [c_{m-1}, c_m), \text{for }m, n_\alpha< m<\omega.$$

Let us argue that $F=\l f^*_\alpha \mid \alpha<\kappa^+\r $ induces a $\Add(\kappa,\kappa^+)^{V^*}$ generic filter over $V^*$.

\begin{claim}
Let $G^*=\{p\in \Add(\kappa,\kappa^+)^{V^*}\mid p\subseteq F\}$, then $G^*$ is a $V^*$-generic filter.
\end{claim}
Let $\mathcal{A}\in V^*$ be a maximal antichain in the forcing $\Add(\kappa, \kappa^+)^{V^*}$. Note that since $\Add(\kappa,\kappa^{+})^{V^*}$ is $\kappa-$closed then $$Cohen(\kappa,\kappa^{+})^{V[G_\kappa]}=Cohen(\kappa,\kappa^{+})^{V^*}.$$ 
By $\kappa^+-$cc of the forcing $\mathcal{P}_{\kappa+1}$, there is $Y\subseteq\kappa^{+}$, $Y\in V$ such that $|Y|=\kappa$ and $\mathcal{A}\subseteq \Add(\kappa,Y)^{V^*}$. Also, since $|\mathcal{A}|=\kappa$, $\mathcal{A}\in V[G_\kappa*F_\kappa]$, there is $Z\subseteq\kappa^{+}$ such that $|Z|=\kappa$ such that $\mathcal{A}\in V[G_\kappa*F_\kappa\restriction Z]$. Without loss of generality assume that $Z=Y\in V$ (Otherwise just take the union). Let $V\ni\phi:\kappa\rightarrow Y$ be a bijection.

\begin{claim}\label{elementary chain construction}
There is an $\in-$increasing continuous chain $\l N_\beta\mid \beta<\kappa\r$ of elementary submodels of $H_\chi$ such that

\begin{enumerate}
  \item $|N_\beta|<\kappa$,
  \item $G_\kappa,F_\kappa,\mathcal{A},\phi,\l s_\alpha\mid \alpha<\kappa^+\r\in N_0$,
  \item $N_\beta\cap \kappa=\gamma_\beta$ is a cardinal $<\kappa$, $\gamma_{\beta+1}$ is regular.
  \item for every $\rho,\delta\in\phi''\gamma_\beta.\rho<\delta\rightarrow\forall\gamma_\beta\leq \mu<\kappa , s_\rho(\mu)<s_\delta(\mu)$.
  \item If $\gamma_\beta$ is regular, then $N_\beta^{<\gamma_\beta}\subseteq N_\beta$. In particular $\Add(\gamma_\beta,\phi''\gamma_\beta)=\Add(\kappa,Y)\cap N_\beta$.
\end{enumerate}
\end{claim}
\textit{Proof of Claim \ref{elementary chain construction}}
Let us construct such a sequence inductively. Note that $(4)$ follows from elementarity and $(2)$. Requirements $(1)-(5)$ are preserved at limit stages due to continuity. At successor stages, suppose we have constructed $N_\beta$,  find an elementary submodel $N^0_{\beta+1}$ such that $N_\beta\subseteq N^0_{\beta+1}, \ \l N_\alpha\mid\alpha<\beta\r\in N^0_{\beta+1}$, then we construct an auxiliary $\in$-increasing and continuous chain of elementary submodels $\l N^\alpha_{\beta+1}\mid \alpha<\kappa\r$ as follows: $N^0_{\beta+1}$ is already defined. At limits we take the union and at successor let us take care of requirements 3,5. Let $\gamma'_\alpha=\sup(N^{\alpha}_{\beta+1}\cap\kappa)<\kappa$. Let $N^{\alpha+1}_{\beta+1}$ be an elementary submodel such that $N^{\alpha}_\beta,^{<\gamma'_\alpha},\subseteq N^{\alpha+1}_{\beta+1}$ and $|N^{\alpha+1}_{\beta+1}|<\kappa$. Note that the sets $$C_1=\{\alpha<\kappa\mid N^{\alpha}_{\beta+1}\cap\kappa=\gamma'_\alpha\in\kappa\}$$ 
$$C_2=\{\alpha\in C_1\mid \text{if }\gamma_\alpha\text{ is regular then } N_{\alpha}^{<\gamma_\alpha}\subseteq N_\alpha\}$$
are clubs and also $\bar C=C_1\cap C_2$ is. It follows that $\{\gamma'_\alpha\mid \alpha\in \bar C\}$ is a club and since $\kappa$ is measurable, there is a $\alpha^*\in \bar C$ limit such that $\gamma'_{\alpha^*}$ is regular. Let $N_{\beta+1}=N^{\alpha^*}_{\beta+1}$, to conclude $2$ since
$\gamma_{\beta+1}=\gamma'_{\alpha^*}$ is regular. $\qed_{\text{Claim }\ref{elementary chain construction}}$
\vskip 0.2 cm

Set
$$C=\{ \beta<\kappa \mid \gamma_\beta=\beta \}.$$
This is club in $\kappa$ since the sequence $\gamma_\beta$ is continuous and since the set $\{\beta\mid \gamma_\beta=\beta\}$ is a club.  

\begin{claim}\label{small claim}
Let $$E:=\{\beta<\kappa\mid\forall\gamma\in\phi''\beta.\exists \delta<\beta^+.f_{\kappa,\gamma}\restriction\beta=f_{\beta,\delta}\}.$$
Then $E\in W$.
\end{claim}
\textit{Proof of claim \ref{small claim}.} By construction, for every $\alpha<\kappa_1^{+}$, $f_{\kappa_2,k(\alpha)}\restriction\kappa_1=f_{\kappa_1,\alpha}$ and therefore for every for every $\alpha\in j^*_2(\phi)''\kappa_1$, there is $\nu<\kappa_1^+$ such that $\alpha=k^*(j_1^*(\phi))(\nu)=k^*(j_1^*(\phi)(\nu))$ and $j_1^*(\phi)(\nu)<\kappa_1^{+}$. Hence $f_{\kappa_2,\alpha}\restriction\kappa_1=f_{\kappa_1,\beta}$ for some $\beta<\kappa_1^{+}$.
Reflecting this we obtain the set $E\in W$. $\qed_{\text{Claim }\ref{small claim}}$

\vskip 0.2 cm

To see that $G^*\cap \mathcal{A}\neq\emptyset$,  we will need to catch a piece of $\mathcal{A}$ in the elementary submodels constructed and pick the Prikry points in the club $C$ prepared:

\begin{claim}\label{condition in model}
For every $\nu_0\in  C\cap E$, there is $d=d^{\nu_0}\in N_{\nu_0}\cap \mathcal{A}$ such that $d$ is extended by $\l h_{\nu_0,s_{\tau}(\nu_0)}\mid \tau\in \phi''\nu_0\r$.
\end{claim}
\textit{Proof of Claim \ref{condition in model}.} Fix any $\nu_0\in C\cap E$. Consider the transitive collapse of $\pi:N_{\nu_0}\rightarrow N_{\nu_0}^*$.
Then the critical point of $\pi^{-1}:N_{\nu_0}^*\rightarrow N_{\nu_0}$ is $\nu_0$ and $\pi^{-1}(\nu_0)=\kappa$. 
Denote by $\overline{F}_\kappa=\pi(F_\kappa), \overline{\phi}=\pi(\phi)$. Denote $\overline{F}_\kappa=\l \overline{f}_{\kappa,\gamma}\mid \gamma< \pi(\kappa^+)\r$.  For every $\gamma\in\overline{\phi}''{\nu_0}$, there is some $\delta<{\nu_0}$ such that, $$\gamma=\pi(\phi)(\delta)=\pi(\phi(\delta))\text{ and  }\overline{f}_{\kappa,\gamma}=\pi(f_{\kappa,\phi(\delta)}).$$ Moreover, since ${\nu_0}\in E$, $f_{\kappa,\phi(\delta)}\restriction{\nu_0}=f_{{\nu_0},\rho}$ for some $\rho<{\nu_0}^{+}$ and therefore $\overline{f}_{\kappa,\gamma}=f_{{\nu_0},\rho}$. 
Recall that $\mathcal{A}=(\lusim{A})_{G_\kappa*F_\kappa\restriction Y}$ hence $\overline{\mathcal{A}}=(\lusim{A})_{G_{\nu_0}*\overline{F}_\kappa\restriction \overline{Y}}$. We conclude that for some subset $Z\subseteq{\nu_0}^{+}$, $$\overline{\mathcal{A}}=(\lusim{A})_{G_{\nu_0}*F_{\nu_0}\restriction Z}\in V[G_{\nu_0}*F_{\nu_0}\restriction Z].$$ Since ${\nu_0}\in L_0$, in $V[G_{\kappa}*F_\kappa]$ we also have $H_{\nu_0}=\l h_{\nu_0,\alpha}\mid \alpha<\nu_0^+\r$ which are mutually Cohen-generic  over $V[G_{\nu_0}*F_{\nu_0}\restriction Z]$.

 By construction, $\forall\tau_1<\tau_2\in\phi''\nu_0$, $s_{\tau_1}(\nu_0)<s_{\tau_2}(\nu_0)$, hence $\l h_{\nu_0,s_\tau(\nu_0)}\mid \tau\in \phi''\nu_0\r$ are  Cohen functions over $\nu_0$ which are
 distinct mutually $V[G_{\nu_0}*F_{\nu_0}\restriction Z]$-generic. Also,  $\overline{\mathcal{A}}\subseteq \pi(\Add(\kappa,Y))= \Add(\nu_0,\pi(\phi)''\nu_0)=\Add(\nu_0,\pi''[\phi''\nu_0])$ is a maximal antichain. Since $|\pi''\phi''\nu_0|=\nu_0=|\phi''\nu_0|$, we can change the enumeration of the functions $\l h_{\nu_0,s_{\tau}(\nu_0)}\mid \tau\in \phi''\nu_0\r$ to $h'_{\pi(\tau)}=h_{\nu_0,s_{\tau}(\nu_0)}$ so that $\l h'_{\rho}\mid \rho\in\pi''\phi''\nu_0\r$ is generic for $\Add(\nu_0,\pi''\phi_0)$. Thus pick $d_0 \in \overline{\mathcal{A}}$ such that $d_0$ is extended by $\l h'_{\rho}\mid \rho\in\pi''\phi''\nu_0\r$. It follows that $$d:=\pi^{-1}(d_0)\in \mathcal{A}\cap N_{\nu_0}$$
is a condition with $\dom(d)=\pi^{-1}(\dom(d_0))$. Since the critical point of $\pi$ is $\nu_0$, for every $\l\alpha,\beta\r\in \dom(d_0)$, $\pi^{-1}(\l\alpha,\beta\r))=\l\alpha,\pi^{-1}(\beta)\r$, hence $$d(\l\alpha,\pi^{-1}(\beta)\r)=\pi^{-1}(d_0(\alpha,\beta))=d_0(\alpha,\beta).$$
In particular for every $\l\gamma,\alpha\r\in \dom(d)$, $$d(\gamma,\alpha)=d_0(\gamma,\pi(\alpha))=h'_{\pi(\alpha)}(\gamma)=h_{\nu_0,s_{\alpha}(\nu_0)}(\gamma).$$  Thus $d$ is extended by $\l h_{\nu_0,s_\tau(\nu_0)}\mid\tau\in \phi''\nu_0 \r$.
$\qed_{\text{Claim }\ref{condition in model}}$
\vskip 0.2 cm

 It suffices to show that any condition in $\Pri(W)$ has an extension which forces that $G^*$ meets a member of $\mathcal{A}$. 
 
Let $p=\l \l\r, B\r$ be a condition (we assume for simplicity that its finite sequence is empty) and shrink $B$ to $B\cap C\cap E$.  For any $\nu_0\in B\cap C\cap E$, we split $\phi''\nu_0$ into two sets: $$X^{\nu_0}_0:=\{\tau\in\phi''\nu_0\mid \nu_0\in A_\tau\}\text{ and }X^{\nu_0}_1=\phi''\nu_0\setminus X^{\nu_0}_0.$$
The condition  $p_0=\l \l \nu_0\r, B\cap C\cap E\cap X\cap (\bigcap_{\tau\in\phi''\nu_0}A_\tau)\r$, forces the following:
\begin{enumerate}
  \item the Prikry sequence is included in each $A_\tau$, $\tau\in X^{\nu_0}_0$, i.e., $n_\tau=0$,

  \item $n_\tau=1$, for every $\tau\in X^{\nu_0}_1$.
\end{enumerate}
In particular, this condition forces some information about the Cohen functions. Namely that:
\begin{enumerate}
    \item for $\tau\in X_0^{\nu_0}$,  $f^*_{\tau}\restriction \nu_0=h_{\nu_0,s_\tau(\nu_0)}$ 
    \item for $\tau\in X_1^{\nu_0}$, $f^*_{\tau}\restriction \nu_0=h_{\lusim{c}_1,s_{\tau}(\lusim{c}_1)}\restriction \nu_0$.
\end{enumerate}

We would like to find a condition in $\mathcal{A}$ which is below these decided parts of the Cohen. By the previous proposition, there is $d\in N_{\nu_0}\cap \Add(\kappa,Y)=\Add(\nu_0,\phi''\nu_0)$, which is extended by $\l h_{\nu_0,s_\tau(\nu_0)}\mid \tau\in \phi''\nu_0\r$. However, by $(1)-(2)$ we can only ensure that the generic $f^*_\tau$   to extend $d\restriction \nu_0\times X^{\nu_0}_0$ in $X^{\nu_0}_0$.  We are left to extend  $d\restriction\nu_0\times X^{\nu_0}_1$. Let us show that for many $\nu_0$, $X^{\nu}_0$ is a relatively large subset of $\phi''\nu_0$:
\begin{claim}\label{property of R}
  Let 
  $$R=\{\nu<\kappa\mid \forall\alpha\in\phi''\pi_{nor}(\nu), \nu\in A_\alpha\}.$$ 
  Then $R\in W$.
\end{claim}
\begin{proof}
Clearly, for every $\alpha\in j^*_2(\phi)''\kappa$, $\alpha=j^*_2(\phi(\gamma))$ and $f_{\kappa_2,\alpha}(\kappa_1)=1$, reflecting this, we can find a $W$-large set of $\nu$'s such that for every $\alpha\in \phi''\pi_{nor}(\nu)$, $f_{\kappa,\alpha}(\nu)=1$. And by definition of $A_\alpha$, $\nu\in A_\alpha$.
\end{proof}
Denote $B_0:=B\cap C\cap E\cap R$.
In order to extend $d\restriction \nu_0\times X_1$, we will need to pick $\nu_0$ high enough in $B_0$, but also the next point $\nu_1\in  B_0\setminus \nu_0+1$ in the Prikry sequence such that it will belong to all $A_\tau$ with $\tau\in X_1$ and in addition the relevant Cohen functions over $\nu_1$ extend $d\restriction\nu_0\times X_1$. 

Let us look at $B_0$ more carefully.
Let $\lusim{B_0}$ be its name in $V$. We fix a condition $m_0\in G_\kappa*F_\kappa$ which forces that if $\nu_0\in\lusim{B_0}$ then the properties of Claims \ref{condition in model}, \ref{property of R} holds, namely there is $d\in \Add(\nu_0,\phi''\nu_0)\cap\lusim{\mathcal{A}}$ which is extended by $\l \lusim{h}_{\nu_0,s_{\tau}(\nu_0)}\mid \nu_0\in\phi''\nu_0\r$, and $\forall \alpha\in \phi''\lusim{\pi_{nor}(\nu_0)}.\ \nu_0\in \lusim{A}_\alpha$. Recall that by the construction of $G_{\kappa_2}$, we have $m_0\in G_{\kappa_2}*F_{\kappa_2}$, 
Let $ m_0\leq t\in G_{\kappa_2}*F_{\kappa_2}$ be a condition such that
$$(1) \ \ \ t\Vdash\kappa_1\in j_2(\lusim{B_0}).$$
By the construction of $G_{\kappa_2}*F_{\kappa_2}$, $t$ has the form:
$$t=\l t_{<\kappa}, t_{\kappa},t_{(\kappa,\kappa_1)},\underset{t_{\kappa_1}}{\underbrace{\l t_{\kappa_1}^0,t_{\kappa_1}^1\r}},t_{(\kappa_1,\kappa_2)},t_{\kappa_2}\r.$$ 
Since $f_{\kappa_2,j_2(\alpha)}(\kappa_1)=1$ for every $\alpha<\kappa^+$, this will hold for every $\alpha\in\phi''\kappa$ as well. Also, recall that $Y\in V$, hence $\phi\in V$. Thus $j_2(\phi)\in M_2$ and  $j_2(\phi)''\kappa\in M_2$. Also, for $(t_{\kappa_2})_{G_{\kappa_2}}\in M_2[G_{\kappa_2}]$, $$j_2''\kappa^+\cap \supp((t_{\kappa_2})_{G_{\kappa_2}})\in M_2[G_{\kappa_2}]$$ and $(t_{\kappa_2})_{G_{\kappa_2}}\restriction \kappa\times \{j_2(\alpha)\}\subseteq f_{\kappa,\alpha}$. We also fix $\theta<\kappa^+$ such that $\supp((t_{\kappa_2})_{G_{\kappa_2}})\subseteq j_2(\theta)$, there is such $\theta$ since $j_2''\kappa^+$ is unbounded in $j_2(\kappa^+)$.
Therefore, we can extend if necessary $t$ such that
$$ (2a)\ \ t_{<\kappa_2}\Vdash (\kappa\cup\{\kappa_1\})\times j_2(\phi)''\kappa \subseteq \dom(t_{\kappa_2})\wedge (0,\kappa_1)\in\dom(t_{\kappa_2})\wedge \supp(t_{\kappa_2})\subseteq j_2(\theta)$$
$$(2b) \ \ t_{<\kappa_2}\Vdash \  t_{\kappa_2}(\kappa_1,\alpha)=1,\text{ for every }\alpha\in j_2(\phi)''\kappa\text{ and } t_{\kappa_2,\kappa_1}(0)=\kappa.$$
$$(2c) \ \  t_{<\kappa_2}\Vdash t_{\kappa_2,j_2(\alpha)}\restriction \kappa=\lusim{f}_{\kappa,\alpha}\text{  for every } j_2(\alpha)\in j_2''\kappa^+\cap\supp(t_{\kappa_2}).$$

Next consider $t_{\kappa_1}=\l t_{\kappa_1}^0,t_{\kappa_1}^1\r$, it is a $\calP_{\kappa_1}-$name for a condition in $F_{\kappa_1}\times H_{\kappa_1}$. By the construction of the generic $F_{\kappa_1}\times H_{\kappa_1}$, for every $\alpha<\kappa^+$, we made sure that $h_{\kappa_1,j_1(\alpha)}\restriction\kappa=f_{\kappa,\alpha}$. Also, $j_1(\phi)''\kappa\in M_2$. Let $$\mu_1=(j_1\restriction \phi''\kappa)^{-1}\in M_1$$ Note that for every $\beta<\kappa^+$, $j_1(s_\beta)=s_{j_1(\beta)}:\kappa_1\rightarrow\kappa_1$ is the canonical function for $j_1(\beta)$ defined in $M_U$, hence $j_2(s_\beta)(\kappa_1)=k(s_{j_1(\beta)})(\kappa_1)=j_1(\beta)$. Hence $$\dom(\mu_1)=j_1(\phi)''\kappa=\{s_\gamma(\kappa_1)\mid \gamma\in j_2(\phi)''\kappa\}, \ rg(\mu_1)=\phi''\kappa\subseteq\kappa^+$$
Extend if necessary $t_{<\kappa_1}$, and assume that $$(3) \ \ t_{<\kappa_1}\Vdash\kappa\times j_1(\phi)''\kappa\subseteq\dom(t^1_{\kappa_1})\wedge \forall j_1(\alpha)\in j_1(\phi)''\kappa, \  t_{\kappa_1,j_1(\alpha)}^1\restriction\kappa= \lusim{f}_{\kappa,\alpha}.$$
 As for the lower part, due to the Easton support, we have $$(4) \ \ t_{<\kappa}\in V_\kappa.$$

 Fix functions $r,\Gamma_1$ which represents $t,\mu$ resp. in the ultrapower $M_{U^2}$, namely $j_2(r)(\kappa,\kappa_1)=t, \ j_2(\Gamma_1)(\kappa,\kappa_1)=\mu$. Without loss of generality, suppose that for every $(\nu',\nu)$, it takes the form $$r(\nu',\nu)=\l r_{<\nu'} r_{\nu'},r_{(\nu',\nu)},\l r_{\nu}^0,r_{\nu}^1\r,r_{(\nu,\kappa)},r_\kappa\r.$$

Reflecting some of the properties of $t$ we obtain a set $B'\in U^2$ such that for every $(\nu',\nu)\in B'$:
\begin{enumerate}
    \item [$(1)_{(\nu',\nu)}$] $r(\nu',\nu)\Vdash \nu\in \lusim{B_0}$.
    \item [$(2a)_{(\nu',\nu)}$] $ r_{<\kappa}\Vdash (\nu'\cup\{\nu\})\times \phi''\nu' \subseteq \dom(r_{\kappa})\wedge \l0,\nu\r\in\dom(r_{\kappa})\wedge \supp(r_{\kappa})\subseteq \theta$

    \item [$(2b)_{(\nu',\nu)}$] $r_{<\kappa}\Vdash  \forall\alpha\in\phi''\nu'.r_{\kappa,\alpha}(\nu)=1$ and $r_{\kappa,\nu}(0)=\nu'$. 
   
    \item  [$(3)_{(\nu',\nu)}$] $ r_{<\nu}\Vdash \nu'\times\dom(\Gamma_1(\nu',\nu))\subseteq\dom(r^1_\nu)$ and for every
    
    $\beta\in\dom(\Gamma_1(\nu',\nu)),  r^1_{\nu,\beta}\restriction\nu'=\lusim{f}_{\nu',\Gamma_1(\nu',\nu)(\beta)}$.
    
    \item [$(4)_{(\nu',\nu)}$] $r_{<\nu'}=t_{<\kappa}\in V_{\nu'}$.
    
\end{enumerate}
\vskip 0.2 cm
Let $$B''=\{\nu \mid\exists(\nu',\nu)\in B'. r(\nu',\nu)\in G_\kappa*F_\kappa\}.$$
Since $B'\in U^2$ we have that $(\kappa,\kappa_1)\in j_2(B')$ and since $j_2(r)(\kappa,\kappa_1)=t\in j^*_2(G_\kappa*F_\kappa)=G_{\kappa_2}*F_{\kappa_2}$, we conclude that $B''\in W$. Also, $B''\subseteq B_0$ by clause $(1)$. 

We proceed by a density argument, recalling that by the definition of $G_2$, we have that $\l t_{<\kappa},t_\kappa\r\in G_\kappa*F_\kappa$. 
\begin{claim}
Let $D$ be the set of all conditions $q\in\mathcal{P}_{\kappa+1}$, such that exists $(\nu'_0,\nu_0),(\nu'_1,\nu_1)\in B', \ \nu'_1>\nu_0$ and a $\mathcal{P}_{\nu_0}-$name $\lusim{d}^{\nu_0}$ such that \begin{enumerate}
    \item [(a)] $r(\nu'_0,\nu_0),r(\nu'_1,\nu_1)\leq q$.
    \item [(b)] $q\Vdash \lusim{d}^{\nu_0}\in \lusim{A}\cap\Add(\nu_0,\phi''\nu_0)$.
    \item [(c)] $q\Vdash \forall \tau\in X_1^{\nu_0}.\lusim{h}_{\nu_1,s_\tau(\nu_1)}\restriction \nu_0=\lusim{d}^{\nu_0}_\tau.$
    
\end{enumerate}
Then $D$ is dense (open) above $\l t_{<\kappa},t_\kappa\r$ and thus $D\cap G_\kappa*F_\kappa\neq\emptyset$
\end{claim}
\begin{proof}
Work in $V$, let $\l t_{<\kappa},t_\kappa\r\leq p:=\l p_{<\kappa},p_\kappa\r \in \calP_{\kappa+1}$.  We will define two extensions $p\leq q\leq q^*$ which corresponds to the choice of $(\nu'_0,\nu_0),(\nu'_1,\nu_1)$ and such that $q^*\in D$. By definition of $\mathcal{P}_{\kappa+1}$, $p_{<\kappa}\Vdash p_\kappa\in \Add(\kappa,\kappa^+)$, by $\kappa-$cc of $\mathcal{P}_\kappa$, for some $Z\subseteq\kappa^+,\ Z\in V$, $|Z|<\kappa$ and some $\gamma<\kappa$, $p_{<\kappa}\Vdash \dom(p_\kappa)\subseteq \gamma\times Z$. Applying $j_2$, we have that $$j_2(p_{<\kappa})=p_{<\kappa}\Vdash\dom(j_2(p_\kappa))\subseteq j_2(\gamma\times Z)= \gamma\times j_2''Z\text{ and } j_2(p_{\kappa})_{j_2(\alpha)}=p_{\kappa,\alpha}\geq t_{\kappa,\alpha}$$ Combining with $(2c)$, we have both
$$p_{<\kappa}\Vdash Z\supseteq \supp(t_\kappa)\wedge \forall \beta\in Z.j_2(p_{\kappa})_{j_2(\beta)}\geq t_{\kappa,\beta}$$  $$t_{<\kappa_2}\Vdash \forall j_2(\tau)\in \supp(t_{\kappa_2})\cap j_2(Z). t_{\kappa_2,j_2(\tau)}\restriction \gamma= \lusim{f}_{\kappa,\tau}\restriction \gamma.$$
To reflect this, denote $\mu=(j_2\restriction (Z\cup \theta))^{-1}\in M_2$, then
$$\dom(\mu)=j_2(Z)\cup j_2''\theta, \ \rng(\mu)=Z\cup\theta, \ \mu\text{ is 1-1}.$$
and we can reformulate:
$$p_{<\kappa}\Vdash \mu''j_2(Z)\supseteq \supp(t_\kappa)\wedge \forall \beta\in j_2(Z).j_2(p_{\kappa})_{\beta}\geq t_{\kappa,\mu(\beta)}$$  $$t_{<\kappa_2}\Vdash \forall \tau\in \supp(t_{\kappa_2})\cap j_2(Z). t_{\kappa_2,\tau}\restriction \gamma= \lusim{f}_{\kappa,\mu(\tau)}\restriction \gamma.$$
Also, since we can find $\delta<\kappa$ such that $t_{<\kappa}\Vdash \phi''(\delta,\kappa)\cap  Z=\emptyset.$ There exists such $\delta$ since $|Z|<\kappa$, $t_{<\kappa}\Vdash |\supp(t_{\kappa})|<\kappa$ and by $\kappa-$cc of $\mathcal{P}_\kappa$.
 Recall that by the definition of $\mu_1$, $\phi''(\delta,\kappa)=\mu_1''\{s_\gamma(\kappa_1)\mid \gamma\in j_2(\phi)''(\delta,\kappa)\}$ and that $\mu''\supp(j_2(p_\kappa))=Z$. Therefore in $M_2$ we will have that
 $$p_{<\kappa}\Vdash [\mu_1''\{s_\gamma(\kappa_1)\mid \gamma\in j_2(\phi)''(\delta,\kappa)\}]\cap[ \mu''\supp(j_2(p_\kappa))]=\emptyset.$$ 
Let $\Gamma$ be such that $j_2(\Gamma)(\kappa,\kappa_1)=\mu$, there is a set $\overline{B}_0\subseteq B'$, $\overline{B}_0\in U^2$ such that for every $(\nu',\nu)\in\overline{B}_0$,
$$(i) \ \ p_{<\kappa}\Vdash \Gamma(\nu',\nu)''Z\supseteq \supp(r_{\nu'})\wedge  \forall\beta\in Z. p_{\kappa,\beta}\geq r_{\nu',\Gamma(\nu',\nu)(\beta)},$$ 
$$(ii) \ \ r_{<\kappa}\Vdash \forall \tau\in Z\cap \supp(r_{\kappa}). r_{\kappa,\tau}\restriction \gamma=\lusim{f}_{\nu',\Gamma(\nu',\nu)(\tau)}\restriction \gamma.$$
$$(iii) \ \  p_{<\kappa}\Vdash\Gamma_1(\nu',\nu)''\{s_\gamma(\nu)\mid \gamma\in \phi''(\delta,\nu')\}\cap [\Gamma(\nu',\nu)''\supp(p_\kappa)]=\emptyset.$$

Let us move to the choice of $(\nu'_0,\nu_0),(\nu'_1,\nu_1)$. In $V[G_{\kappa}*F_{\kappa}]$, there exists  $(\nu^0_0,\nu_0),(\nu^0_1,\nu_1)\in \overline{B}_0$ such that $r(\nu^0_0,\nu_0),r(\nu^0_1,\nu_1)\in G_{\kappa}*F_{\kappa}$ (hence they are compatible) such that $\nu^0_0>\delta,\gamma,\sup(\supp(p_{<\kappa}))$ and $\nu^0_1>\nu_0,\supp(r_{<\kappa}(\nu^0_0,\nu_0))$.  
In particular, in $V$ we can find $(\nu '_0,\nu_0),(\nu'_1,\nu_1)\in \overline{B}_0$ such that $r(\nu_0',\nu_0),r(\nu_1',\nu_1)$ are compatible, $\nu'_0>\delta,\gamma,\sup(\supp(p_{<\kappa}))$, and $\nu'_1>\nu_0,\sup(\supp(r_{<\kappa}(\nu_0',\nu_0)))$. Denote  $$r^0:=r(\nu'_0,\nu_0)=\l r^0_{<\nu'_0},r^0_{\nu'_0},r^0_{(\nu'_0,\kappa)},r^0_\kappa\r,$$
$$r^1:=r(\nu'_1,\nu_1)=(r^1_{<\nu'_1},r_{\nu'_1},r_{(\nu'_1,\nu_1)},\l r^{0,1}_{\nu_1},r^{1,1}_{\nu_1}\r,r^1_{(\nu_1,\kappa)},r^1_\kappa\r$$
 Let us define the first extension $q$, it has the form:
$$q=p_{<\kappa}{}^{\smallfrown}q_{\nu'_0}{}^{\smallfrown}r^0_{(\nu'_0,\kappa)}{}^{\smallfrown}q_\kappa$$
First, $q_{\nu'_0}$ is a $\mathcal{P}_{\nu'_0}-$name for a condition with $\supp(q_{\nu'_0})= \Gamma(\nu'_0,\nu_0)''Z$, by $(i)$ $\supp(q_{\nu'_0})\supseteq\supp(r^0_{\nu'_0})$. Set
$q_{\nu'_0,\Gamma(\nu'_0,\nu_0)(\beta)}= p_{\kappa,\beta}$. As for $q_\kappa$, we set it to be a $\mathcal{P}_\kappa-$name for $r^0_\kappa\cup p_\kappa$.

Once we will prove that $p_{<\kappa},r^0_{<\kappa}\leq q_{<\kappa}$, from $(i),(ii)$ it will follow that $q_{<\kappa}$ forces $q_\kappa$ to be a partial function. Indeed, for every $\beta\in \supp(r^0_\kappa)\cap Z$, $q_{<\kappa}$ will force
$$r^0_{\kappa,\beta}\restriction \gamma=\lusim{f}_{\nu',\Gamma(\nu'_0,\nu_0)(\beta)}\restriction \gamma\geq q_{\nu'_0,\Gamma(\nu'_0,\nu_0)(\beta)}=p_{\kappa,\beta}$$

Clearly $p\leq q$. To see that $r^0\leq q$, up to $\nu'_0$, we have that by $(4)_{(\nu'_0,\nu_0)}$ that $$q_{<\nu'_0}=p_{<\kappa}\geq t_{<\kappa}=r^0_{<\nu'_0}.$$ At $\nu'_0$, if $\alpha= \Gamma(\nu'_0,\nu_0)(\beta)$, then $(i)$  insures that $q_{\nu'_0,\alpha}=p_{\kappa,\beta}\geq r^0_{\nu',\alpha}$. Since in the interval $(\nu'_0,\kappa)$, $q$ and $r^0$ are the same, it follows that $q_{<\kappa}\geq r^0_{<\kappa}$ and at $\kappa$ it is clear that $q_{<\kappa}\Vdash r^0_\kappa\leq q_\kappa$. 

Next let us move to the choice of $\lusim{d}^{\nu_0}$. Since $r^0\leq q$ and $m_0\leq \l t_{<\kappa},t_\kappa\r\leq q\Vdash \nu_0\in \lusim{B}_0$, use the maximality principal to find a $\mathcal{P}_{\nu_0}-$name, $\lusim{d}^{\nu_0}$ such that $q$ forces $(b)$.\footnote{Since the tail forcing $\mathcal{P}_{[\nu_0,\kappa]}$ is $\nu_0-$closed, if there is such $d^{\nu_0}\in V[G_{\kappa}*F_\kappa]$ then $|d^{\nu_0}|<\nu_0$, hence $d^{\nu_0}\in V[G_{\nu_0}]$.} 

Define the final condition $q\leq q^*$,  $$q^*=q_{<\kappa}{}^{\smallfrown}q^*_{\nu'_1}{}^{\smallfrown}r^1_{(\nu_1',\kappa)}{}^{\smallfrown}q^*_\kappa.$$ The crucial point here is that by $(2b)_{(\nu'_1,\nu_1)}$ $$r^0_{<\kappa}\Vdash \nu_0^0=\lusim{f}_{\kappa,\nu_0}(0)=r^0_{\kappa,\nu_0}(0)=\nu_0'$$ and since $r^0\Vdash \nu_0\in \lusim{R}$ we have that $r^0\Vdash X_1^{\nu_0}\subseteq \phi''(\nu'_0,\nu_0)\subseteq\phi''(\nu'_0,\nu_1')$. By $(iii)$ we have that  $q_{<\kappa}\Vdash [\Gamma_1(\nu'_1,\nu_1)''\{s_\gamma(\nu_1)\mid \gamma\in X^{\nu_0}_1\}]\cap [\Gamma(\nu_1',\nu_1)''Z]=\emptyset$. This will permit to code $d^{\nu_0}$, let $$\supp(q^*_{\nu'_1})=[\Gamma_1(\nu'_1,\nu_1)''\{s_\gamma(\nu_1)\mid \gamma\in X^{\nu_0}_1\}]\uplus [\Gamma(\nu_1',\nu_1)''Z]$$ and
$$q^*_{\nu'_1,\alpha}=\begin{cases} q_{\kappa,\beta} &\exists \beta\in \Gamma(\nu_1',\nu_1)''Z. \alpha=\Gamma(\nu'_1,\nu_1)(\beta)\\
\lusim{d}^{\nu_0}_\tau & \exists \tau\in X^{\nu_0}_1.\alpha=\Gamma_1(\nu'_1,\nu_1)(s_\tau(\nu_1))\end{cases}$$ 
and $q^*_\kappa=q_\kappa\cup r^1_\kappa$. Note that if $\tau\in \supp(q_{\kappa})\cap\supp(r^1_{\kappa})$ then either $\tau\in \supp(r^0_{\kappa})\cap\supp(r^1_{\kappa})$, and $r^0_{\kappa},r^1_{\kappa}$ are forced to be compatible by $q_{<\kappa}$ and if $\tau\in Z\cap\supp(r^1_{\kappa})$ then the same argument as before works. 
We conclude that $r^0\leq q\leq q^*$, $r^1\leq q^*$, namely $(a)$. Finally, for every $\tau\in X^{\nu_0}_1$, $s_\tau(\nu_1)\in \dom(\Gamma_1(\nu'_1,\nu_1))$ and by $(3)_{(\nu'_1,\nu_1)}$ we have that $q^*$ forces that
$$\lusim{h}_{\nu_1,s_\tau(\nu_1)}\restriction \nu_0=\lusim{f}_{\nu'_1,\Gamma_1(\nu'_1,\nu_1)(s_\tau(\nu_1))}\restriction \nu_0\geq q^*_{\nu'_1,\Gamma_1(\nu'_1,\nu_1)(s_\tau(\nu_1))}=\lusim{d}^{\nu_0}_\tau$$


Then $p\leq q^*$ and $q^*\in D$ 
\end{proof}
By density, we can find such a condition $p^*\in G_\kappa*F_\kappa\cap D$ and points $(\nu'_0,\nu_0),(\nu_1',\nu_1)\in B'$ witnessing $p^*\in D$.
It follows that $r(\nu'_0,\nu_0),r(\nu'_1,\nu_1)\in G_\kappa*F_\kappa$, and by $(1)_{(\nu'_0,\nu_0)},(1)_{(\nu'_1,\nu_1)}$, $\nu_0,\nu_1\in B_0$. Extend $\l \l\r,B\r$ by $p^*=\l \nu_0,\nu_1,B_0\cap(\cap_{\tau\in\phi''\nu_0}A_\tau\r\setminus \nu_1+1\r$. By $(2b)_{(\nu'_1,\nu_1)}$, for every $\tau\in \phi''\nu_0\subseteq\phi''\nu'_1$, $f_{\kappa,\tau}(\nu_1)=r_{\kappa,\tau}(\nu_1)=1$, hence $\nu_1\in \cap_{\tau\in\phi''\nu_0}A_\tau$ and $p^*\Vdash n_\tau=\begin{cases} 0 &\tau\in X_0\\
1 & \tau\in X_1\end{cases}$. In other words,
since $\nu_0\in B_0$,  $$p^*\Vdash \forall\tau\in X_0. \lusim{f}^*\tau\restriction\nu_0=h_{\nu_0,s_\tau(\nu_0)}$$
$$p^*\Vdash \forall\tau\in X_1. \lusim{f}^*\tau\restriction\nu_1=h_{\nu_1,s_\tau(\nu_1)}$$
Let $d=(\lusim{d}^{\nu_0})_{G_{\nu_0}}\in \Add(\nu_0,\phi''\nu_0)\cap \mathcal{A}$, it follows that $p^*\Vdash \forall\tau\in X_0. \lusim{f}^*_\tau$ extends $d_\tau$, and by $(c)$ of the definition of $D$, $p^*\Vdash \forall\tau\in X_1 \lusim{f}^*_\tau$ extends $d_\tau$. Thus $p^*\Vdash d\in \lusim{G}^*\cap\mathcal{A}$. This concludes the genericity proof.
\end{proof}

\end{proof}
\section{The results where $2^\kappa=\kappa^{++}$}
\subsection{Strong  non-Galvin witnesses of length $2^\kappa=\kappa^{++}$}
In this section we produce a model with a non-Galvin ultrafilter with a strong  witnessing sequence of length $2^{\kappa}=\kappa^{++}$. This will of course require to violate \textrm{GCH} on a measurable cardinal and in turn to start with a stronger large cardinal assumption (see \cite{Gitik1989TheNO},\cite{mitchell_1984}).
We will follow a similar construction to the one given in the case of $\kappa^+$ addressed in previous sections. Indeed, instead of iterating $\Add(\alpha,\alpha^+)$ we will iterate $\Add(\alpha,\alpha^{++})$ aiming to force $\Add(\kappa,\kappa^{++})$, from which we will be able to define a non-Galvin ultrafilter and a strong witness of length $\kappa^{++}$ in a similar fashion to the one we have on $\kappa^+$, distinguishing between $\alpha$'s which are in the  image of the  second iteration and those which are in the image of the factor map. The difficulty is, as always, to extend a ground model embedding. By the large cardinal lower bound, we can no longer work with an ultrapower by an ultrafilter. The usual embedding to lift in the context of violation of GCH at measurables is a $(\kappa,\kappa^{++})$-extender ultrapower embedding, which we will use here. This makes the lifting argument more involved and the existence of generic filters for the iteration requires variations of \textit{Woodin's surgery method} (See \cite[Sec. 25]{CummingsHand}).
\begin{theorem}\label{Non-galvin for longer}
Assume GCH and that there is a $(\kappa,\kappa^{++})$-extender over $\kappa$ in $V$. Then there is a cofinality preserving forcing extension $V^*$ such that $V^*\models 2^\kappa=\kappa^{++}$, in $V^*$ there is a $\kappa$-complete ultrafilter $W$ over $\kappa$ which concentrates on regulars, extends $Cub_\kappa$, and has a strong witness of length $\kappa^{++}$ for the failure of Galvin's property.
\end{theorem}

\begin{proof}
Let $E$ be a $(\kappa,\kappa^{++})-$ extender.
Let $j_1=j_E:V\to M_E=:M_1$ be its ultrapower embedding with $crit(j_E)=\kappa$ and ${}^\kappa M_E\subseteq M_E$. Denote by $E_\alpha$ the ultrafilter $$E_\alpha:=\{X\subseteq \kappa\mid \alpha\in j_E(X)\}$$
Denote $U:=E_\kappa$  the normal ultrafilter and let $k:M_U \to M_E$ be the factor map defined by setting $k(j_U(f)(\kappa))=j_E(f)(\kappa)$ such that $j_E=k\circ j_U$.
Define an Easton support iteration
$\l \calP_\alpha, \lusim{Q}_\beta\mid \alpha\leq\kappa+1, \beta\leq\kappa\r$
as follows:
\\$\lusim{Q}_\beta$ is trivial unless $\beta$ is  inaccessible, in which case $Q_{\beta }=\Add(\beta, \beta^{++})$.

Let $G_{\kappa+1}:=G_\kappa*g_\kappa$ be a $V$-generic subset of $\calP_{\kappa+1}=\calP_\kappa*\lusim{Q}_\kappa$. Keeping similar notations to those from previous sections, let $\l f_{\kappa,\alpha} \mid \alpha<\kappa^{++}\r $ be the Cohen generic functions from
$\kappa$ to $2$ introduced by $g_\kappa$.

Now we apply Woodin's argument (see \cite[Section 25]{CummingsHand}, and Ben Shalom \cite{shalom2017woodin} for constructing generics without additional forcing) to see that there will be $G_{j_E(\kappa)+1}*H^*\subseteq j_E(\calP_{\kappa+1})*\mathbb{S}$ in $V^*_1:=V[G_{\kappa+1}][H]$, where $H\subseteq \mathbb{S}$ is a $V[G_{\kappa+1}]$-generic filter, where $\mathbb{S}_0$ is some $\kappa^+$-distributive in $V[G_{\kappa+1}]$ (In the case of Ban-Shalom, there is no need for $H^*$ and we can work directly in $V[G_{\kappa+1}]$) generic over $M_E$ and an elementary embedding
$$j^*_1: V^*_1\to M_E[G_{j_E(\kappa)+1}*f^*]$$ which extends $j_1$.
Recall that the generic filter constructed for $j_1(Q_\kappa)$ is obtained by a surgery argument,  making small changes on an $M_1[G_{j_1(\kappa)}]$-generic filter $f$ to be compatible with $j''_1g_\kappa$. For our purposes, we need some additional changes to be made, for every $p\in f$ we change $p$ to $p^*$ such that $\dom(p^*)=\dom(p)$ and $$p^*(\l\gamma,\alpha\r)=\begin{cases}f_\beta(\gamma) & \gamma<\kappa \wedge \alpha=j_1(\beta)\\ \beta & \gamma=\kappa\wedge \alpha=j_1(\beta)\\
p(\l\gamma,\alpha\r) & else\end{cases}$$ 
To see that $p$ was only changed at $\kappa$-many places, find $a\in[\kappa^{++}]^{<\omega}$ such that $j_E(P)(a)=p$, where $P:\kappa^{|a|}\rightarrow Q_\kappa$. By elementarity, for every $\l \alpha,j_1(\beta)\r\in\kappa\times j_1''\kappa^{++}\cap\dom(p)$, there is $x\in[\kappa]^{|a|}$ such that $\l\alpha,\beta\r\in \dom(P(x))$. It follows that   $|\kappa\times j_1''\kappa^{++}\cap\dom(p)|\leq\kappa$. Moreover,  $|\{\kappa\}\times j''_1\kappa^{++}\cap\dom(p)|\leq\kappa$, since otherwise there would be some $\alpha<\kappa^{++}$ such that $$\cf(\alpha)=\kappa^+\text{ and  }\sup\{j_E(\beta)\mid \l\kappa,j_E(\beta)\r\in\dom(p)\}=j_E(\alpha).$$
But $|\dom(p)|^{M_1}<j_1(\kappa)$ and $\cf^{M_1}(j_1(\alpha))=j_1(\kappa)^+$ which is a contradiction. Hence $p^*\in M_1[G_{j_1(\kappa)}]$ since we have only changed $p$ at $\kappa$-many values and ${}^\kappa M_1[G_{j_1(\kappa)}]\subseteq M_1[G_{j_1(\kappa)}]$. 

The argument that such changes do not affect the genericity is the same as in \cite{CummingsHand}. So we additionally obtain that 
$f_{\kappa_2,\alpha}(\kappa)=\alpha,$ for every $\alpha<\kappa^{++}$.

We also claim that $j^*_1$ is actually the ultrapower embedding by the normal ultrafilter
$$U^*=\{X\subseteq \kappa\mid \kappa\in j^*_1(X)\}$$ extending $U$. To see this, consider  $k^*:M_{U^*}\rightarrow M_1[G_{j_1(\kappa)+1}*H^*]$ defined by $k^*([f]_{U^*})=j^*_1(f)(\kappa)$, which is clearly elementary. To see that $k^*=id$, let us prove that $k^*$ is onto. Fix $A=(\lusim{A})_{G_{j_1(\kappa)+1}*H^*}\in M_1[G_{j_1(\kappa)+1}]$ and let $f\in V$, $a=\{\alpha_1,...,\alpha_r\}\in[\kappa^{++}]^{<\omega}$ be such that $j_1(f)(a)=\lusim{A}$.  Define in $V[G_{\kappa+1}]$ the function $f^*(x)=(f(\{f_{\alpha_1}(x),...,f_{\alpha_r}(x)\}))_{G_{\kappa+1}*H}$. Then
$$k^*(j_{U^*}(f^*)(\kappa))=j^*_1(f^*)(\kappa)=(j_1(f)(\{j_1^*(f_{\alpha_1})(\kappa),....,j^*_1(f_{\alpha_2})(\kappa)\}))_{G_{j_1(\kappa)+1}*H^*}$$
$$=(j_1(f)(a))_{G_{j_1(\kappa)+1}}=(\lusim{A})_{G_{j_1(\kappa)+1}*H^*}=A$$

We would like now to construct a $\kappa-$complete ultrafilter $W\in V[G_{\kappa+1}]$ over $\kappa$ which includes $Cub_\kappa$ and the family
$\l A_\alpha \mid \alpha<\kappa^{++}\r$ which is a strong witness that $W$ fails to satisfy the Galvin Property. Set
$$A_\alpha:=\{\nu<\kappa\mid f_\alpha(\nu) \text{ is odd}\},$$
 for every $\alpha<\kappa^{++}$.

Consider the second ultrapower (of $V$) by $E$, i.e., $\Ult(M_E, j_E(E))$. In order to simplify the notation let us denote $M_E$ by $M_1$ and $\Ult(M_1, j_1(E))$ by $M_2$ and $j_{2,1}:=j_{j_1(E)}:M_1\rightarrow M_2$.
Also, let $\kappa_1=j_1(\kappa), E_1=j_1(E),$ and $\kappa_2=j_{2,1}(\kappa_1)$.
 Let $j_2:V\to M_2$ be the composition of $j_1$ with $j_{2,1}$.

Work in $M_1[G_{\kappa_1+1}*H^*]$ apply there the Woodin argument to $E_1$. There will be $G_{\kappa_2+1}*H^{**}\subseteq j_2(P_{\kappa}*Q_\kappa*\mathbb{S}_0)$ (in $M_1[G_{\kappa_1+1}*H^*]$) generic over $M_2$ and an elementary embedding
$$j_{2,1}^*: M_1[G_{\kappa_1+1}*H^*]\to M_2[G_{\kappa_2+1}*H^{**}]$$ which extends $j_{E_1}$. Additionally, for every $\alpha<(\kappa_1^{++})^{M_1}$ let us arrange the following:

\begin{enumerate}
  \item $f_{\kappa_2,j_{2,1}(\alpha)}(\kappa_1)$ is odd, if $\alpha \in j_E{}''\kappa^{++}$,
  \item $f_{\kappa_2, j_{2,1}(\alpha)}(\kappa_1)$ is an even, if $\alpha \in (\kappa_1^{++})^{M_1}\setminus j_E{}''\kappa^{++}$.
  \item $f_{\kappa_2,\kappa_1}(\kappa_1)=\kappa$.
\end{enumerate}
The point being that this requires only small changes of conditions in $(\Add(\kappa_2, \kappa_2^{++}))^{M_2}$, and so preserves the genericity.
\\Namely, given $p \in  (\Add(\kappa_2, (\kappa_2)^{++}))^{M_2}$, define $p^*$ such that $\dom(p^*)=\dom(p)$ and
$$p^*(\l\gamma,\alpha\r)=\begin{cases} f_{\kappa_1,\beta}(\gamma) & \gamma<\kappa_1\wedge\exists \beta<\kappa_1^{++} \alpha=j_{2,1}(\beta)\\
\beta\cdot 2+1 & \gamma=\kappa_1\wedge \exists\beta\in j_1''\kappa^{++}.\alpha=j_{2,1}(\beta)\\ \beta\cdot 2  & \gamma=\kappa_1\wedge \exists\beta\in\kappa_1^{++}\setminus j_1''\kappa^{++}.j_{2,1}(\beta)=\alpha\\
\kappa & \gamma=\alpha=\kappa_1\\
p(\l\gamma,\alpha\r) & otherwise\end{cases}$$
In $V[G_{\kappa+1}*H]$, $|\supp(p)\cap j_2^{*''}\kappa^{++}|\leq \kappa$ and $M_1[G_{\kappa_1+1}*H^*]$ is closed under $\kappa$-sequences hence $p^*\in M_1$. The argument we have seen before applied in $M_1[G_{\kappa_1+1}*H^{*}]$ shows that  $$M_1[G_{\kappa_1+1}^*]\models |\dom(p) \cap (\kappa_1+1)\times j''_{2,1}{}(\kappa_1^{++})^{M_1[G_{\kappa_1+1}]}|\leq \kappa_1.$$
 This implies that $p^*\in M_2[G_{\kappa_2+1}*H^{**}]$ since $M_2[G_{\kappa_2+1}*H^{**}]$ is closed under $\kappa_1$-sequences from $M_1[G_{\kappa_1+1}*H^*]$. Then the embedding $j_2:V\to M_2$ extends to
$$j^{*}_2: V[G_{\kappa+1}*H^*]\to M_2[G_{\kappa_2+1}*H^{**}].$$
Define now
$$W=\{X\subseteq \kappa \mid \kappa_1 \in j_2^{*}(X)\}.$$
\begin{claim}\label{final claim 2}
\begin{enumerate}
    \item $j_W=j^*_2$, $[id]_W=\kappa_1$, $U^*\leq_{R-K}W$.
    \item $Cub_\kappa\subseteq W$, $\{\alpha<\kappa\mid \cf(\alpha)=\alpha)\}\in W$.
    \item The sequence $\l A_\alpha\mid \alpha<\kappa^{++}\r$ is a strong witness for $\neg\text{Gal}(W,\kappa,\kappa^{++})$, where
    $$A_\alpha:=\{\nu<\kappa\mid f_{\kappa,\alpha}(\nu)\text{ is odd}\}$$
\end{enumerate}
\end{claim}
\begin{proof}
Indeed  $Cub_\kappa\subseteq W$ and $\{\alpha<\kappa\mid \cf(\alpha)=\alpha\}\in W$, is the same as in Claim \ref{final claim} from the last section.
 To see $(1)$, we let $k_W:M_W\rightarrow M_2[j_2^*(G)]$ be the usual factor map $k_W([f]_W)=j^*_2(f)(\kappa_1)$ and we prove that $k_W=id$ by proving that $k_W$ is onto. Let $A\in M_2[G_{\kappa_2+1}*H^{**}]$, then $A=(\lusim{A})_{G_{\kappa_2+1}*H^{**}}$ where $\lusim{A}\in M_2$ is a $\mathcal{P}_{\kappa_2+1}*j_2(\mathbb{S}_0)$-name. Since $j_{2,1}$ is an $(\kappa_1,\kappa_1^{++})-$extender ultrapower, there is $f\in M_1$ and $a\in[\kappa_1^{++}]^{<\omega}$ such that $\lusim{A}=j_{2,1}(f)(a)$. Suppose that $a=\{\alpha_1,..,\alpha_n\}$ is an increasing enumeration. Then by construction, $f_{\kappa_2,j_{2,1}(\alpha_i)}(\kappa_1)\in\{\alpha_i\cdot2,\alpha_i\cdot 2+1\}$. In particular we derive $\alpha_i$ from $f_{\kappa_2,j_{2,1}(\alpha_i)}(\kappa_1)$
 \footnote{ An easy transfinite induction, proves that if an ordinal $\gamma=\beta\cdot 2$ or $\gamma=\beta \cdot2+1$, then $\beta$ is unique, we denote  $\beta=\lfloor \frac{\gamma}{2}\rfloor$.}. Define $g_{\alpha_i}:\kappa_1\rightarrow\kappa_1\in M_1[G_{\kappa_1+1}*H^*]$ by $g_{\alpha_i}(\alpha)=\lfloor \frac{f_{\kappa_1,\alpha_i}(\alpha)}{2}\rfloor$, then $j^*_{2,1}(g_{\alpha_i})(\kappa_1)=\lfloor \frac{f_{\kappa_2,j_{2,1}(\alpha_i)}(\kappa_1)}{2}\rfloor=\alpha_i$. 
 Finally, let $g(\alpha)=f(g_{\alpha_1}(\alpha),...,g_{\alpha_n}(\alpha))$.
 Then, $$j^*_{2,1}(g)(\kappa_1)=j_{2,1}(f)(j^*_{2,1}(g_{\alpha_1})(\kappa_1),...,j^*_{2,1}(g_{\alpha_n})(\kappa_1))=j_{2,1}(f)(a)=\lusim{A}$$
 We already know that $M_1[G_{\kappa_1+1}*H^*]$ is the ultrapower by $U^*$, hence $g=j^*_1(h)(\kappa)$ for some $h\in V[G_{\kappa+1}*H]$ and in turn $\lusim{A}=j^*_2(h)(\kappa,\kappa_1)$. Finally, we made sure that $\kappa$ is expressible by $\kappa_1$, so we define in $V[G_{\kappa+1}*H]$ $f^*:\kappa\rightarrow\kappa$ by
 $$f^*(\alpha)=(h(f_{\kappa,\alpha}(\alpha),\alpha))_G$$
 It follows that:
 $$k_W([f^*]_W)=j_2(f^*)(\kappa_1)=(j_2^*(h)(f_{\kappa_2,\kappa_1}(\kappa_1),\kappa_1))_{G_{\kappa_2+1}*H^{**}}$$
 $$=(j^*_2(h)(\kappa,\kappa_1))_{G_{\kappa_2+1}*H^{**}}=(\lusim{A})_{G_{\kappa_2+1}*H^{**}}=A,$$
 this concludes $(1)$,
 $(2),(3)$ is completely analogous to Claim \ref{final claim}.
\end{proof}

\end{proof}

\subsection{Adding $\kappa^{++}$-Cohens using Prikry forcing}

The construction of the previous section can be modified to obtain a model in which there is a $\kappa$-complete ultrafilter $U^*$ over $\kappa$ such that $\Pri(U^*)$ adds a generic filter for $\Add(\kappa,\kappa^{++})$. This will require the violation of $SCH$ and in turn larger cardinals \cite{Gitikstrength},\cite{Mit}. 

\begin{theorem}\label{Many-Cohens}
Assume $GCH$ and that $E$ is a $(\kappa,\kappa^{++})-$ extender in $V$. Then there is a cofinality preserving forcing extension $V^*$ in which $2^{\kappa}=\kappa^{++}$ and a non-Galvin ultrafilter $W\in V^*$ such that forcing with $\Pri(W)$ introduces a $V^*$-generic filter for $Cohen^{V^*}(\kappa,\kappa^{++})$-generic filter.
\end{theorem}
\begin{proof}

Let $j_1:V\to M_E=:M_1$ be the ultrapower embedding of $E$ with $crit(j_1)=\kappa$ and ${}^\kappa M_1\subseteq M_1$ and $\kappa_1=j_1(\kappa)$. Denote by $E_\alpha$ the ultrafilter $\{X\subseteq \kappa\mid \alpha\in j_E(X)\}$. As before, denote $E_\kappa$ by $U$ and let $k:M_U \to M_E$ be defined by setting $k(j_U(f)(\kappa))=j_E(f)(\kappa)$. 
Define an Easton support iteration
$\l \calP_\alpha, \lusim{Q}_\beta\mid \alpha\leq\kappa+1, \beta<\kappa\r$ as follows:

$\lusim{Q}_\beta$ is trivial unless $\beta$ is  inaccessible. If $\beta<\kappa$ is inaccessible, then $$\lusim{Q}_\beta=\Lot(\Add(\beta, \beta^{++}),\Add(\beta, \beta^{++})\times \Add(\beta, \beta^{++}))$$
Over $\kappa$, we let $\lusim{Q}_\kappa=\Add(\kappa, \kappa^{++})$.

Let $G_{\kappa+1}=G_\kappa*F_\kappa$ be a $V$-generic filter of $\calP_{\kappa+1}$. We denote by $F_\alpha:=\l f_{\alpha, \gamma}\mid \gamma<\alpha^{++}\r$ the generic Cohen function if $\Add(\alpha,\alpha^{++})$ was forced in $G_\kappa$ and by $$F_\alpha:=\l f_{\alpha, \gamma}\mid \gamma<\alpha^{++}\r,\ H_\alpha:=\l h_{\alpha,\gamma}\mid \gamma<\alpha^{++}\r$$ if $\Add(\alpha,\alpha^{++})\times \Add(\alpha,\alpha^{++})$ was. The elementary embedding
$j_1$ extends to $j^*_1:V[G_{\kappa+1}]\to M_1[G_{\kappa_1+1}]$ such that at $\kappa$ we forced one block of Cohen's,  $\Add(\kappa,\kappa^{++})$, and for every $\alpha<\kappa^{++}$, $$f_{\kappa_1,j_1(\alpha)}(\kappa)=\alpha.$$ Indeed, in the Woodin and Ben-Shalom argument we first build the generic $G_{\kappa_1}$ up to $\kappa_1$ not including $\kappa_1$ in the same standard fashion as in \cite{CummingsHand}.
The original construction of Woodin or Ben-Shalom of the Cohen generic $F_{\kappa_1}$ which is $M_1[G_{\kappa_1}]$-generic for $\Add(\kappa_1,\kappa_1^{++})^{M_1[G_{\kappa_1}]}$ applies in our case, as it only uses the fact that $M_1[G_{\kappa_1}]$ is closed under $\kappa$-sequences and properties of
$\Add(\kappa_1,\kappa_1^{++})$. Since $$\Add(\kappa_1,\kappa_1^{++})\simeq \Add(\kappa_1,\kappa_1^{++})\times \Add(\kappa_1,\kappa_1^{++}),$$
we can split the generic $F_{\kappa_1}$ and assume it is of the form $F_{\kappa_1}\times H_{\kappa_1}$, which is $M_1[G_{\kappa_1}]-$generic for $\Add(\kappa_1,\kappa_1^{++})\times \Add(\kappa_1,\kappa_1^{++})$. Work inside $V[G_{\kappa}*F_\kappa]$, modify the values of $F_{\kappa_1}$ and $H_{\kappa_1}$, as in the previous section so that for every $\alpha<\kappa^{++}$, $$f_{\kappa_1,j_1(\alpha)}\restriction \kappa= h_{\kappa_1,j_1(\alpha)\cdot2+1}\restriction \kappa=f_{\kappa,\alpha}$$ and for every $\alpha<\kappa^{++}$, $f_{\kappa_1,j_1(\alpha)}(\kappa)=\alpha$.

Lift $j_1$ to the embedding $j_1\subseteq j_1^*:V[G_{\kappa+1}]\rightarrow M_E[G_{\kappa_1}*F_{\kappa_1}]$. Note that $H_{\kappa_1}$ will be used only later.
Set $$U^*=\{X\subseteq\kappa\mid \kappa\in j^*_1(X)\},$$ then $U\subseteq U^*$ and $j^*_1$ is actually the ultrapower embedding by  $U^*$.
Continuing as before, consider the second ultrapower (of $V$) by $E$. Denote $M_E$ by $M_1$ and $\Ult(M_E, j_E(E))$ by $M_2$, $j_{2,1}=j_{j_1(E)}:M_1\rightarrow M_2$ the ultrapower embedding.
Also, let $ E_1=j_1(E)$ and $\kappa_2=j_{2,1}(\kappa_1)$.
 Let $j_2:V\to M_2$ be the composition of $j_1$ with $j_{2,1}$. The extension of $j_{2,1}$ will be such that at $\kappa_1$ we force with  $\Add(\kappa_1,\kappa_1^{++})\times \Add(\kappa_1,\kappa_1^{++})$ part of the Lottery sum. 
 To realize this, we define in $M_1[G_{\kappa_1}*(F_{\kappa_1}\times H_{\kappa_1})]$ we take the generic $G_{\kappa_1}$ up to $\kappa_1$. At $\kappa_1$ we take $F_{\kappa_1}\times H_{\kappa_1}$, then in $M_1[G_{\kappa_1}*(F_{\kappa_1}\times H_{\kappa_1})]$ we construct as in Wooding and Ben-shalom argument in $V[G_{\kappa}*F_\kappa]$ an $M_2[G_{\kappa_1}*(F_{\kappa_1}\times H_{\kappa_1})]$-generic $G_{(\kappa_1,\kappa_2)}*F_{\kappa_2}$  such that $j_{2,1}''G_{\kappa_1}*F_{\kappa_1}\subseteq G_{\kappa_2}*F_{\kappa_2}$. Denote by $\l f_{\kappa_2,\alpha}\mid\alpha<(\kappa_2^{++})^{M_2}\r$ the Cohen function induced by $F_{\kappa_2}$. We also secure that for every $\alpha<(\kappa_1^{++})^{M_1}$: 
\begin{enumerate}
  \item $f_{\kappa_2 k(\alpha)}(\kappa_1)=\alpha\cdot 2+1$, if $\alpha \in j_E{}''\kappa^{++}$,
  \item $f_{\kappa_2 k(\alpha)}(\kappa_1)=\alpha\cdot 2$, if $\alpha \in (\kappa_1^{++})^{M_1}\setminus j_E{}''\kappa^{++}$.
  \item $f_{\kappa_2,\kappa_1}(\kappa_1)=\kappa$.
\end{enumerate}
Formally, given $p \in  (\Add(\kappa_2, (\kappa_2)^{++}))^{M_2[G_{\kappa_2}]}$, define $p^*$ such that $\dom(p^*)=\dom(p)$ and
$$p^*(\l\gamma,\beta\r)=\begin{cases} f_{\kappa_1,\alpha}(\gamma) & \gamma<\kappa_1\wedge \beta=k(\alpha)\\
\alpha\cdot 2+1 & \gamma=\kappa_1\wedge \beta=k(\alpha)\wedge \alpha\in j''_E\kappa^{++}\\ \alpha\cdot 2  & \gamma=\kappa_1\wedge \beta=k(\alpha)\wedge \alpha\in (\kappa_1^{++})^{M_1}\setminus j_E''\kappa^{++}\\
\kappa & \alpha=\gamma=\kappa_1\\ p(\l\gamma,\alpha\r) & otherwise\end{cases}$$
In $V[G_{\kappa+1}]$, $|\dom(p)\cap j''_{E^2}\kappa^{++}|\leq \kappa$ and $M_1[G_{\kappa_1+1}]$ is closed under $\kappa$-sequences hence $p^*\in M_1[G_{\kappa_1+1}]$. The argument we have seen before applied in $M_1[G_{\kappa_1+1}]$, thus $$M_1[G_{\kappa_1+1}]\models |\dom(p) \cap (\kappa_1+1)\times j''_{1 2}{}(\kappa_1^{++})^{M_1[G_{\kappa_1+1}]}|\leq \kappa_1.$$
 This implies that $p^*\in M_2[G_{\kappa_2+1}]$ since $M_2[G_{\kappa_2+1}]$ is closed under $\kappa_1$-sequences from $M_1[G_{\kappa_1+1}]$.
 
 Extend in $V[G_{\kappa}*F_\kappa]$, $j_{2,1}\subseteq j^*_2:M_1[G_{\kappa_1}*F_{\kappa_1}\rightarrow M_2[G_{\kappa_2}*F_{\kappa_2}]$
 and let $j^*_2:V[G_{\kappa}*F_\kappa]\rightarrow M_2[G_{\kappa_2}*F_{\kappa_2}]$ be the composition $j^*_{2,1}\circ j_1^*$. Note that $j_{2,1}^*$ is definable only in $V[G_{\kappa}*(F_{\kappa}]$. Denote by $V[G_{\kappa}*F_\kappa]=V^*$, define $$W=\{X\subseteq\kappa\mid \kappa_1\in j^*_2(X)\}\in V^*\text{ and  }A_\alpha=\{\beta<\kappa\mid f_\alpha(\beta)\text{ is odd}\}.$$

\begin{claim}
$W$ is a $\kappa$-complete ultrafilter over $\kappa$ such that

\begin{enumerate}
    \item $j_W=j^*_2$, $[id]_W=\kappa_1$, $U^*\leq_{R-K}W$.
    \item $Cub_\kappa\subseteq W$, $\{\alpha<\kappa\mid \cf(\alpha)=\alpha)\}\in W$.
    \item $L_0:=\{\beta<\kappa\mid \Add(\beta,\beta^{++})\times \Add(\beta,\beta^{++})\text{ was forced in }G_{\kappa+1}\}\in W.$
    \item For every $\alpha<\kappa^{++}$, $L_{1,\alpha}:=\{\nu<\kappa\mid f_{\kappa,\alpha}(\nu)<\nu^{++}\}\in W$
    \item The sequence $\l A_\alpha\mid \alpha<\kappa^{++}\r$ is a strong witness for $\neg\text{Gal}(W,\kappa,\kappa^{++})$. Moreover, the sequence $\l A_\alpha\cap L_{1,\alpha}\mid \alpha<\kappa^{++}\r$ is a witness for $\neg\text{Gal}(W,\kappa,\kappa^{++})$.
\end{enumerate}

\end{claim}
\begin{proof}
 $(1),(2)$ and the first part of $(5)$ is the same argument as in Claim \ref{final claim 2}.
As for $(3)$, note that we have constructed the generic $G_{\kappa_2+1}=j^*_2(G_{\kappa+1})$ so that on $\kappa_1$ we have forced $\Add(\kappa_1,\kappa_1^{++})\times \Add(\kappa_1,\kappa_1^{++})$. To see $(4)$, for every $\alpha<\kappa^{++}$, $$j_2^*(f_{\kappa,\alpha})(\kappa_1)=f_{\kappa_2,j_{2,1}(j_1(\alpha))}(\kappa_1)=j_1(\alpha)\cdot 2+1<\kappa_1^{++}.$$ 
Hence by elementarity, $\kappa_1\in j^*_2(L_{1,\alpha})$. Finally, the moreover part of $(5)$, toward a contradiction if there would be a set $I\in [\kappa^{++}]^{\kappa}$ such that $\cap_{i\in I}A_\alpha\cap L_{1,\alpha}\in W$ then clearly $\cap_{i\in I}A_\alpha\in W$, contradicting the first part of $(5)$ that $A_\alpha$'s form a witness for $\neg\text{Gal}(W,\kappa,\kappa^{++})$. 
\end{proof}
Denoted by $\nu\mapsto \pi_{nor}(\nu)$ the Rudin-Keisler projection from $W$ to $U^*$, and
let us prove that $W$ witnesses the theorem:
\begin{proposition}\label{Proof for W is Cohen 2}
Let  $H\subseteq \Pri(W)$ be a $V^*$-generic filter. There  is $G^*\in V^*[H]$ which is $V^*$-generic for $Cohen(\kappa,\kappa^{++})^{V^*}$.  
\end{proposition}
\begin{proof}[ Proof of proposition \ref{Proof for W is Cohen 2}:]
Let $\l c_n \mid n<\omega \r $ be the $W$-Prikry sequence corresponding to $H$. Suppose without loss of generality that for every $n<\omega$, $c_n\in L_0$.

Define, for every $n<\omega$, the set
$$Z_n=\{\alpha<\kappa^{++}\mid \{c_m \mid n\leq m<\omega\} \subseteq A_\alpha\cap L_{1,\alpha} \text{ and } n \text{ is least possible } \}.$$
For every $\alpha<\kappa^{++}$, let $n_\alpha$ be the unique $n$ such that $\alpha\in Z_n$.
Let $\alpha<\kappa^+$, define $f^*_\alpha:\kappa\to \kappa$ as follows:

 Denote by $$\l f_{c_n,\alpha}\mid\alpha<c_n^{++}\r, \ \l h_{c_n,\alpha}\mid\alpha<c_n^{++}\r$$
the generic $c_n-$Cohen functions forced by $G$ and define the function
$f^*_\alpha:\kappa\rightarrow\kappa$ by
$$f^*_\alpha=h_{c_{n_\alpha},f_{\kappa,\alpha}(c_{n_\alpha})}\cup(\bigcup_{n_\alpha<n<\omega}h_{c_n,f_{\kappa,\alpha}(c_n)}\restriction[c_{n-1},c_n))$$

Note that the Cohen functions on $\kappa$ play the role of the canonical functions from the previous section.
Let us prove that $F=\l f^*_\alpha\mid \alpha<\kappa^{++}\r$ are Cohen generic functions over $V^*$.

\begin{claim}
Let $G^*=\{p\in \Add(\kappa,\kappa^{++})^{V^*}\mid p\subseteq F\}$, then $G^*$ is a $V^*$-generic filter.
\end{claim}
Let $\mathcal{A}\in V^*$ be a maximal antichain in the forcing $\Add(\kappa, \kappa^{++})^{V^*}$. Note that since $\Add(\kappa,\kappa^{++})^{V^*}$ is $\kappa-$closed then $$Cohen(\kappa,\kappa^{++})^{V[G_\kappa]}=Cohen(\kappa,\kappa^{++})^{V^*}.$$ 
By $\kappa^+-$cc of the forcing, there is $Y'\subseteq\kappa^{++}$, $Y'\in V$ such that $|Y'|=\kappa$ and $\mathcal{A}\subseteq \Add(\kappa,Y')^{V^*}$.  Also, since $|\mathcal{A}|=\kappa$, $\mathcal{A}\in V[G_\kappa*F_\kappa]$, there is $Z\subseteq\kappa^{++}$ such that $|Z|=\kappa$ such that $\mathcal{A}\in V[G_\kappa*F_\kappa\restriction Z]$. Without loss of generality assume that $Z=Y\in V$. Let $V\ni\phi:\kappa\rightarrow Y$ be a bijection.

As in claim \ref{elementary chain construction}, we can construct an $\in-$increasing continuous chain $\l N_\beta\mid \beta<\kappa\r\in V^*$ of elementary submodels of $H_\chi$ such that

\begin{enumerate}
  \item $|N_\beta|<\kappa$,
  \item $G_{\kappa+1},\mathcal{A},\phi,Y\in N_0$,
  \item $N_\beta\cap \kappa=\gamma_\beta$ is a cardinal $<\kappa$, $\gamma_{\beta+1}$ is regular.
\item If $\gamma_\beta$ is regular, then $\Add(\gamma_\beta,\phi''\gamma_\beta)=\Add(\kappa,Y)\cap N_\beta$.
\end{enumerate}

Set
$$C=\{ \beta<\kappa \mid \gamma_\beta=\beta \}.$$
This is club in $\kappa$ since the sequence $\gamma_\beta$ is continuous and since the set $\{\beta\mid \gamma_\beta=\beta\}$ is a club.  

Recall that by construction $j^*_2(\l f_{\kappa,\alpha}\mid \alpha<\kappa^{++}\r)=\l f_{\kappa_2,\alpha}\mid\alpha<\kappa_2^{++}\r$. Also, for every $\nu\in j_2(\phi)''\kappa_1$ there is $\gamma<\kappa_1$ such that $\nu=j_2(\phi)(\gamma)$, and since $crit(j_{2,1})=\kappa_1$, $\nu=j_{2,1}(j_1(\phi)(\gamma))$. Since $j_1(\phi):\kappa_1\rightarrow \kappa_1^{++}$ we conclude that $\nu=j_{2,1}(\alpha)$ for some $\alpha<(\kappa_1^{++})^{M_1}$ which implies that $$f_{\kappa_2,\nu}(\kappa_1)\in\{\alpha\cdot 2,\alpha\cdot 2+1\}.$$ Since $\phi$ is a bijection, for every distinct $\nu_1,\nu_2\in j_2(\phi)''\kappa_1$,  $f_{\kappa_2,\nu_1}(\kappa_1)\neq f_{\kappa_2,\nu_2}(\kappa_1)$. Reflecting this, we obtain that the set $$E:=\{\nu<\kappa\mid \forall\nu_1,\nu_2\in \phi''\nu.\nu_1\neq\nu_2\rightarrow f_{\kappa,\nu_1}(\nu)\neq f_{\kappa,\nu_2}(\nu)\}\in W.$$ 

Also, by construction, for every $\alpha<\kappa_1^{++}$, $f_{\kappa_2,j_{2,1}(\alpha)}\restriction\kappa_1=f_{\kappa_1,\alpha}$ and therefore for every for every $\alpha\in j_2(\phi)''\kappa_1$, there is $\nu<\kappa_1^{++}$ such that $$\alpha=j_{2,1}(j_1(\phi))(\nu)=j_{2,1}(j_1(\phi)(\nu))$$ and $j_1(\phi)(\nu)<\kappa_1^{++}$. Hence $f_{\kappa_2,\alpha}\restriction\kappa_1=f_{\kappa_1,\beta}$ for some $\beta<\kappa_1^{++}$.
Reflecting this we obtain that the set $$F:=\{\beta<\kappa\mid \forall \gamma\in \phi''\beta.\exists\delta<\beta^{++}.f_{\kappa,\gamma}\restriction\beta=f_{\beta,\delta}\} \in W.$$ 

Now the argument of  Claim \ref{condition in model} applies since for every $\nu_0\in C\cap E\cap F$, $\forall\tau_1<\tau_2\in\phi''\nu_0$, $f_{\kappa,\tau_1}(\nu_0)\neq f_{\kappa,\tau_2}(\nu_0)$, hence   $\l h_{\nu_0,f_{\kappa,\tau}(\nu_0)}\mid \tau\in \phi''\nu_0\r$ are  distinct mutually  $V[G_{\nu_0}*F_{\nu_0}]$-generic Cohen functions over $\nu_0$. Thus , we can find $d\in\mathcal{A}\cap \Add(\nu_0,\nu_0^{++})$ such that $d$ is extended by $\l h_{\nu_0,f_{\kappa,\alpha}(\nu_0)}\mid \alpha\in \phi''\nu_0\r$.
Finally we note that
$$R:=\{\nu<\kappa\mid \forall \alpha\in\phi''\pi_{nor}(\nu). f_{\kappa,\alpha}(\nu)\text{ is odd}\}\in W$$

Let $p=\l \l\r, B\r$ be a condition, shrink $B$ to $B_0:=B\cap C\cap E\cap F\cap R\in W$ and pick now any $\nu_0\in B_0$. Split $\phi''\nu_0$ into two sets: $$X^{\nu_0}_0:=\{\tau\in\phi''\nu_0\mid \nu_0\in A_\tau\}\text{ and }X^{\nu_0}_1=\phi''\nu_0\setminus X^{\nu_0}_0.$$
Since $\nu_0\in R$ we have that $X_1\subseteq \phi''(\pi_{nor}(\nu_0),\nu_0)$.
The condition  $p_0=\l \l \nu_0\r, B_0\cap (\bigcap_{\tau\in\phi''\nu_0}A_\tau)\r$, forces the following:
\begin{enumerate}
  \item the Prikry sequence is included in each $A_\tau$, $\tau\in X^{\nu_0}_0$, i.e., $n_\tau=0$,

  \item $n_\tau=1$, for every $\tau\in X^{\nu_0}_1$.
\end{enumerate}
In particular, this condition forces some information about the Cohen functions. Namely that:
\begin{enumerate}
    \item for $\tau\in X_0^{\nu_0}$,  $f^*_{\tau}\restriction \nu_0=h_{\nu_0,f_{\kappa,\tau}(\nu_0)}$ 
    \item for $\tau\in X_1^{\nu_0}$, $f^*_{\tau}\restriction \nu_0=h_{\lusim{c}_1,f_{\kappa,\tau}(\lusim{c}_1)}\restriction \nu_0$.
\end{enumerate}

We would like to find a condition in $\mathcal{A}$ which is below these decided parts of the Cohen. By the previous paragraph, there is $d\in N_{\nu_0}\cap \Add(\kappa,Y)=\Add(\nu_0,\phi''\nu_0)$, which is extended by $\l h_{\nu_0,f_{\kappa,\tau}(\nu_0)}\mid \tau\in \phi''\nu_0\r$. As before we will need to pick $\nu_0,\nu_1$ so that $d^{\nu_0}\in G^*$.

Let $\lusim{B_0}$ be a name in $V$ for $B_0$.
We fix a condition $m_0\in G_\kappa*F_\kappa$ which forces that if $\nu_0\in\lusim{B_0}$ then  there is $d\in \Add(\nu_0,\phi''\nu_0)\cap\lusim{\mathcal{A}}$ which is extended by $\l \lusim{h}_{\nu_0,f_{\kappa,\tau}(\nu_0)}\mid \nu_0\in\phi''\nu_0\r$, and $\forall \alpha\in \phi''\lusim{\pi_{nor}(\nu_0)}.\ \nu_0\in \lusim{A}_\alpha$. Recall that by the construction of $G_{\kappa_2}$, we have $m_0\in G_{\kappa_2}*F_{\kappa_2}$. 
Let $ m_0\leq t\in G_{\kappa_2}*F_{\kappa_2}$ be a condition such that
$$(1) \ \ \ t\Vdash\kappa_1\in j_2(\lusim{B_0}).$$
By the construction of $G_{\kappa_2}*F_{\kappa_2}$, $t$ has the form:
$$t=\l t_{<\kappa}, t_{\kappa},t_{(\kappa,\kappa_1)},\underset{t_{\kappa_1}}{\underbrace{\l t_{\kappa_1}^0,t_{\kappa_1}^1\r}},t_{(\kappa_1,\kappa_2)},t_{\kappa_2}\r.$$ 
Distinguishing from the case of $\kappa^+$, we now have that  $f_{\kappa_2,j_2(\alpha)}(\kappa_1)=j_1(\alpha)\cdot 2+1$ for every $\alpha<\kappa^+$, this will hold for every $\alpha\in\phi''\kappa$ as well. Also, recall that $Y\in V$, hence $\phi\in V$. Thus $j_2(\phi)\in M_2$ and  $j_2(\phi)''\kappa\in M_2$. Also, for $(t_{\kappa_2})_{G_{\kappa_2}}\in M_2[G_{\kappa_2}]$, $$j_2''\kappa^{++}\cap \supp((t_{\kappa_2})_{G_{\kappa_2}})\in M_2[G_{\kappa_2}]$$ and $(t_{\kappa_2})_{G_{\kappa_2}}\restriction \kappa\times \{j_2(\alpha)\}\subseteq f_{\kappa,\alpha}$. We also fix $X\in V$, $X\subseteq\kappa^{++}$, $|N_0|\leq\kappa$ such that $\supp((t_{\kappa_2})_{G_{\kappa_2}})\subseteq j_2(N_0)$.

Therefore, we can extend if necessary $t$ such that
$$ (2a)\ \ t_{<\kappa_2}\Vdash (\kappa\cup\{\kappa_1\})\times j_2(\phi)''\kappa \subseteq \dom(t_{\kappa_2})\wedge (0,\kappa_1)\in\dom(t_{\kappa_2})\wedge \supp(t_{\kappa_2})\subseteq j_2(N_0)$$
$$(2b) \ \ t_{<\kappa_2}\Vdash \  t_{\kappa_2}(\kappa_1,j_2(\alpha))=j_1(\alpha)\cdot 2+1,\text{ for every }j_2(\alpha)\in j_2(\phi)''\kappa\text{ and } t_{\kappa_2,\kappa_1}(0)=\kappa.$$
$$(2c) \ \  t_{<\kappa_2}\Vdash t_{\kappa_2,j_2(\alpha)}\restriction \kappa=\lusim{f}_{\kappa,\alpha}\text{  for every } j_2(\alpha)\in j_2''\kappa^+\cap\supp(t_{\kappa_2}).$$

Next consider $t_{\kappa_1}=\l t_{\kappa_1}^0,t_{\kappa_1}^1\r$, it is a $\calP_{\kappa_1}-$name for a condition in $F_{\kappa_1}\times H_{\kappa_1}$. By the construction of the generic $F_{\kappa_1}\times H_{\kappa_1}$, for every $\alpha<\kappa^{++}$, we made sure that, $h_{\kappa_1,j_1(\alpha)2+1}\restriction\kappa=f_{\kappa,\alpha}$. Also, $(j_1(\phi)''\kappa)\cdot 2+1\in M_2$\footnote{ For a set of ordinals $A$, let $A\cdot 2+1=\{\alpha\cdot 2+1\mid \alpha\in A\}$.}. Let $$\mu_1=\{\l j_1(\alpha)\cdot2+1,\alpha\r\mid \alpha\in \phi''\kappa\}\in M_1$$ The fact that for every $\beta<\kappa^{++}$,  $f_{\kappa_2,j_2(\beta)}(\kappa_1)=j_1(\beta)\cdot2+1$ implies $$\dom(\mu_1)=(j_1(\phi)''\kappa)\cdot 2+1=\{f_{\kappa_2,\gamma}(\kappa_1)\mid \gamma\in j_2(\phi)''\kappa\}, \ \rng(\mu_1)=\phi''\kappa\subseteq\kappa^{++}$$
Extend if necessary $t_{<\kappa_1}$, and assume that $$(3) \ \ t_{<\kappa_1}\Vdash\kappa\times (j_1(\phi)''\kappa)\cdot 2+1\subseteq\dom(t^1_{\kappa_1})\wedge \forall j_1(\alpha)\in j_1(\phi)''\kappa, \  t_{\kappa_1,j_1(\alpha)\cdot 2+1}^1\restriction\kappa= \lusim{f}_{\kappa,\alpha}.$$
 As for the lower part, due to the Easton support, we have $$(4) \ \ t_{<\kappa}\in V_\kappa.$$
Fix functions $r,\Gamma_1$ which represents $t,\mu$ resp. in the ultrapower $M_{E^2}$, namely for some 
$\vec{\xi}\in [\kappa_1^{++}]^{<\omega}$, $j_2(r)(\vec{\xi})=t, \ j_2(\Gamma_1)(\vec{\xi})=\mu$. 
Without loss of generality, suppose that   both $\kappa$ and $\kappa_1$ appear in $\vec{\xi}$, $\kappa=\min(\vec{\xi})=\vec{\xi}(0)$ and $\kappa_1=\vec{\xi}(i_0)$. Then the functions $\vec{\nu}\in[\kappa]^{|\vec{\xi}|}\mapsto (\vec{\nu}(0),\vec{\nu}(i_0))$ represent $(\kappa,\kappa_1)$.
Without loss of generality, suppose that for every $\vec{\nu}$, it takes the form $$r(\vec{\nu})=\l r_{<\vec{\nu}(0)}, r_{\vec{\nu}(0)},r_{(\vec{\nu}(0),\vec{\nu}(i_0))},\l r_{\vec{\nu}(i_0)}^0,r_{\vec{\nu}(i_0)}^1\r,r_{(\vec{\nu}(i_0),\kappa)},r_\kappa\r.$$

Reflecting some of the properties of $t$ we obtain a set $B'\in E(\vec{\xi})$ such that for every $\vec{\nu}\in B'$:
\begin{enumerate}
    \item [$(1)_{\vec{\nu}}$] $r(\vec{\nu})\Vdash \vec{\nu}(i_0)\in \lusim{B_0}$.
    \item [$(2a)_{\vec{\nu}}$] $ r_{<\kappa}\Vdash (\vec{\nu}(0)\cup\{\vec{\nu}(i_0)\})\times \phi''\vec{\nu}(0) \subseteq \dom(r_{\kappa})\wedge$
    
    $\l0,\vec{\nu}(i_0)\r\in\dom(r_{\kappa}) \wedge\supp(r_{\kappa})\subseteq N_0$

    \item [$(2b)_{\vec{\nu}}$] $r_{<\kappa}\Vdash  \forall\alpha\in\phi''\vec{\nu}(0).r_{\kappa,\alpha}(\vec{\nu}(i_0))$ is odd and $r_{\kappa,\vec{\nu}(i_0)}(0)=\vec{\nu}(0)$. 
   
    \item  [$(3)_{\vec{\nu}}$] $ r_{<\vec{\nu}(i_0)}\Vdash \vec{\nu}(0)\times\dom(\Gamma_1(\vec{\nu}))\subseteq\dom(r^1_{\vec{\nu}(i_0)})$ and for every
    
    $\beta\in\dom(\Gamma_1(\vec{\nu})),  r^1_{\vec{\nu}(i_0),\beta}\restriction\vec{\nu}(0)=\lusim{f}_{\vec{\nu}(0),\Gamma_1(\vec{\nu})(\beta)}$.
    
    \item [$(4)_{\vec{\nu}}$] $r_{<\vec{\nu}(0)}=t_{<\kappa}\in V_{\vec{\nu}(0)}$.
    
\end{enumerate}
\vskip 0.2 cm
Let $$B''=\{\nu(i_0) \mid\exists\vec{\nu}\in B'. r(\vec{\nu})\in G_\kappa*F_\kappa\}.$$
Since $B'\in E(\vec{\xi})$ we have that $\vec{\xi}\in j_2(B')$ and since $j_2(r)(\vec{\xi})=t\in j^*_2(G_\kappa*F_\kappa)=G_{\kappa_2}*F_{\kappa_2}$, we conclude that $B''\in W$. Also, $B''\subseteq B_0$ by clause $(1)$. 

We proceed by a density argument, recall that by the definition of $G_2$, we have that $\l t_{<\kappa},t_\kappa\r\in G_\kappa*F_\kappa$. 
\begin{claim}
Let $D$ be the set of all conditions $q\in\mathcal{P}_{\kappa+1}$, such that exists $\vec{\nu}_0,\vec{\nu}_1\in B', \ \vec{\nu}_1(0)>\vec{\nu}_0(i_0)$ and a $\mathcal{P}_{\vec{\nu}_0(i_0)}-$name $\lusim{d}^{\vec{\nu}_0(i_0)}$ such that \begin{enumerate}
    \item [(a)] $r(\vec{\nu}_0),r(\vec{\nu}_1)\leq q$.
    \item [(b)] $q\Vdash \lusim{d}^{\vec{\nu}_0(i_0)}\in \lusim{A}\cap\Add(\vec{\nu}_0(i_0),\phi''\vec{\nu}_0(i_0))$.
    \item [(c)] $q\Vdash \forall \tau\in X_1^{\vec{\nu}_0(i_0)}.\lusim{h}_{\nu_1,f_{\kappa,\tau}(\vec{\nu}_1(i_0))}\restriction \vec{\nu}_0(i_0)=\lusim{d}^{\vec{\nu}_0(i_0)}_\tau.$
    
\end{enumerate}
Then $D$ is dense (open) above $\l t_{<\kappa},t_\kappa\r$ and thus $D\cap G_\kappa*F_\kappa\neq\emptyset$
\end{claim}
\begin{proof}

Work in $V$, let $\l t_{<\kappa},t_\kappa\r\leq p:=\l p_{<\kappa},p_\kappa\r \in \calP_{\kappa+1}$.  We will define two extensions $p\leq q\leq q^*$ as before such that $q^*\in D$. By definition of $\mathcal{P}_{\kappa+1}$, $p_{<\kappa}\Vdash p_\kappa\in \Add(\kappa,\kappa^{++})$, by $\kappa-$cc of $\mathcal{P}_\kappa$, for some $Z\subseteq\kappa^{++},\ Z\in V$, $|Z|<\kappa$ and some $\gamma<\kappa$, $p_{<\kappa}\Vdash \dom(p_\kappa)\subseteq \gamma\times Z$ The same argument as before indicate that
$$p_{<\kappa}\Vdash Z\supseteq \supp(t_\kappa)\wedge \forall \beta\in Z.j_2(p_{\kappa})_{j_2(\beta)}\geq t_{\kappa,\beta}$$  $$t_{<\kappa_2}\Vdash \forall j_2(\tau)\in \supp(t_{\kappa_2})\cap j_2(Z). t_{\kappa_2,j_2(\tau)}\restriction \gamma= \lusim{f}_{\kappa,\tau}\restriction \gamma.$$
Denote $\mu=(j_2\restriction (Z\cup N_0))^{-1}\in M_2$, then
$$\dom(\mu)=j_2(Z)\cup j_2''N_0, \ \rng(\mu)=Z\cup\theta, \ \mu\text{ is 1-1}.$$
and we can reformulate:
$$p_{<\kappa}\Vdash \mu''j_2(Z)\supseteq \supp(t_\kappa)\wedge \forall \beta\in j_2(Z).j_2(p_{\kappa})_{\beta}\geq t_{\kappa,\mu(\beta)}$$  $$t_{<\kappa_2}\Vdash \forall \tau\in \supp(t_{\kappa_2})\cap j_2(Z). t_{\kappa_2,\tau}\restriction \gamma= \lusim{f}_{\kappa,\mu(\tau)}\restriction \gamma.$$
Also, find $\delta<\kappa$ such that $t_{<\kappa}\Vdash \phi''(\delta,\kappa)\cap  Z=\emptyset.$ 
We have that  $$\phi''(\delta,\kappa)=\mu_1''\{f_{\kappa_2,\gamma}(\kappa_1)\mid \gamma\in j_2(\phi)''(\delta,\kappa)\},\text{ and }\mu''\supp(j_2(p_\kappa))=Z.$$ Therefore in $M_2$ we will have that
 $$p_{<\kappa}\Vdash [\mu_1''\{\lusim{f}_{\kappa_2,\gamma}(\kappa_1)\mid \gamma\in j_2(\phi)''(\delta,\kappa)\}]\cap[ \mu''\supp(j_2(p_\kappa))]=\emptyset.$$ 
Let $\Gamma$ be such that $j_2(\Gamma)(\vec{\xi})=\mu$, there is a set $\overline{B}_0\subseteq B'$, $\overline{B}_0\in E(\vec{\xi})$ such that for every $\vec{\nu}\in\overline{B}_0$,
$$(i) \ \ p_{<\kappa}\Vdash \Gamma(\vec{\nu})''Z\supseteq \supp(r_{\vec{\nu}(0)})\wedge  \forall\beta\in Z. p_{\kappa,\beta}\geq r_{\vec{\nu}(0),\Gamma(\vec{\nu})(\beta)},$$ 
$$(ii) \ \ r_{<\kappa}\Vdash \forall \tau\in Z\cap \supp(r_{\kappa}). r_{\kappa,\tau}\restriction \gamma=\lusim{f}_{\vec{\nu}(0),\Gamma(\vec{\nu})(\tau)}\restriction \gamma.$$
$$(iii) \ \  p_{<\kappa}\Vdash\Gamma_1(\vec{\nu})''\{\lusim{f}_{\kappa,\gamma}(\vec{\nu}(i_0))\mid \gamma\in \phi''(\delta,\vec{\nu}(0))\}\cap [\Gamma(\vec{\nu})''\supp(p_\kappa)]=\emptyset.$$
 Find $\vec{\nu}_0,\vec{\nu}_1\in \overline{B}_0$ such that $r(\vec{\nu}_0),r(\vec{\nu}_1)$ are compatible, $\vec{\nu}_0(0)>\delta,\gamma,\sup(\supp(p_{<\kappa}))$, and $\vec{\nu}_1(0)>\vec{\nu}_0(i_0),\sup(\supp(r_{<\kappa}(\vec{\nu}))$. Denote  $$r^0:=r(\vec{\nu}_0)=\l r^0_{<\vec{\nu}_0(0)},r^0_{\vec{\nu}_0(0)},r^0_{(\vec{\nu}_0(0),\kappa)},r^0_\kappa\r,$$
$$r^1:=r(\vec{\nu}_1)=(r^1_{<\vec{\nu}_1(0)},r_{\vec{\nu}_1(0)},r_{(\vec{\nu}_1(0),\vec{\nu}_1(i_0))},\l r^{0,1}_{\vec{\nu}_1(i_0)},r^{1,1}_{\vec{\nu}_1(i_0)}\r,r^1_{(\vec{\nu}_1(i_0),\kappa)},r^1_\kappa\r$$
As before, $q$ has the form:
$q=p_{<\kappa}{}^{\smallfrown}q_{\vec{\nu}_0(0)}{}^{\smallfrown}r^0_{(\vec{\nu}_0(0),\kappa)}{}^{\smallfrown}q_\kappa$.
We have $q_{\vec{\nu}_0(0)}$ is a $\mathcal{P}_{\vec{\nu}_0(0)}-$name for a condition with $\supp(q_{\vec{\nu}_0(0)})= \Gamma(\vec{\nu}_0)''Z$ and
$q_{\nu'_0,\Gamma(\nu'_0,\nu_0)(\beta)}= p_{\kappa,\beta}$. As for $q_\kappa$, we set it to be a $\mathcal{P}_\kappa-$name for $r^0_\kappa\cup p_\kappa$.

The argument that $r^0\leq q$ is the same as in the case of $\kappa^+$.

The choice of $\lusim{d}^{\vec{\nu}_0(i_0)}$ is possible since $r^0\leq q$ and $m_0\leq \l t_{<\kappa},t_\kappa\r\leq q\Vdash \vec{\nu}_0(i_0)\in \lusim{B}_0$.

Define the final condition $q\leq q^*$,  $$q^*=q_{<\kappa}{}^{\smallfrown}q^*_{\vec{\nu}_1(0)}{}^{\smallfrown}r^1_{(\vec{\nu}_1(0),\kappa)}{}^{\smallfrown}q^*_\kappa.$$ Again we have that $r^0\Vdash X_1^{\vec{\nu}_0(i_0)}\subseteq \phi''(\vec{\nu}_0(0),\vec{\nu}_0(i_0))\subseteq\phi''(\vec{\nu}_0(0),\vec{\nu}_1(0))$ and by $(iii)$   $$q_{<\kappa}\Vdash [\Gamma_1(\vec{\nu}_1)''\{\lusim{f}_{\kappa,\gamma}(\nu_1)\mid \gamma\in X^{\nu_0}_1\}]\cap [\Gamma(\vec{\nu}_1)''Z]=\emptyset.$$ Now for the code  of $\lusim{d}^{\vec{\nu}_0(i_0)}$, let $$\supp(q^*_{\vec{\nu}_1(0)})=[\Gamma_1(\vec{\nu}_1)''\{\lusim{f}_{\kappa,\gamma}(\vec{\nu}_1(i_0))\mid \gamma\in X^{\vec{\nu}_0(i_0)}_1\}]\uplus [\Gamma(\vec{\nu}_1)''Z]$$ and
$$q^*_{\vec{\nu}_1(0),\alpha}=\begin{cases} q_{\kappa,\beta} &\exists \beta\in \Gamma(\vec{\nu}_1)''Z. \alpha=\Gamma(\vec{\nu}_1)(\beta)\\
\lusim{d}^{\vec{\nu}_0(i_0)}_\tau & \exists \tau\in X^{\vec{\nu}_0(0)}_1.\alpha=\Gamma_1(\vec{\nu}_1)(\lusim{f}_{\kappa,\tau}(\vec{\nu}_1(i_0)))\end{cases}$$ 
and $q^*_\kappa=q_\kappa\cup r^1_\kappa$. We conclude that $r^0\leq q\leq q^*$, $r^1\leq q^*$, namely $(a)$. Finally, for every $\tau\in X^{\vec{\nu}_0(i_0)}_1$, $\lusim{f}_{\kappa,\tau}(\vec{\nu}_1(i_0))\in \dom(\Gamma_1(\vec{\nu}))$ and by $(3)_{(\vec{\nu}_1)}$ we have that $q^*$ forces that
$$\lusim{h}_{\vec{\nu}_1(i_0),\lusim{f}_{\kappa,\tau}(\vec{\nu}_1(i_0))}\restriction \vec{\nu}_0(i_0)=\lusim{f}_{\vec{\nu}_1(0),\Gamma_1(\vec{\nu}_1)(\lusim{f}_\tau(\vec{\nu}_1(i_0)))}\restriction \vec{\nu}_0(i_0)\geq$$
$$\geq q^*_{\vec{\nu}_1(0),\Gamma_1(\vec{\nu}_1)(\lusim{f}_{\kappa,\tau}(\vec{\nu}_1(i_0)))}=\lusim{d}^{\vec{\nu}_0(i_0)}_\tau$$


Then $p\leq q^*$ and $q^*\in D$ 
\end{proof}
The rest of the argument remains unchanged.
\end{proof}

\end{proof}

\section{On the Extender-based Prikry forcings and adding subsets to $\kappa$ }

H. Woodin asked in the early 90s whether, assuming that there is no inner model with a strong cardinal, it is possible to have a model $M$ in which $2^{\aleph_\omega}\geq \aleph_{\omega+3}$, GCH holds  below $\aleph_\omega$, there is an inner model $N$ such that $\kappa=(\aleph_\omega)^M$ is a measurable and 
$2^\kappa\geq (\aleph_{\omega+3})^M$. His question was natural given the results known back then: Magidor proved \cite{MagAnnals} that it is consistent relative to a supercompact cardinal and a huge cardinal above it to have $2^{\aleph_\omega}\geq \aleph_{\omega+m}$ and $GCH_{<\aleph_\omega}$ using the supercompact Prikry forcing with collapses. Woodin, in an unpublished work which can be found in \cite{CummingWoodin} reduced Magidor's large cardinal assumption to get $2^{\aleph_{\omega}}=\aleph_{\omega+2}+GCH_<\aleph_{\omega}$ to a strong cardinal (actually to a $p_2\kappa-$hypermeasurable). 
Later, Gitik and Magidor \cite{Git-Mag} proved using the Extender-based Prikry forcing with collapses that starting from the optimal large cardinal assumption, it is possible to obtain $\aleph_{\omega+m}=2^{\aleph_\omega}$ and $GCH_{<\aleph_\omega}$. 
However, Woodin's question remained unanswered.

 A natural approach to answer Woodin's question is to force with the Extender-based Prikry forcing over $\kappa$ and then argue that in some intermediate where $\kappa$ is measurable we added $\lambda\geq\kappa^{++}$ many subsets to $\kappa$.
\\Our purpose will be to show that this direction is doomed. More precisely, we will prove that in any intermediate model of the Extender-based Prikry forcing where $\kappa^{++}$-many subsets of $\kappa$ were introduced, $\kappa$ is singularized (and in particular not measurable.
We will analyze the situation in both the original version of Gitik and Magidor from \cite{Git-Mag} and Merimovich version of the Extender-based Prikry forcing from \cite{Mer,Mer1,Mer2}.
We will rely on the following theorem from \cite[Theorem 6.7]{TomMoti}:
\begin{theorem}\label{ThesisRes}
 Suppose that $\mathbb{U}=\l U_a\mid a\in[\kappa]^{<\omega}\r$ is a tree of $P-$point ultrafilters. Let $G\subseteq P(\mathbb{U})$ be $V-$generic, then for every set of ordinals $A\in V[G]\setminus V$, $cf^{V[A]}(\kappa)=\omega$.
\end{theorem}
Note that if $U$ is any $\kappa$-complete ultrafilter, then the forcing $\Pri(U)$ which we use in this paper is forcing equivalent to $P(\mathbb{U})$ where $\mathbb{U}=\l U_a\mid a\in[\kappa]^{<\omega}\r$ is such that $U_a=U$ for every $a$.

Assume $2^\kappa=\kappa^+$.
Let $E$ be an  extender over $\kappa$.
We consider two sorts of Extender-based Prikry forcings - the original one, see \cite{Git-Mag} or \cite{Gitik2010}, and a more elegant version of Carmi Merimovich \cite{Mer, Mer1,Mer2}.

Let us start with Merimovich version, but in which the measures of $E$ are $P-$points as in \cite{Git-Mag}.

\subsection{ The Merimovich version with $P-$points }

Suppose that there is $h:\kappa\to \kappa$ such that all the generators of $E$ are below $j_E(h)(\kappa)$.
\\For example, if $E$ is a $(\kappa, \kappa^{++})-$extender, this holds with $h(\nu)=\nu^{++}$, $\nu<\kappa$. This is sufficient to ensure that for every $\alpha<\lambda$, $U_\alpha$ is a $P-$point ultrafilter.

Denote by $\mathbb{P}_E$ the Merimovich Extender-based Prikry forcing with $E$, as defined in \cite{Mer2}(or see definition \ref{Carmi Version}).

\begin{theorem}\label{Characterization of P-point extension}
Let $G\subseteq \mathbb{P}_E$ be a generic. Suppose that $A\in V[G]\setminus V$ is a subset of $\kappa$.
Then $\kappa$ changes its cofinality to $\omega$ in $V[A]$.

\end{theorem}
\begin{proof}

Work in $V$. Suppose that $\lusim{A}$ is a name of a subset of $\kappa$ and some $p\in \mathbb{P}_E$ forces that it is a new subset.

Let us use $\kappa^+-$properness of the forcing $\mathbb{P}_E$, see \cite[Claim 2.7]{Mer2} or \cite[Claim 3.29]{Mer}.
Pick now $N\preceq H_\chi$, for some $\chi$ large enough such that:

\begin{enumerate}
  \item $|N|=\kappa$,
  \item $N\supseteq {}^{\kappa>}N$,
  \item $E, \mathbb{P}_E, p, \lusim{A} \in N$.
\end{enumerate}

The properness implies that there is $p^*\geq^* p$ which is $\l N, \mathbb{P}_E\r-$generic, i.e.
$$p^*\Vdash (\forall D\in N (\text{ if } D \text{ is a dense open, then } D\cap N\cap \lusim{G}\not = \emptyset)).$$
In particular, for every $\nu<\kappa$, the dense open set
$$D_\nu:=\{q \mid \exists\alpha. q\Vdash otp(\lusim{A})>\nu\rightarrow\text{  the } \nu-\text{th element of } \lusim{A} \text{ is } \alpha\}$$
is definable from $\lusim{A}$ and $\nu$, hence in $N$ and it is  dense open by elementarity.

Consider $X=\cup_{p\in \mathbb{P}_E\cap N}\supp(p)$, since $\supp(p),N$ are of size $\kappa$, we have that $|X|\leq\kappa$. There exists $\alpha^*<\lambda$ such that for some $f\in V$, $j_E(f)(\alpha^*)=(j\restriction X)^{-1}$(See for example \cite[Lemma 3.3]{Gitik2010}). 

Denote $Y=X\cup\{\alpha^*\}$ and fix a set $R\in E_{Y}$ such that if $\mu\in R$, then $f(\mu(\alpha^*))=\mu\restriction X$. Such a set exists since $j_E(f)(j^{-1}(j(\alpha^*)))=(j\restriction X)^{-1}$, hence $$(j\restriction Y)^{-1}\in j_E(\{\mu\in ob(Y)\mid f(\mu(\alpha^*))=\mu\restriction X\}).$$
Find a condition $p_*\in G$ such that $Y\subseteq \supp(p)$ and $A^{p_*}\restriction Y\subseteq R$. 
Define $G\restriction Y=\{p\restriction Y \mid  p\in G/p_*\}$. Then by genericity of $p^*$ and definition of $Y$, for every $\alpha<\kappa$ there is $p_\alpha\in G\cap D_\nu\cap N$, hence $\supp(p_\alpha)\subseteq Y$ and we can find $p_\alpha\leq p^*_\alpha\in G\restriction Y\cap D_\nu$. It follows that $A\in V[G\restriction Y]$.  Let $G_{\alpha^*}=\{p\restriction\{\alpha^*\}\mid  p\in G/p_*\}$, in particular, $p_0:=p_*\restriction \{\alpha^*\}\in G_{\alpha^*}$. Note that $G_{\alpha^*}$ is essentially a Prikry generic filter for  $\Pri(U_{\alpha^*})$
\begin{claim}
$V[G\restriction Y]=V[G_{\alpha^*}]$.
\end{claim}
\begin{proof}
Inclusion from right to left is clear as $\alpha^*\in Y$. For the other direction,  let  $p_0=\l t_0,B_0\r\leq q=\l t, B\r\in G_{\alpha^*}$. For every $|t_0|<i\leq |t|$ $t(i)\in B\subseteq B_0$, by the property of $R$, we have that $\mu_i:=f(t(i))\smallfrown t(i)\in A^{p^*}$ such that $\mu_i(\alpha^*)=t(i)$. Now define $q'=\l f,B'\r$ as follows:
 $\dom(f)=Y$ and $$f=f^{p^*}{}^{\smallfrown}\mu_{|t_0|+1}{}^{\smallfrown}...{}^{\smallfrown}\mu_{|t|}.$$ In particular $f(\alpha^*)=t\geq f^{p_*}(\alpha^*)$. Also, let $B'=\{\mu\mid \mu(\alpha^*)\in B', \ f(\mu(\alpha^*))=\mu\restriction X\}$. We claim that $G\restriction Y=\{q'\mid q\in G_{\alpha^*}/p_0\}$. Indeed if $p\in G/p_*$ then $q=p\restriction\{\alpha^*\}\in G_{\alpha^*}$ and it is straightforward to check that $q'=p\restriction Y$.  It follows that $G\restriction Y$ is definable in $V[G_{\alpha^*}]$.
\end{proof}

By our assumption $U_{\alpha^*}$ is a $P-$point ultrafilter.
Now, Theorem~\ref{ThesisRes} applies, so
$$V[A]\models \cof(\kappa)=\omega.$$

\end{proof}

\subsection{ The original version }

The difference here from the forcing of the previous section is that the order $\leq^*$ is not $\kappa^+-$closed.
However, we will show that the forcing is still  $\kappa^+-$proper.

Assume for simplicity that $E$ is a $(\kappa, \kappa^{++})-$extender and the function $\nu\mapsto \nu^{++}$ represents $\kappa^{++}$ in the ultrapower.

Let $\calP_E$ be the forcing of \cite{Git-Mag} with $E$.

\begin{lemma}
Assume $p\in \calP_E$.
Let $N\preceq H_\chi$, for some $\chi$ large enough such that:

\begin{enumerate}
  \item $|N|=\kappa$,
  \item $N\supseteq {}^{\kappa>}N$,
  \item $E, \calP_E, p \in N$.
\end{enumerate}
Then there is $p^*\geq p$ which is $\l N, \calP_E\r-$generic.

\end{lemma}
\begin{proof}
Let $\l D_\nu \mid \nu<\kappa\r$ be an enumeration of all dense open subsets of $\calP_E$ which are in $N$.
Proceed by induction and define a $\leq^*-$increasing sequence $\l p_\nu\mid \nu<\kappa\r$ of extensions of $p$ such that,
for every $\nu<\kappa$,
\begin{enumerate}
  \item [(a)]$p_\nu\in N$.
  \item [(b)] $\min(A_\nu^0)>\nu$, where $A_\nu^0=\{\rho^0\mid \rho\in A_\nu\}$ is the projection of $A_\nu$ to the normal measure,
  \item [(c)] there is $k<\omega$ such that for every $\l \rho_1,..., \rho_k\r \in [A_\nu]^k$, $p_\nu{}^\frown \l \rho_1,..., \rho_k\r \in D_\nu$.
\end{enumerate}

It is natural now to move now to a coordinate $\eta$ which is above everything in $N$ and to take the diagonal intersection $\Delta^*$ of the pre-images of $A_\nu$'s according to the normal measure. 
However, in order to have the property (c) above, something more is needed.
Namely, we would like to have the following:

\begin{enumerate}
  \item[{(d)}] for every $\l \xi_1, ..., \xi_m\r \in [\min(A_\nu^0)]^{<\omega}$, if $p_\nu{}^\frown \l \xi_1, ..., \xi_m\r\in\calP_E$ then there is $k<\omega$ such that: $$\text{for every }\l \rho_1,..., \rho_k\r \in [A_\nu]^k, \ p_\nu{}^\frown \l \xi_1, ..., \xi_m\r^\frown \l \rho_1,..., \rho_k\r \in D_\nu.$$
\end{enumerate}

Given (d), as we will see, the idea above works fine.
Let us construct a sequence which satisfies the conditions (a)-(d).

Pick $p_0\in N$ such that $p_0\geq^* p$ and (d) is satisfied. To define $p_1$, use the strong Prikry property to pick a condition $p_1' \in N$, $p_1'\geq^*p_0$ and
$$\text{ there is } k<\omega \text{ such that for every } \l \rho_1,..., \rho_k\r \in [A_1']^k, p_1'{}^\frown \l \rho_1,..., \rho_k\r \in D_1.$$
 Let $\eta_0=\min((A_1')^0)$, by definition of $\pi_{\alpha,\kappa}$ it follows that $\eta_0$ is an inaccessible cardinal.
\\Let $\l \vec{\xi}_i \mid i<\eta_0\r $ be an enumeration of $[\eta_0]^{<\omega}$.
\\Define $\leq^*-$increasing sequence $\l q_i\mid i<\eta_0\r$.
\\Consider $p_1'{}^\frown \vec{\xi}_0$. If it does not extend $p_0$, then set $q_0=p_1'$.
Otherwise, pick (inside $N$) $r_0\geq^* p_1'{}^\frown \vec{\xi}_0$ such that
$$\text{ there is } k<\omega \text{ such that for every } \l \rho_1,..., \rho_k\r \in [A(r_0)]^k, r_0{}^\frown \l \rho_1,..., \rho_k\r \in D_1.$$
Let $q_0=\l f^{q_0},A^{q_0}\r$  be obtained from $r_0$ by  removing $\vec{\xi}_0$ from all coordinates which appear in $p_1'$ (and leaving at new ones), and then, adding a larger maximal coordinate. Namely, $\dom(f^{q_0})=\dom(f^{r_0})\cup\{\alpha_0\}$ where $\alpha_0$ is $\leq_E$ strictly above all the ordinals in $\dom(f^{r_0})$. Let $t$ be such that $\pi_{\alpha,\kappa}''t=f^{p'_1}(\kappa)$ and for every $\gamma\in \dom(f^{q_0})$
$$f^{q_0}(\gamma)=\begin{cases}f^{p'_1}(\gamma) & \gamma\in Supp(p'_1)\\
f^{r_0}(\gamma) &\gamma\in Supp(r_0)\setminus Supp(p'_1)\\ t & \gamma=\alpha_0\end{cases}.$$
Let $A^{q_0}=\pi_{\alpha_0,mc(r_0)}^{-1}[A^{r_0}]$. Then $q_0\in N$ and also $q_0\in\calP_E$. 
By shrinking $A^{q_0}$ a bit more (as in \cite[Lemma 3.10]{Gitik2010}) we secure condition (6), and $p_1'\leq^* q_0$.
\\Define $q_1$ in the exact same fashion only replacing $p_1'$ by $q_0$ and $\vec{\xi}_0$ by $\vec{\xi}_1$.
\\Continue similarly for every $i<\eta_0$, and finally,  let $q_{\eta_0}$ be a $\leq^*-$ extension of all $q_i$'s.
\\If $\eta_0=\min((A(q_{\eta_0}))^0)$, then set $p_1=q_{\eta_0}$.
Otherwise, let $\eta_1=\min((A(q_{\eta_0}))^0)$.
Repeat the process above with $\eta_1$ replacing $\eta_0$ and $q_{\eta_0}$ replacing $p_1'$.
Continuing in a similar fashion, we hope to reach some $\eta$ which is a fixed point, i.e., $\eta=\min((A(q_{\eta}))^0)$. However, we need to do this a bit more carefully at limit stages. Let us pick an elementary substructure $N'\prec V_{\mu}$ for sufficiently large $\mu$ of cardinality $\kappa^+$, closed under $\kappa-$sequences, including $p_1',p_0,\calP_E,E,...$. We can find some $\alpha<\kappa^{++}$ such that for every $p\in N'\cap\calP_E$ and every $\gamma\in\supp(p)$, $\gamma<_E\alpha$. Define a sequence of condition $\l q_{\eta_i}\mid i<\eta\r$ of conditions of $N'$. 

We start with $q_{\eta_0}$ which is already defined. Let $Y_0\in U_\alpha$ such that the commutativity requirement from Definition \ref{Original Version}(6) holds with respect to  $\supp(q_{\eta_0})$.
If $\eta_0=\min(Y_0^0)$ we are done. Otherwise, let $\eta_1=\min(Y_0^0)$ and construct $q_{\eta_1}$ in a similar fashion going over all possible $\vec{\xi}\in[\eta_1]^{<\omega}$, construct $Y_1\in U_\alpha$ to satisfy (6) with respect to $\supp(q_{\eta_1})$.
At a general successor step , we are given $\eta_i, q_{\eta_i}$, and $Y_i$.  Check if $\eta_i=\min(Y_i^0)$, if yes, stop the construction, set $p_1=q_{\eta_i}$ and we are done. Otherwise, let $\eta_{i+1}=\min(Y^0_{i})$, construct $q_{\eta_{i+1}}$ above $q_{\eta_i}$ as we did with $q_{\eta_0}$, going over all possible $\vec{\xi}\in[\eta_{i+1}]^{<\omega}$, then find $Y_{i+1}\in U_\alpha$ satisfying (6) with respect to $\supp(q_{\eta_{i+1}})$.   
At limit stages $\delta$  take $\eta_\delta=\sup_{i<\delta}\eta_i$, check if $\eta_\delta=\min((\cap_{i<\delta}Y_i)^0)$, if yes, stop the construction and consider the condition $p_1=q_{\eta_\delta}$ with maximal coordinate $\alpha$ , putting $\cap_{i<\delta}Y_i$ as his measure one set. Then $q_{\eta_\delta}$ will be as desired. Otherwise,  we find any $q_{\eta_\delta}\in N'$ above all the previous $q_{\eta_i}$, construct $Y_{\delta}\in U_\alpha$ with respect to $\supp(q_{\eta_\delta})$.   We can further require that $\pi_{\alpha,mc(q_{\eta_i})}{}'' Y_i\subseteq A(q_{\eta_i})$ and that $\min(A(q_{\eta_i})^0)>i$.

Assume toward a contradiction that no suitable $q_{\eta_\delta}$ was found and that the process goes all the way up to $\kappa$. Consider $Y^*=\Delta_{i<\kappa}^*Y_i\in U_\alpha$ and let $\mu$ be any limit point of $Y^*$. Consider step $\mu^0$ of the construction, we have $\eta_{\mu^0}=\sup_{i<\mu^0}\eta_i$. For every $i<\mu^0$, we have that $\mu\in Y_i$, hence $\mu\in\cap_{i<\mu^0}Y_i$ and $\mu^0\in (\cap_{i<\mu^0}Y_i)^0$, it follows that $\eta_{\mu^0}\geq\mu^0\geq\min((\cap_{i<\mu^0}Y_i)^0)\geq\eta_{\mu^0}$. This means that $\eta_{\mu^0}=\mu^0=\min((\cap_{i<\mu^0}Y_i)^0)$ which indicates that the construction should have terminated at step $\mu_0$, contradiction. \\
We conclude that $p_1$ is defined. The further construction of $p_\nu$'s is similar, exploiting the $\kappa-$closure of $\leq^*$.

Pick now some $\alpha\geq_E \beta,$ for every $\beta\in N\cap \dom(E)$ which exists since $|N|=\kappa$. Set
$$A=\Delta^*_{\nu<\kappa}\tilde{A}(p_\nu)=\{\rho<\kappa\mid \forall \nu<\rho^0 (\rho\in \tilde{A}(p(\nu)))\},$$
where $\tilde{A}(p_\nu)$ is the pre-image of ${A}(p_\nu)$ under the projection from $\alpha$ to $mc(p_\nu)$.
Define a condition $p^*=\l f^*,A^*\r$ from the sequence $\l p_\nu \mid \nu<\kappa\r$ as follows: $\supp(p^*)=\cup_{\nu<\kappa}\supp(p_\nu)\cup\{\alpha\}$, from the way we defined $p_\nu$ there is no problem defining $f^*=\cup_{\nu<\kappa}f^{p_\nu}\cup\{\l\alpha,t\r\}$ where $t$ is any sequence such that $\pi_{\alpha,\kappa}''t=f^*(\kappa)$. Then we take $A^*=A$. It follows that $p^*\in \calP_E$, and it has the property that for every $\nu<\kappa$ and any sequence $$\xi_1<..,\xi_k<\min(A^0_\nu)\leq\xi_{k+1}<...<\xi_n$$ of ordinals from $A$, $p_\nu^{\smallfrown}\l\xi_1,...,\xi_n\r\leq p^*{}^{\smallfrown}\l \xi_1,..,\xi_n\r$. \footnote{Although $\xi_1,..,\xi_k\notin A_\nu$, the condition $p_\nu^{\smallfrown}\l\xi_1,...,\xi_n\r$ is a legitimate condition which is simply not above $p_\nu$.} Let us argue that it is $\l N,\calP_E\r-$generic. Let $G$ be generic with $p^*\in G$. we need to prove that $G\cap N\cap D_\nu\neq\emptyset$ for every $\nu<\kappa$. By density, pick any $p^{\smallfrown}\l\xi_1,...,\xi_{k_1}\r\leq^*q\in D_\nu\cap G$, and let $m$ be such that $\xi_1,..,\xi_m<\min(A(p_\nu))\leq \xi_{m+1}<...<\xi_{k_1}$. By condition (d), there is $k_2$ such that any $\l\nu_1,...,\nu_{k_2}\r\in[A_\nu]^{k_2}$ extension  $p_\nu^{\smallfrown}\l\xi_1,..,\xi_m\r^{\smallfrown}\l \nu_1,...,\nu_{k_2}\r\in D_\nu$. If necessary, extend $q$ to $$q{}^{\smallfrown}\l \xi_{k_1+1},...,\xi_{k_1+k_2}\r\in G\cap D_\nu,$$ and suppose without loss of generality that $k_1\geq m+k_2$.
Since $\nu<\min(A(p_\nu)^0)\leq \xi_{m+1}$, by definition of $\pi_{\alpha,\kappa}$, it follows that $\nu<\xi_{m+1}^0$, and by diagonal intersection, $\xi_{m+1},...,\xi_{k_1}\in A_\nu$. It follows that  $$p_{\nu}^{\smallfrown}\l\xi_1,..,\xi_m{}\r^{\smallfrown}\l\xi_{m+1},...,\xi_{m+k}\r\in D_\nu.$$ Also, $p_{\nu}^{\smallfrown}\l\xi_1,..,\xi_m{}\r^{\smallfrown}\l\xi_{m+1},...,\xi_{m+k}\r\leq q$ hence in $G$.   Hence $$p_{\nu}^{\smallfrown}\l\xi_1,..,\xi_m\r{}^{\smallfrown}\l\xi_{m+1},...,\xi_{m+k}\r\in G\cap D_\nu\cap N$$ as wanted.

\end{proof}

Now, as in the previous section the following holds.

\begin{theorem}\label{Woodin-answer}
Let $G\subseteq \calP_E$ be a generic. Suppose that $A\in V[G]\setminus V$ is a subset of $\kappa$.
Then $\kappa$ changes its cofinality to $\omega$ in $V[A]$.

\end{theorem}

\subsection{ The Merimovich version }

The previous subsection implies in particular that $\calP_E$ and  $\mathbb{P}_E$ with $P-$points cannot add $\kappa^{++}$-many mutually generic Cohen functions. In this subsection, we will provide the general argument that the Extender-based Prikry forcing $\mathbb{P}_E$ cannot add $\kappa^{++}$-many distinct subsets of $\kappa$ which preserves even the regularity of $\kappa$.

\begin{theorem}\label{CarmiVersion theorem}
Assume GCH\footnote{ $2^\kappa=\kappa^+$ is enough, since $\kappa$ is a measurable, and so, $2^\nu=\nu^+$ on  relevant sets.} and let $E$ an extender over $\kappa$.
Let $G$ be a generic subset of $\mathbb{P}_E$ and let
$\l {A}_\alpha\mid \alpha<\kappa^{++} \r$ be different subsets of $\kappa$ in $V[G]$.
Then there is $I\subseteq \kappa^{++}, I\in V, |I|=\kappa$ such that $\kappa$ is a singular cardinal of cofinality $\omega$ in
$V[\l {A}_\alpha\mid \alpha\in I\r]$.
In particular, there is no intermediate model of $V[G]$ where $\kappa$ is measurable and $2^\kappa>\kappa^+$.
\end{theorem}
\begin{proof}
Let $\l \lusim{A}_\alpha\mid \alpha<\kappa^{++} \r$ be $\mathbb{P}_E-$names of  subsets of $\kappa$.
We will confuse them sometimes with their characteristic functions. Work in $V$, for every $\alpha<\kappa^{++}$, let $N_\alpha$ be an elementary submodel of $H_\theta$ of cardinality $\kappa$ such that
${}^{\kappa>}N_\alpha\subseteq N_\alpha$, $E, P_E, \alpha, \l \lusim{A}_\alpha\mid \alpha<\kappa^{++} \r\in N_\alpha$.
\\Let $f_\alpha\in \mathbb{P}^*_E$ be $N_\alpha-$completely generic, i.e. $f_\alpha^\frown \l \vec{\nu}_1,...,\vec{\nu}_n\r\in \mathbb{P}^*_E$ is $N_\alpha-$generic.

Using $\Delta-$system-like arguments, we can assume that $\l f_\alpha \mid \alpha<\kappa^{++}\r$ form a $\Delta-$system such that for every $\alpha, \beta<\kappa^{++}$,
\begin{enumerate}
  \item $otp(\dom(f_\alpha))=otp(\dom(f_\beta))$, and the order isomorphism between $\dom(f_\alpha)$ and $\dom(f_\beta)$, $\sigma_{\alpha,\beta}$ is constant on the intersection $\dom(f_{\alpha})\cap \dom(f_{\beta})$.
  \item for every $\rho\in \dom(f_\alpha)$, $f_\alpha(\rho)=f_\beta(\sigma_{\alpha\beta}(\rho))$.
\end{enumerate}

Attach to each $\alpha<\kappa^+$ associate an $E(\dom(f_\alpha))-$large tree $T_\alpha$. Define $T_\alpha$ level by level as follows.
Set $Lev_1(T_\alpha)=S_\alpha^0\cup S_\alpha^1$, where


 \begin{enumerate}
        \item for every $\vec{\nu}\in S_\alpha^0$, $\dom(\vec{\nu})$ contains elements in $\dom(f_\alpha)\setminus \dom(f_0)$, if $\alpha>0$,
        \item if $\alpha=0$, then $S_\alpha^0= S_\alpha^1$,
        \item 
        $S_\alpha^1=
        \{\vec{\nu}\mid \vec{\nu} \text{ is an increasing partial function from } \dom(f_0)\cap \dom(f_\alpha) \text{ to } \kappa\}$,
if $\alpha>0$,
        
          \item  for every $ \vec{\nu} \in S_\alpha^0$ the following holds:

           $\l f_\alpha{}^\frown\vec{\nu},B_{\vec{\nu}}\r$
 decides $\lusim{A}_\alpha\cap \vec{\nu}(\kappa)$ for some $E(\dom(f_\alpha))$-tree $B_{\vec{\nu}}$ and such that the decision depends only on
  $\vec{\nu}(\kappa)$.
  \end{enumerate}
In order to find such a tree, we will use the fact that $f_\alpha\in \mathbb{P}_E^*$ is $N_\alpha$-generic, and the set 
$$E=\{f\mid \exists B.\l f,B\r\text{ decides }\lusim{A}_\alpha\cap\vec{\nu}(\kappa)\}$$ being dense open in $\mathbb{P}_E^*$. This implies the existence of a $E(\dom(f_\alpha))-$tree $B_{\vec{\nu}}$   such that
     $$\l f_\alpha{}^\frown  \vec{\nu}, B_{\vec{\nu}}\r \text{ decides }   \lusim{A}_\alpha\cap \vec{\nu}(\kappa).$$
Next, in order to make the decision  to depend only on $\vec{\nu}(\kappa)$, we use ineffability:
Suppose that $\l f_\alpha{}^\frown\vec{\nu}, B_{\vec{\nu}}\r$ forces that $\lusim{A}_\alpha\cap\vec{\nu}(\kappa)=A_\alpha(\vec{\nu})$. Let $g$ be the function $g(\vec{\nu})=A_\alpha(\vec{\nu})$. It follows that:
$$X_\alpha(\l\r):=j(g)( (j\restriction\dom(f_\alpha))^{-1} )\subseteq \kappa.$$
Also, since $crit(j)=\kappa$, it follows that $j(X_\alpha(\l\r))\cap\kappa=X_\alpha(\l\r)$. 
Combine this together with the fact that:
$$j''\dom(f_\alpha))\text{ contains elements not in }j(\dom(f_0))$$
to find a $E(\dom(f_\alpha))$-large set  $S^0_\alpha$, such that $(1)$ holds and for all  $\vec{\nu}\in S^*_\alpha$, $$A_\alpha(\vec{\nu})=X_\alpha(\l\r)\cap \vec{\nu}(\kappa)$$
Finally, we let $Lev_1(T_{\alpha})=S^0_\alpha\cup S^1_\alpha$.
Note that if $\alpha>0$, then   $S^0_\alpha$ and $S^1_\alpha$ are disjoint and therefore    $S_\alpha^1\not \in E(\dom(f_\alpha))$. Define now the next level of $T_\alpha$.
   So let $\vec{\rho}  \in Lev_1(T_\alpha)$ and set $\Succ_{T_\alpha}(\vec{\rho})=S_{\alpha \vec{\rho}}^0\cup S_{\alpha \vec{\rho}}^1$, where

 \begin{enumerate}
 \item for every $\vec{\nu}\in  S_{\alpha \vec{\rho}}^0\cup S_{\alpha \vec{\rho}}^1$, $\vec{\nu}(\kappa)> \sup(\rng(\vec{\rho}))$.
 \item $S_{\alpha \vec{\rho}}^0\subseteq Suc_{B_{\vec{\rho}}}(\vec{\rho})$.
        \item if $\alpha>0$, then for every $\vec{\nu}\in S_{\alpha \vec{\rho}}^0$, $\dom(\vec{\nu})$ contains elements in $\dom(f_\alpha)\setminus \dom(f_0)$, 
        \item if $\alpha=0$, then $S_{\alpha \vec{\rho}}^0= S_{\alpha \vec{\rho}}^1$,
        \item if $\vec{\rho}\in S_\alpha^0$ and $\alpha>0$, then $S_{\alpha \vec{\rho}}^1=\emptyset$,

        \item if $\vec{\rho}\in S_\alpha^1$ and $\alpha>0$, then $S_{\alpha \vec{\rho}}^1
        =\{\vec{\nu}\mid \vec{\nu} \text{ is an increasing partial function from } \\\dom(f_0)\cap \dom(f_\alpha) \text{ to } \kappa,
        \vec{\nu}(\kappa)>\sup(\rng(\vec{\rho}))\}$,

          \item  for every $ \vec{\nu} \in S_{\alpha \vec{\rho}}^0$ the following holds:

           $\l f_\alpha{}^\frown \l \vec{\rho}, \vec{\nu}\r, B_{\vec{\rho} \vec{\nu}}\r$
 decides $\lusim{A}_\alpha\cap \vec{\nu}(\kappa)$ and the decision depends only on
  $\vec{\rho}{}^\frown\vec{\nu}(\kappa)$,\\ for some $E(\dom(f_\alpha))-$tree $ B_{\vec{\rho} \vec{\nu}}$ , which is a subtree of $B_{\vec{\rho}}$.

  \end{enumerate}
The further levels are defined in the same fashion. Denote by $T_\alpha^0$ the tree $T_\alpha$ with $S^1_{\alpha \vec{\nu}_1,...,\vec{\nu}_n}$ removed from $\Succ_{T_\alpha}(\l\vec{\nu}_1,...,\vec{\nu}_n\r)$\footnote{Even if $\l\vec{\nu}_1,...,\vec{\nu}_n\r\in T_\alpha\setminus T^0_\alpha$ the set $\Succ_{T^0_\alpha}(\l\vec{\nu}_1,...,\vec{\nu}_n\r)$ is still defined.}. Clearly, $T_\alpha^0$ is still  $E(\dom(f_\alpha))-$tree.

The tree $T^0_{\alpha}$ has the property that for every $\l\vec{\nu}_1,...,\vec{\nu}_n\r\in T_{\alpha}$, and every $\vec{\nu}\in\Succ_{T^0_\alpha}(\l\vec{\nu}_1,...,\vec{\nu}_n\r)$, item $(2)$ above ensures that $(T^0_\alpha)_{\l\vec{\nu}_1,...,\vec{\nu}_n,\vec{\nu}\r}\subseteq B_{\l\vec{\nu}_1,...,\vec{\nu}_n,\vec{\nu}\r}$ and by item $(7)$ we obtain $$(*) \ \ \l f_{\alpha}^{\smallfrown}\l\vec{\nu}_1,...,\vec{\nu}_n,\vec{\nu}\r,(T^0_{\alpha})_{\l\vec{\nu}_1,...,\vec{\nu}_n,\vec{\nu}\r}\r\Vdash X_\alpha(\l \vec{\nu}_1,...,\vec{\nu}_n\r)\cap\vec{\nu}(\kappa)=\lusim{A}_\alpha\cap\vec{\nu}(\kappa)$$





By shrinking if necessary, we can assume that the trees are isomorphic under the obvious isomorphism induced by the $\Delta-$system.Moreover, by $GCH$, there are only $\kappa^+$-many decisions possible on a fixed isomorphism-type of trees and therefore we can stabilize the decisions so they do not depend on a particular choice of $\alpha$.
Let us now take $\kappa$ elements and combine them into a single condition. Namely, we consider $\l\l f_\alpha, T_\alpha\r \mid 0<\alpha<\kappa\r$ and define a condition $\l f^*,T^*\r$ as follows:
\\Let $f^*=\bigcup_{0<\alpha<\kappa}f_\alpha$. 
Define  a $E(\dom(f^*))-$tree $ T^*$. It will be
 a sort of a diagonal intersection of $T_\alpha, 0<\alpha<\kappa$. Set 

$$X=\{\vec{\nu}\mid \vec{\nu} \text{ is an increasing partial function from } \dom(f^*) \text{ to } \kappa, $$$$
\dom(\vec{\nu})\subseteq \bigcup_{\xi<\vec{\nu}(\kappa)}\dom(f_\xi), (\forall \xi<\vec{\nu}(\kappa))
|\dom(\vec{\nu})\cap \dom(f_\xi)|=\vec{\nu}(\kappa)\}.$$
To see that $X\in E(\dom(f^*))$, note that $$\dom( (j\restriction \dom(f^*))^{-1})=j''\dom(f^*)\subseteq \cup_{\xi<\kappa}\dom(j(f_\xi)).$$ Also, for every $\xi<\kappa$, $|j''\dom(f^*)\cap \dom(j(f_\xi))|=|j''\dom(f_\xi)|=|\dom(f_\xi)|$ and since $f_\xi$ is completely generic we conclude that this cardinality must be $\kappa$. Hence $(j\restriction \dom(f^*))^{-1}\in j(X)$. 
Define the first level of the tree
$$Lev_1(T^*)=\Succ_{T^*}(\l\r):=X \cap \Delta^*_{\xi<\kappa}\pi^{-1}_{\dom(f^*) \dom(f_\xi)} \Succ_{T^0_{\xi}}(\l\r).$$
Let now $\vec{\rho} \in Lev_1(T^*)$, and define $\Succ_{T^*}(\vec{\rho})$.
As above, we consider first the set
$$X_{\vec{\rho}}=\{\vec{\nu}\mid \vec{\nu} \text{ is an increasing partial function from } \dom(f^*) \text{ to } \kappa, \vec{\nu}(\kappa)>\sup(\rng(\vec{\rho})),$$$$
\dom(\vec{\nu})\subseteq \bigcup_{\xi<\vec{\nu}(\kappa)}\dom(f_\xi), (\forall \xi<\vec{\nu}(\kappa))
|\dom(\vec{\nu})\cap \dom(f_\xi)|=\vec{\nu}(\kappa)\}.$$

Clearly, $X_{\vec{\rho}}\in E(\dom(f^*))$. Let
$$\Succ_{T^*}(\vec{\rho})=X_{\vec{\rho}} \cap \Delta^*_{\xi<\kappa}\pi^{-1}_{\dom(f^*) \dom(f_\xi)} \Succ_{T_\xi^0}(\vec{\rho}\restriction \dom(f_\xi)).$$
Continue to define $T^*$ in a similar fashion. We need to ensure that for every $\xi<\kappa$, $\Succ_{T_\xi^0}(\vec{\rho}\restriction \dom(f_\xi))$ is well defined. Namely:
\begin{claim}\label{level1claim}
 For every $\xi<\kappa$, $\vec{\rho}\restriction \dom(f_\xi)\in Lev_1(T_\xi)$. Moreover,  $\xi<\vec{\rho}(\kappa)$ iff $\vec{\rho}\restriction \dom(f_\xi)\in Lev_1(T^0_\xi)$.
\end{claim}  
\begin{proof}[Proof of claim \ref{level1claim}:]
For every $\xi<\vec{\rho}(\kappa)$,  we have $$\vec{\rho}\in \pi^{-1}_{\dom(f^*)\dom(f_\xi)}(Lev_1(T^0_\xi))$$ and therefore $\vec{\rho}\restriction \dom(f_\xi)\in  Lev_1(T^0_\xi)$.
If $\xi\geq \vec{\rho}(\kappa)$, then since $\vec{\rho}\in X$, $\dom(\vec{\rho})\cap\dom(f_\xi)=\dom(\vec{\rho})\cap\dom(f_0)$ and therefore $\vec{\rho}\in S^1_\alpha=Lev_1(T_\alpha)$.
\end{proof}

 $Lev_1(T^*)$ has the property that for all $\vec{\rho}\in Lev_1(T^*)$ and $\alpha<\vec{\rho}(\kappa)$,
 $$\l f^*{}^\frown\vec{\rho}, (T^*)_{\vec{\rho}}\r\geq^* \l f_\alpha{}^\frown\vec{\rho} \restriction \dom(f_\alpha), (T^0_{\alpha})_{\vec{\rho} \restriction \dom(f_\alpha)}\r.$$
 Hence, by $(*)$, $\l f^*{}^\frown\vec{\rho}, (T^*)_{\vec{\rho}}\r$ also forces $X_{\alpha}(\l\r)\cap\vec{\rho}(\kappa)=\lusim{A}_\alpha\cap \vec{\rho}(\kappa)$.
In addition, if we have $\alpha,\beta< \vec{\rho}(\kappa)$, then $\lusim{A}_\alpha\cap \vec{\rho}(\kappa)$, $\lusim{A}_\beta\cap \vec{\rho}(\kappa)$ depends only on $(\vec{\rho}\restriction\dom(f_\alpha))(\kappa)=\vec{\rho}(\kappa)=(\vec{\rho}\restriction\dom(f_\beta))(\kappa)$ and since the isomorphism $\sigma_{\alpha,\beta}$ fixes $\kappa$ (as $\kappa\in\dom(f_\alpha)\cap\dom(f_\beta)$) it follows that  $\lusim{A}_\beta\cap \vec{\rho}(\kappa),\lusim{A}_\alpha\cap \vec{\rho}(\kappa)$
are decided to be the same set. 

Next consider $\l\vec{\rho},\vec{\nu}\r\in T^*$, as in claim \ref{level1claim}, we have that for all $\alpha<\vec{\nu}(\kappa)$,
$$(**) \l f^{*}{}^{\smallfrown}\l\vec{\rho},\vec{\nu}\r,(T^*)_{\l\vec{\rho},\vec{\nu}\r}\r\geq\l f_\alpha^{\smallfrown}\l \vec{\rho}\restriction \dom(f_\alpha),\vec{\nu}\restriction \dom(f_\alpha)\r,(T^0_\alpha)_{\l \vec{\rho}\restriction \dom(f_\alpha),\vec{\nu}\restriction \dom(f_\alpha)\r}\r$$
However, since the decision about $\lusim{A}_\alpha\cap\vec{\nu}(\kappa)$ depends now on $\vec{\rho}^{\smallfrown}\vec{\nu}(\kappa)$, then if $\alpha$ or $\beta$ are below $\vec{\rho}(\kappa)$, then $\rho\restriction\dom(f_\alpha)$ might include in its domain ordinals which are moved under the isomorphism $\sigma_{\alpha,\beta}$ and therefore we are not guaranteed that the decision about $\lusim{A}_\alpha\cap\vec{\nu}(\kappa),\lusim{A}_\alpha\cap\vec{\nu}(\kappa)$ is the same (up to $\vec{\rho}(\kappa)$ it is still the same decision). However, if both $\alpha,\beta\in [\vec{\rho}(\kappa),\vec{\nu}(\kappa))$ we have the following claim:
\begin{claim}\label{level2}
 If $\alpha,\beta\in[\vec{\rho}(\kappa),\vec{\nu}(\kappa))$ then $\l f^*{}^{\smallfrown}\l\vec{\rho},\vec{\nu}\r,(T^*)_{\l\vec{\rho},\vec{\nu}\r}\r$ decides the values of $\lusim{A}_\alpha\cap\vec{\nu}(\kappa)$ and $\lusim{A}_\beta\cap\vec{\nu}(\kappa)$ to be the same.
\end{claim}
\begin{proof}[Proof of claim \ref{level2}:]
By definition, since $\vec{\rho}\in X$ and $\alpha,\beta\geq \vec{\rho}(\kappa)$, 
 $\vec{\rho}\restriction\dom(f_\alpha)=\vec{\rho}\restriction\dom(f_\beta)$ and $\dom(\vec{\rho}\restriction \dom(f_\alpha))\subseteq\dom(f_\alpha)\cap\dom(f_0)$. Since the isomorphism $\sigma_{\alpha,\beta}$ fixes the kernel of the $\Delta$-system, we have that the decision of $$\l f_\alpha^{\smallfrown}\l \vec{\rho}\restriction \dom(f_\alpha),\vec{\nu}\restriction \dom(f_\alpha)\r,(T^0_\alpha)_{\l \vec{\rho}\restriction \dom(f_\alpha),\vec{\nu}\restriction \dom(f_\alpha)\r}\r$$
 about $\lusim{A}_\alpha\cap\vec{\nu}(\kappa)$ and the decision of
 $$\l f_\beta^{\smallfrown}\l \vec{\rho}\restriction \dom(f_\beta),\vec{\nu}\restriction \dom(f_\beta)\r,(T^0_\beta)_{\l \vec{\rho}\restriction \dom(f_\beta),\vec{\nu}\restriction \dom(f_\beta)\r}\r$$
 about $\lusim{A}_\beta\cap\vec{\nu}(\kappa)$ is the same. By $(**)$, the condition $\l f^*{}^{\smallfrown}\l\vec{\rho},\vec{\nu}\r,(T^*)_{\vec{\rho},\vec{\nu}}\r$ decides the values the same way.
\end{proof}
Similar properties persists if we move to higher levels of $T^*$.

Using density arguments we can assume that such defined condition $\l f^*,T^*\r$ is in the generic subset $G$ of $\mathbb{P}_E$.
Denote by $\l \kappa_n \mid n<\omega\r$ the Prikry sequence for the normal measure $E_\kappa$.

It follows that the sets $\l A_\alpha\mid\alpha<\kappa\r$ have the following property in $V[G]$:
$$(**)\ \ \forall n<\omega.\forall \alpha,\beta\in[\kappa_{n-1},\kappa_n). A_\alpha\cap \kappa_n=A_\beta\cap\kappa_n$$

Now, let us turn to the model $M^*=V[\l A_\alpha \mid \alpha<\kappa\r]$ and prove that $cf^{M^*}(\kappa)=\omega$.
Let us define in $M^*$ an $\omega$-sequence $\l\zeta_n\mid n<\omega\r$ as follows:

First, let $\zeta'_0$ be the least such that for some for some $\alpha,\beta<\kappa$, $A_\alpha\cap\zeta_0'\neq A_\beta\cap\zeta_0'$. There exists such $\zeta_0'$ since the sets in the sequence $\l A_\alpha\mid \alpha<\kappa\r$ are distinct. Let $\zeta''_0$ be the least such that for some $\alpha<\zeta''_0$, $A_\alpha\cap\zeta_0'\neq A_{\zeta_0''}\cap\zeta_0'$. Define $\zeta_0=\max(\zeta_0',\zeta_0'')$
\begin{claim}\label{inductionbase}
$\zeta_0\geq\kappa_0$
\end{claim}
\begin{proof}[Proof of claim \ref{inductionbase}:]
If $\zeta_0'\geq\kappa_0$ then we are done. Otherwise, suppose $\zeta_0'\leq\kappa_0$, then by $(**)$ for every $\alpha<\beta<\kappa_0$, we have $A_\alpha\cap\zeta_0'=A_\beta\cap\zeta'_0$. Hence by the definition of $\zeta''_0$, we have $\zeta''_0\geq\kappa_0$ and also $\zeta_0\geq\kappa_0$
\end{proof}
Suppose that $\zeta_n<\kappa$ was defined. Then the sequence $\l A_\alpha\mid \zeta_n<\alpha<\kappa\r$ consist of $\kappa$-many distinct subsets of $\kappa$. Since $\kappa$ is strong limit in $V[G]$, $2^{\zeta_n}<\kappa$, hence there must be $\zeta_n<\alpha<\beta<\kappa$ such that $A_\alpha\setminus\zeta_n+1\neq A_\beta\setminus\zeta_n+1$. Let $\zeta'_{n+1}$ be the minimal such that for some $\zeta_n<\alpha<\beta<\kappa$, $A_\alpha\cap\zeta'_{n+1}=A_{\beta}\cap\zeta'_{n+1}$. Finally, let $\zeta_n<\zeta''_{n+1}$ be the minimal such that for some $\alpha<\zeta''_{n+1}$, $A_\alpha\cap\zeta'_{n+1}\neq A_{\zeta''_{n+1}}\cap\zeta'_{n+1}$ and $\zeta_{n+1}=\max(\zeta'_{n+1},\zeta''_{n+1})$. To conclude that $cf^{M^*}(\kappa)=\omega$ is suffices to prove the following lemma:
\begin{claim}\label{fullinduction}
For every $n<\omega$, $\zeta_n\geq\kappa_n$.
\end{claim}
\begin{proof}[Proof of claim \ref{fullinduction}:]
By induction, for $n=0$ this is just the previous claim. Suppose that $\zeta_n\geq\kappa_n$, and toward a contradiction suppose that $\zeta_{n+1}<\kappa_{n+1}$. Then by definition, there is $\alpha$, such that $\kappa_n\leq \zeta_n<\alpha<\zeta''_{n+1}<\kappa_{n+1}$ such that $A_\alpha\cap \zeta'_{n+1}\neq A_{\zeta''_{n+1}}\cap\zeta'_{n+1}$. However, since $\zeta'_{n+1}<\kappa_{n+1}$ we reached a contradiction to $(**)$, since we found two indices $\alpha,\beta\in[\kappa_n,\kappa_{n+1})$ such that $A_\alpha\cap\kappa_{n+1}\neq A_\beta\cap\kappa_{n+1}$.
\end{proof}

The sequence $\l \zeta_n \mid n<\omega \r$ will be a cofinal sequence in $\kappa$ which \\belongs to $V[\l A_\alpha \mid \alpha<\kappa\r]$.
\end{proof}

 It turns out that $\mathbb{P}_E$ can add $\kappa^+$-many mutually generic over $V$ Cohen functions, for specially chosen extender $E$.

\begin{theorem}
 Assume $GCH$ and suppose that $E$ is a $(\kappa,\kappa^{++})$-extender.  Then after the preparation of Theorem \ref{main theorem}, there exists an extender $E'$ such that $\mathbb{P}_{E'}$  adds $\kappa^+$ mutually generic over $V$ Cohen functions.

 \end{theorem}
\begin{proof}
Let $j=j_E:V\rightarrow M$ be the natural ultrapower by the $(\kappa,\kappa^{++})-$ extender $E$, then $j(\kappa)>\kappa^{++}$, $crit(j)=\kappa$ and ${}^\kappa M\subseteq M$. Recall that the preparation forcing in Theorem \ref{main theorem} is an Easton support iteration $$\l \calP_{\alpha},\lusim{Q}_\beta\mid \alpha\leq\kappa+1,\beta\leq\kappa\r$$ such that $\lusim{Q}_\beta$ is trivial unless $\beta$ is inaccessible in which case if $\beta<\kappa$ then $\lusim{Q}_\beta$ is a $\calP_\beta$-name for $\Lot(\Add(\beta,\beta^+),\Add(\beta,\beta^+)^2)$. At $\kappa$, $\lusim{Q}_\kappa$ is a name for $\Add(\kappa,\kappa^+)$. Let $G_\kappa*g_\kappa$ be $V-$generic for $P_\kappa*\lusim{Q}_\kappa$. In $V[G_{\kappa}*g_\kappa]$ we can construct an $M-$generic filter for $j(\calP_\kappa*\lusim{Q}_\kappa)$ by taking $G_\kappa*g_\kappa$ to be the generic up to $\kappa$, including $\kappa$ and choosing that the lottery sum forces $\Add(\kappa,\kappa^+)$ (this forcing is the same in $V[G_\kappa]$ and $M[G_\kappa]$ since $(\kappa^+)^{M[G_\kappa]}=\kappa^+$ and $M[G_\kappa]$ is closed under $\kappa$-sequences of $V[G_\kappa]$). Above $\kappa$ we have sufficient closure, from the point of view of $V[G_\kappa*g_\kappa]$, and by $GCH$ there are not too many dense open subsets of the tail forcing $\calP_{(\kappa,j(\kappa)]}$ to meet, hence the
 embedding $j$ lifts to $$j\subseteq j^*:V[G_{\kappa}*g_\kappa]\rightarrow M[j(G_{\kappa})*j(g_\kappa)].$$ Since the cardinals in all the models are preserved, it follows that (\cite[Proposition 8.4]{CummingsHand}) $$(\kappa^{++})^{M[j(G_{\kappa})*j(g_\kappa)]}=\kappa^{++}<j(\kappa)\text{ and  }{}^{\kappa}M[j(G_{\kappa})*j(g_\kappa)]\subseteq M[j(G_{\kappa})*j(g_\kappa)].$$ So in $V[G_\kappa*g_\kappa]$ the extender $E$ extender to an extender $E'=\l E'_a\mid a\in[\kappa]^{<\omega}\r$ defined by $E'_a=\{X\subseteq\kappa^{|a|}\mid a\in j^*(X)\}$.

Let $W$ be the non-Galvin, $\kappa-$complete ultrafilter over $\kappa$  with preparation for adding $\kappa^+$-many Cohens (See Theorem \ref{Proof for W is Cohen}).

Combine $E',W$ together as follows.
First take an ultrapower with $E'$. Let $j_{E'}:V \to M_{E'}$ be the corresponding embedding. Denote $j_{E'}(\kappa)$ by $\kappa_1$ and let $W'=j_{E'}(W)$.
Then take an ultrapower of $M_{E'}$ with $W'$. Let $j_{W'}:M_{E'} \to M$ be the corresponding embedding.
\\Consider $j_*=j_{W'}\circ j_{E'}:V\to M$. Let $E^*$ be the derived $(\kappa,\lambda)$-extender for some $\kappa_1<\lambda\leq j_*(\kappa)$.
\\Note that $E^*(\kappa_1)=W$, since for any $X\subseteq \kappa$,
$$X\in E^*(\kappa_1) \Leftrightarrow \kappa_1 \in j_*(X)\Leftrightarrow\kappa_1\in j_{W'}( j_{E'}(X))\Leftrightarrow j_{E'}(X)\in W'$$ $$\Leftrightarrow j_{E'}(X)\in j_{E'}(W)\Leftrightarrow X\in W.$$

The Prikry forcing with $W$ adds $\kappa^+-$many Cohens over $V$. This forcing is a part of $\mathbb{P}_{E^*}$, since $W$ appears as one of the measures of $E^*$, which implies the Theorem.

\end{proof}

\subsection{ Cohen subsets of $\kappa^+$ }

Let us argue here that both versions add $\kappa^{++}-$many (or $\lambda-$many if the extender has $\lambda$ generators for a regular $\lambda>\kappa$) Cohen subsets of $\kappa^+$ mutually generic over $V$.

Start with $\calP_E$ of \cite{Git-Mag}.

\begin{theorem}\label{thm 15}
 Let $G\subseteq \calP_E$ be a generic.
 Then in $V[G]$ there is a sequence $\l Z_\xi \mid \xi<\kappa^{++}\r$ of mutually generic over $V$ Cohen subsets of $\kappa^+$.

 \end{theorem}
\begin{proof}
Let  $\l t_\alpha \mid \alpha<\kappa^{++}\r$ be the Prikry sequences added by $G$.

Split, in $V$, $\kappa^{++} $ into disjoint intervals $\l I_\xi\mid \xi<\kappa^{++}\r$ order type of each $\kappa^+$.
Denote by $\sigma_\xi$ the order isomorphism between $I_\xi$ and $\kappa^+$.
\\Now, in $V[G]$, set
$$Z_\xi=\{\sigma_\xi(\alpha)\in I_\xi \mid t_\alpha(0) \text{ is even }\}.$$
Let us argue that such a sequence is as desired.

Work in $V$. Let $p\in \calP$ and $D$ be a dense open subset of $\Add(\kappa^+, \kappa^{++})$.
\\Let us find $q\geq p$ such that

$$q\Vdash \l \lusim{Z}_\xi \mid \xi<\kappa^{++}\r \text{ extends an element of } D.$$

Extend first $p$ to some $r$ such that for every $\gamma\in \supp(r), r^\gamma $ is not equal to the empty sequence.
Now, using $I_\xi,\sigma_\xi$'s turn $\l r^\gamma(0) \mid \gamma\in \supp(r)\r$ into a condition in $\Add(\kappa^+, \kappa^{++})$.
Extend it to one in $D$ and move back to $\calP$ using $I_\xi,\sigma_\xi^{-1}$'s. 
Finally, turn the result into a condition $q$ in $\calP$ stronger than $r$.
It will be as desired.

\end{proof}

The situation in the case of the Merimovich version is very similar:

\begin{theorem}
 Let $G\subseteq \mathbb{P}_E$ be a generic.
 Then in $V[G]$ there is a sequence $\l Z_\xi \mid \xi<\kappa^{++}\r$ of mutually generic over $V$ Cohen subsets of $\kappa^+$.

 \end{theorem}
\begin{proof}
Proceed as in \ref{thm 15} and define $\l Z_\xi \mid \xi<\kappa^{++}\r$.

Work in $V$. Let $p\in \calP$ and $D$ be a dense open subset of $\Add(\kappa^+, \kappa^{++})$.
\\Let us find $q\geq p$ such that

$$q\Vdash \l \lusim{Z}_\xi \mid \xi<\kappa^{++}\r \text{ extends an element of } D.$$

A slight difference here is that the support of $p=\l f, T\r$, i.e., $\dom(f)$ may have $\kappa$ many places $\gamma$ with $f(\gamma)=\l\r$.
\\As a result, for such $\gamma$, $t_\gamma(0)$ will be determined only after an element of the corresponding set of measure one is picked,
and there are $\kappa-$many such $\gamma$'s.
\\However, we do not need the exact value of $t_\gamma(0)$, but rather to know whether it is even or odd.
This is determined (on a set of measure one) by $\gamma$ itself. Namely, in this situation, $t_\gamma(0)$ will be even iff $\gamma$ is even.  
\\The rest of the argument is as in \ref{thm 15}.

\end{proof}
\section{Acknowledgement}
The authors would like to thank the referee for their insightful remarks and the improvement of the content of the paper. Also, to Mohammad Golshani for his crucial correction in the first draft of the paper. Finally, they would like to thank the participants of the Set Theory Seminar of Tel-Aviv University, and in particular to Menachem Magidor, Carmi Merimovich and 
Sittinon Jirattikansakul for their wonderful comments and suggestions.





 \bibliographystyle{amsplain}
\bibliography{ref}

\providecommand{\bysame}{\leavevmode\hbox to3em{\hrulefill}\thinspace}
\providecommand{\MR}{\relax\ifhmode\unskip\space\fi MR }
\providecommand{\MRhref}[2]{%
  \href{http://www.ams.org/mathscinet-getitem?mr=#1}{#2}
}
\providecommand{\href}[2]{#2}
\begin{thebibliography}{10}

\bibitem{AbrahamShelah1994}
Uri Abraham and Saharon Shelah, \emph{{F}orcing {C}losed {U}nbounded {S}ets},
  {T}he {J}ournal of {S}ymbolic {L}ogic (1983), 643--657.

\bibitem{MR0369081}
J.~E. Baumgartner, A.~H{\c{a}}j{\c{n}}al, and A.~Mate, \emph{Weak {S}aturation
  {P}roperties of {I}deals}, Infinite and finite sets ({C}olloq., {K}eszthely,
  1973; dedicated to {P}. {E}rd{\H o}s on his 60th birthday), {V}ol. {I},
  North-Holland, Amsterdam, 1975, pp.~137--158. Colloq. Math. Soc. J\'anos
  Bolyai, Vol. 10. \MR{0369081 (51 \#5317)}

\bibitem{shalom2017woodin}
Yoav Ben-Shalom, \emph{On the {W}oodin {C}onstruction of {F}ailure of {G}{C}{H}
  at a {M}easurable {C}ardinal}, 2017.

\bibitem{TomTreePrikry}
Tom Benhamou, \emph{Prikry {F}orcing and {T}ree {P}rikry {F}orcing of {V}arious
  {F}ilters}, Arch. Math. Logic \textbf{58} (2019), 787–--817.

\bibitem{NegGalSing}
Tom Benhamou, Shimon Garti, and Alejandro Poveda, \emph{Negating the {G}alvin
  {P}roperty}, preprint (2021), arXiv:2112.13373.

\bibitem{BenGarShe}
Tom Benhamou, Shimon Garti, and Saharon Shelah, \emph{Kurepa {T}rees and {T}he
  {F}ailure of the {G}alvin {P}roperty}, Proceedings of the {A}merican
  {M}athematical {S}ociety (2021), arXiv:2111.11823.

\bibitem{partOne}
Tom Benhamou and Moti Gitik, \emph{{I}ntermediate {M}odels of {M}agidor-{R}adin
  {F}oring-{P}art {I}}, Israel Journal of Mathematics (2021), arXiv:2009.12775.

\bibitem{TomMoti}
\bysame, \emph{{S}ets in {P}rikry and {M}agidor {G}eneric {E}xtesions}, Annals
  of Pure and Applied Logic \textbf{172} (2021), no.~4, 102926.

\bibitem{Parttwo}
\bysame, \emph{{I}ntermediate {M}odels of {M}agidor-{R}adin {F}orcing-{P}art
  {I}{I}}, Annals of Pure and Applied Logic \textbf{173} (2022), 103107.

\bibitem{YairTomMoti}
Tom Benhamou, Moti Gitik, and Yair Hayut, \emph{{T}he {V}ariety of
  {P}rojections of a {T}ree-{P}rikry {F}orcing}, preprint (2021),
  arXiv:2109.09069.

\bibitem{CummingWoodin}
James Cummings, \emph{A model in which gch holds at successors but fails at
  limits}, Transactions of the American Mathematical Society \textbf{329}
  (1992), no.~1, 1--39.

\bibitem{CummingsHand}
James Cummings, \emph{Iterated {F}orcing and {E}lementary {E}mbeddings},
  pp.~775--883, Springer Netherlands, Dordrecht, 2010.

\bibitem{MR3604115}
Shimon Garti, \emph{Weak {D}iamond and {G}alvin's {P}roperty}, Period. Math.
  Hungar. \textbf{74} (2017), no.~1, 128--136. \MR{3604115}

\bibitem{MR3787522}
\bysame, \emph{Tiltan}, C. R. Math. Acad. Sci. Paris \textbf{356} (2018),
  no.~4, 351--359. \MR{3787522}

\bibitem{Gitik1989TheNO}
Moti Gitik, \emph{The {N}egation of the {S}ingular {C}ardinal {H}ypothesis from
  $o(\kappa) = \kappa^{++}$}, Annals of Pure and Applied Logic \textbf{43}
  (1989), 209--234.

\bibitem{Gitikstrength}
\bysame, \emph{The {S}trength of the {F}ailure of the {S}ingular {C}ardinal
  {H}ypothesis}, Annals of Pure and Applied Logic \textbf{51} (1991), no.~3,
  215--240.

\bibitem{Gitik2010}
\bysame, \emph{{P}rikry-{T}ype {F}orcings}, pp.~1351--1447, Springer
  Netherlands, Dordrecht, 2010.

\bibitem{GitDensity}
\bysame, \emph{{O}n {D}ensity of {O}ld {S}ets in {P}rikry {T}ype {E}xtensions},
  Proceedings of the American Mathematical Society \textbf{145} (2017), no.~2,
  881--887.

\bibitem{GitikOnCompactCardinals}
\bysame, \emph{On $\kappa$-{C}ompact {C}ardinals}, Israel Journal of
  Mathematics \textbf{237} (2020), 457--483.

\bibitem{PrikryCaseGitikKanKoe}
Moti Gitik, Vladimir Kanovei, and Peter Koepke, \emph{{I}ntermediate {M}odels
  of {P}rikry {G}eneric {E}xtensions},
  http://www.math.tau.ac.il/~gitik/spr-kn.pdf (2010).

\bibitem{Git-Mag}
Moti Gitik and Menachem Magidor, \emph{{E}xtender {B}ased {F}orcings}, The
  Journal of Symbolic Logic \textbf{59} (1994), no.~2, 445--460.

\bibitem{Hamkins2000-HAMTLP}
Joel~David Hamkins, \emph{The {L}ottery {P}reparation}, Annals of Pure and
  Applied Logic \textbf{101} (2000), no.~2-3, 103--146.

\bibitem{Jech2003}
Thomas Jech, \emph{Set {T}heory}, Springer Monographs in Mathematics,
  Springer-Verlag, Berlin, 2003, The third millennium edition, revised and
  expanded. \MR{1940513}

\bibitem{kanamori1994}
Akihiro Kanamori, \emph{The {H}igher {I}nfinite}, Springer, 1994.

\bibitem{MinimalPrikry}
Peter Koepke, Karen Rasch, and Philipp Schlicht, \emph{{M}inimal
  {P}rikry-{T}ype {F}orcing for {S}ingularizing a {M}easurable {C}ardinal},
  Journal of Symbolic Logic \textbf{78} (2013), 85–--100.

\bibitem{MagAnnals}
Menachem Magidor, \emph{On the singular cardinals problem ii}, Annals of
  Mathematics \textbf{106} (1977), no.~3, 517--547.

\bibitem{Maharam}
Dorothy Maharam, \emph{An {A}lgebraic {C}haracterization of {M}easure
  {A}lgebras}, Annals of Mathematics \textbf{48} (1947), 154–167.

\bibitem{Mathias}
A.~R.~D. Mathias, \emph{{O}n {S}equences {G}eneric in the {S}ense of {P}rikry},
  Journal of Australian Mathematical Society \textbf{15} (1973), 409--414.

\bibitem{Mer}
Carmi Merimovich, \emph{Prikry on {E}xtenders, {R}evisited}, Israel Journal of
  Mathematics \textbf{160} (2003), 253--280.

\bibitem{Mer1}
\bysame, \emph{Supercompact {E}xtender {B}ased {P}rikry {F}orcing}, Archive for
  Mathematical Logic \textbf{50} (2011), 591--602.

\bibitem{Mer2}
\bysame, \emph{Mathias {L}ike {C}riterion for the {E}xtender {B}ased {P}rikry
  {F}orcing}, Annals Pure and Applied Logic \textbf{172} (2021), 102994.

\bibitem{mitchell_1984}
William~J. Mitchell, \emph{The {C}ore {M}odel for {S}equences of {M}easures.
  {I}}, Mathematical Proceedings of the Cambridge Philosophical Society
  \textbf{95} (1984), no.~2, 229–260.

\bibitem{Mit}
\bysame, \emph{On the singular cardinal hypothesis}, Transactions of the
  American Mathematical Society \textbf{329} (1992), no.~2, 507--530.

\bibitem{Sacks}
Gerald~E. Sacks, \emph{Forcing with perfect closed sets, axiomatic set theory},
  Proceedings of symposia in pure mathematics \textbf{13} (1971), no.~1,
  331–355.

\end{thebibliography}
\end{document}